\documentclass[a4paper, oneside]{amsart}
\usepackage{amssymb}
\usepackage{amscd}
\usepackage{amsmath}
\usepackage[all]{xy}
\usepackage{mathrsfs}
 \usepackage{color}

\usepackage{geometry}

\theoremstyle{plain}
\newtheorem{theorem}{Theorem}[section]
\newtheorem{proposition}[theorem]{Proposition}
\newtheorem{lemma}[theorem]{Lemma}
\newtheorem{corollary}[theorem]{Corollary}

\theoremstyle{remark}
\newtheorem{remark}[theorem]{Remark}

\theoremstyle{definition}
\newtheorem{definition}[theorem]{Definition}
\newtheorem{example}[theorem]{Example}

  1

\DeclareMathOperator{\Ext}{Ext}
\DeclareMathOperator{\Gal}{Gal}

\DeclareMathOperator{\Hom}{Hom}

\DeclareMathOperator{\im}{im}

\DeclareMathOperator{\Fitt}{Fitt}

\newcommand{\bC}{\mathbb{C}}
\newcommand{\bF}{\mathbb{F}}

\newcommand{\bQ}{\mathbb{Q}}
\newcommand{\bL}{\mathbb{L}}

\newcommand{\bT}{\mathbb{T}}

\newcommand{\bZ}{\mathbb{Z}}

\newcommand{\cD}{\mathcal{D}}
\newcommand{\cE}{\mathcal{E}}
\newcommand{\cF}{\mathcal{F}}

\newcommand{\cH}{\mathcal{H}}
\newcommand{\cI}{\mathcal{I}}

\newcommand{\cK}{\mathcal{K}}
\newcommand{\cL}{\mathcal{L}}

\newcommand{\cN}{\mathcal{N}}
\newcommand{\cO}{\mathcal{O}}
\newcommand{\cP}{\mathcal{P}}

\newcommand{\fa}{\mathfrak{a}}

\newcommand{\fd}{\mathfrak{d}}

\newcommand{\fl}{\mathfrak{l}}
\newcommand{\fq}{\mathfrak{q}}
\newcommand{\fp}{\mathfrak{p}}
\newcommand{\fr}{\mathfrak{r}}
\newcommand{\fs}{\mathfrak{s}}
\newcommand{\fn}{\mathfrak{n}}
\newcommand{\fm}{\mathfrak{m}}

\newcommand{\fS}{\mathfrak{S}}

\newcommand{\Z}{\mathbb{Z}}

\newcommand{\Reg}{\mathrm{Reg}}

\newcommand{\Ann}{\mathrm{Ann}}
\newcommand{\bz}{\mathbb{Z}}

\newcommand{\sgn}{\mathrm{sgn}}

\newcommand{\hooklongrightarrow}{\lhook\joinrel\longrightarrow}

\begin{document}

\title[]{A higher rank Euler system for the multiplicative group over a totally real field}

\author{Ryotaro Sakamoto}

\begin{abstract} 
In this paper, we construct a higher rank Euler system for the multiplicative group over a totally real field by using 
the Iwasawa main conjecture proved by Wiles. 
A key ingredient of the construction is to generalize the notion of the characteristic ideal. 
Under certain technical assumptions, we prove that all higher Fitting ideals of a certain $p$-ramified Iwasawa module are described by analytic invariants canonically associated with Stickelberger elements. 
\end{abstract}

\address{Graduate School of Mathematical Sciences, The University of Tokyo,
3-8-1 Komaba, Meguro-Ku, Tokyo, 153-8914, Japan}
\email{sakamoto@ms.u-tokyo.ac.jp}

\maketitle

\setcounter{tocdepth}{2}

\tableofcontents

\section{Introduction}

Euler systems were introduced by Kolyvagin in order to study the structure of Selmer groups of elliptic curves. 
The theory of Euler systems has been playing a crucial role in studying the arithmetic of global Galois representations. 
However, there are very few known examples of Euler systems (see \cite{Kato04, LLZ14, LSZ17, LZ17, R}). 

In the paper \cite{PR}, Perrin-Riou introduced the notion of higher rank Euler systems.  
However, at the present moment, we know an example of a higher rank Euler system only in the $\mathrm{GL}_{2} \times \mathrm{GL}_{2}$-case (see \cite{BL19,  BLLV, BLV}). 
In this paper, we construct a higher rank Euler system for the multiplicative group over a totally real field by using the classical Iwasawa main conjecture for totally real fields proved by Wiles in \cite{Wiles90}. 

\begin{remark}
The definition of higher rank Euler systems in this paper is not the same as the one introduced by Perrin-Riou in \cite{PR}. 
Perrin-Riou used the exterior power in order to define higher rank Euler systems, while we use the exterior bi-dual in this paper. 
Our definition of higher rank Euler systems is a generalization of Perrin-Riou's one and 
these definitions are essentially the same in the rank $1$ case (see \cite{bss, sbA, sakamoto-bessatsu}). 
\end{remark}

To explain the  result we obtain in this article more precisely, 
we introduce some notations. 
Let $k$ be a totally real field with degree $r := [k \colon \bQ]$ and 
\[
\chi \colon G_{k} \longrightarrow \overline{\bQ}^{\times}
\] 
a non-trivial finite order even character. 
Here, for a field $k'$, we write $G_{k'}$ for the absolute Galois group of $k'$.   
Put $L := \overline{k}^{\ker(\chi)}$. 
Let $p$ be an odd prime such that both the order of $\chi$ and the class number of $k$ are coprime to $p$. 
%,and \item $p$ is unramified in $L := \overline{k}^{\ker(\chi)}$. 
Fix a finite abelian $p$-extension $K/k$ satisfying $S_{{\rm ram}}(K/k) \cap (S_{p}(k) \cup S_{\rm ram}(L/k)) = \emptyset$.
Here $S_{p}(k)$ denotes the set of all primes of $k$ above $p$ and, for an algebraic extension $K'/k$, we write $S_{\rm ram}(K'/k)$ for the set of all primes of $k$ at which $K'/k$ is ramified. 
Fix an embedding $\overline{\bQ} \hooklongrightarrow \overline{\bQ}_{p}$ and put 
\begin{align*}
\cO := \bZ_{p}[\im(\chi)] \ \textrm{ and } \ T := \cO(1) \otimes \chi^{-1}.
\end{align*} 
We write $M_{KL,\infty}$ for the maximal $p$-ramified pro-$p$ abelian extension of $KL(\mu_{p^{\infty}})$. 
We note that $M_{KL,\infty}/k$ is a Galois extension. 
Since $M_{KL,\infty}/KL(\mu_{p^{\infty}})$ is an abelian extension, the Galois group $\Gal(KL(\mu_{p^{\infty}})/k)$ acts on the Galois group  $\Gal(M_{KL,\infty}/KL(\mu_{p^{\infty}}))$. Furthermore, by using this Galois action, we see that the module 
 \[
 X_{K} := \cO \otimes_{\bZ_{p}} \Gal(M_{KL,\infty}/KL(\mu_{p^{\infty}}))
 \]
 has an $\cO[[\Gal(KL(\mu_{p^{\infty}})/k)]]$-module structure. 
 Let $k_{\infty}/k$ denote the cyclotomic $\bZ_{p}$-extension. Then we have the canonical isomorphism 
 \[
 \Gal(KL(\mu_{p^{\infty}})/k) \stackrel{\sim}{\longrightarrow} \Gal(k_{\infty}K/k) \times \Gal(L(\mu_{p})/k)
 \]
 since $p \nmid [L \colon k]$. 
Set 
\[
e_{\chi} := \frac{1}{[L(\mu_{p}) \colon k]} \sum_{\sigma \in \Gal(L(\mu_{p})/k)}\chi(\sigma)\sigma^{-1} 
= \frac{1}{[KL(\mu_{p^{\infty}}) \colon k_{\infty}K]} \sum_{\sigma \in \Gal(KL(\mu_{p^{\infty}})/k_{\infty}K)}\chi(\sigma)\sigma^{-1}.  
\]
By the above isomorphism, we see that the Iwasawa module 
\begin{align*}
X^{\chi}_{K} := e_{\chi}X_{K}
\end{align*}
is a $\Lambda_{K} := \cO[[\Gal(k_{\infty}K/k)]]$-module.

Suppose that 
\begin{itemize}
\item $H^{0}(G_{k_{\fp}},T/\fm T) = H^{2}(G_{k_{\fp}},T/\fm T) = 0$
for each prime $\fp \in S_{p}(k)$, 
\end{itemize}
where $\fm$ denotes the maximal ideal of $\cO$. 
Then one can show that, for any finite abelian extension $K'/k$ and prime $\fp \in S_{p}(k)$, 
\[
H^{1}(G_{k_{\fp}},T \otimes_{\bZ_{p}} \bZ_{p}[\Gal(K'/k)]^{\iota})
\] 
is a free $\cO[\Gal(K'/k)]$-module of rank $[k_{\fp} \colon \bQ_{p}]$  (see Corollary~\ref{cor:H2vanish}). 
Here $\bZ_{p}[\Gal(K'/k)]^{\iota}$ is a free $\bZ_{p}[\Gal(K'/k)]$-module of rank $1$ and the action of $\bZ_{p}[\Gal(K'/k)]$ is defined by 
\[
\iota \colon \bZ_{p}[\Gal(K'/k)] \longrightarrow \bZ_{p}[\Gal(K'/k)]; \sum_{g \in \Gal(K'/k)}a_{g}g \mapsto \sum_{g \in \Gal(K'/k)}a_{g}g^{-1}. 
\]
The action of $G_{k}$ on  $\bZ_{p}[\Gal(K'/k)]^{\iota}$ is induced by its  $\bZ_{p}[\Gal(K'/k)]$-action.  

Fix an isomorphism 
\[
\varprojlim_{K'}\bigoplus_{\fp \mid p}H^{1}(G_{k_{\fp}},T \otimes_{\bZ_{p}} \bZ_{p}[\Gal(K'/k)]^{\iota}) \stackrel{\sim}{\longrightarrow} \varprojlim_{K'}\cO[\Gal(K'/k)]^{r}, 
\]
where $K'$ runs over all the finite abelian extensions of $k$. 

For each integer $s \geq 0$, we denote by ${\rm ES}_{s}(T)$ the module of Euler systems of rank $s$ for $T$ (see Definition~\ref{def:euler system}). 
The above fixed isomorphism induces an injective `rank reduction' homomorphism 
\[
{\rm ES}_{r}(T) \hooklongrightarrow {\rm ES}_{0}(T).  
\]

\begin{remark}
A rank reduction homomorphism ${\rm ES}_{r}(T) \longrightarrow {\rm ES}_{s}(T)$ (for $0 \leq s \leq r$) was considered firstly by Rubin in \cite{rubinstark} and Perrin-Riou in \cite{PR}. 
After that, in the papers \cite{Bu09a, Bu09b, Bu10}, 
B\"uy\"ukboduk developed a machinery for higher rank Euler systems by using the rank reduction homomorphism.  
\end{remark}

An Euler system of rank $0$ is a collection of elements in group rings satisfying certain  norm-relations, and, philosophically, it should be related to $p$-adic $L$-functions. 
We actually show in Lemma~\ref{lem:rel2} that the collection of $p$-adic $L$-functions for the multiplicative group over totally real fields is an Euler system of rank $0$. 
Let $\cL_{p}^{\chi} \in {\rm ES}_{0}(T)$ denote this Euler system.  

Under this setting, the following is the main result of this paper.

\begin{theorem}[Theorem~\ref{thm:main2}]\label{main}
Suppose that 
\begin{itemize}
\item $H^{0}(G_{k_{\fp}},T/\fm T) = H^{2}(G_{k_{\fp}},T/\fm T) = 0$
for each prime $\fp$ of $k$ above $p$, and 
\item the module $X_{k}^{\chi}$ is $p$-torsion-free. 
\end{itemize}
%Let us fix an isomorphism $\varprojlim_{K \in \Omega}H^{1}_{\Sigma}(\bT_{K}) \stackrel{\sim}{\longrightarrow} \Lambda[[\Gal(\cK/k)]]^{r}$. 
Then there is a unique Euler system  of rank $r$ in ${\rm ES}_{r}(T)$ such that its image  under the injection ${\rm ES}_{r}(T) \hooklongrightarrow {\rm ES}_{0}(T)$  is $\cL_{p}^{\chi}$. 
\end{theorem}

\begin{remark}
When $k = \bQ$, Theorem~\ref{main} follows from the well-known relation between 
cyclotomic units and Kubota--Leopoldt $p$-adic $L$-functions.  
In general, we expect that the higher rank Euler system constructed in Theorem~\ref{main} gives us  the conjectural Rubin--Stark units (up to unit multiples). 
Under the Rubin--Stark conjecture, this expectation is equivalent to the certain relation between the conjectural Rubin--Stark units and $p$-adic $L$-functions which has already been considered by B\"uy\"ukboduk in the papers \cite{Bu09a, Bu09b, Bu11a}. 
\end{remark}

A key ingredient of the proof of Theorem~\ref{main} is to generalize the notion of the characteristic ideal. 
In order to define the characteristic ideal of a finitely generated torsion module over a noetherian ring $R$, we need to impose that $R$ is a normal ring since the definition relies on the structure theorem for finitely generated modules over a normal ring. 
%is crucial to define the characteristic ideal. 
In this paper, by using an exterior bi-dual instead of the structure theorem, we give a new definition of the characteristic ideal. 
%Thanks to this new definition, we can consider the characteristic ideal of a finitely generated  module over any noetherian ring $R$, 
This new definition is applicable to a finitely generated  module over any noetherian ring, 
and we can compute the image of the injection ${\rm ES}_{r}(T) \hooklongrightarrow {\rm ES}_{0}(T)$ by using characteristic ideals of certain Iwasawa modules. 
Furthermore, by using the classical Iwasawa main conjecture for totally real fields proved by Wiles in \cite{Wiles90}, we show that the image contains the Euler system $\cL_{p}^{\chi}$.

As an application of Theorem~\ref{main}, by using the theory of higher rank Euler, Kolyvagin and Stark systems (which have been developed by 
Rubin in \cite{R}, Mazur--Rubin in \cite{MRkoly,MRselmer}, B\"uy\"ukboduk in \cite{Bu09a, Bu09b, Bu10, Bu11b}, the author in \cite{sakamoto}, Burns--Sano in \cite{sbA}, and Burns--Sano and the author in \cite{bss, bss2}), we prove that all higher Fitting ideals of $X_{K}^{\chi}$ are described by analytic invariants canonically associated with Stickelberger elements.  
To be more precise, one can define an increasing sequence $\{\Theta_{K,\infty}^{i}\}_{i\geq0}$ of ideals of the Iwasawa algebra $\Lambda_{K} = \cO[[\Gal(k_{\infty}K/k)]]$. 
The ideals $\Theta_{K,\infty}^{i}$ are generated by 
some Kolyvagin derivative classes associated with Stickelberger elements (see Definition~\ref{def:theta}). 

\begin{corollary}[{Theorem~\ref{thm:main4}}]\label{thm1}
Suppose that 
\begin{itemize}
\item the same assumptions as in Theorem~\ref{main} hold, and 
\item $H^{2}(k_{\fq}, T/\fm T)$ vanishes for each prime $\fq \in S_{\rm ram}(K/k)$. 
\end{itemize}
Then, for any integer $i \geq 0$, 
we have  
\[
\Theta_{K,\infty}^{i} = {\rm Fitt}^{i}_{\Lambda_{K}}(X^{\chi}_{K}).
\] 
\end{corollary}

\begin{remark}\label{rem:kurihara}
When $K=k$, Corollary~\ref{thm1} has already been proved by Kurihara in \cite{Kur03, Kur12}. 
Hence  Corollary~\ref{thm1} can be viewed as an equivariant generalization of  Kurihara's result. 
In the proof of \cite[Theorem~2.1]{Kur12}, one of the most important steps is the construction of a certain system of elements $\{\kappa_{\fn,\fl}\}$, and Kurihara proved that the system $\{\kappa_{\fn,\fl}\}$ has some important and beautiful properties. 

To prove Corollary~\ref{thm1}, we will also construct a similar system of elements $\{\kappa_{n,\fn,\fl}^{\sigma}\}$ in \S~\ref{sec:fitt}, and prove that the system $\{\kappa_{n,\fn,\fl}^{\sigma}\}$ has the same kind of properties as 
$\{\kappa_{\fn,\fl}\}$ (see Proposition~\ref{prop:coh rel}). 
However, the construction of $\{\kappa_{\fn,\fl}\}$ and that of $\{\kappa_{n,\fn,\fl}^{\sigma}\}$ are completely different. In this paper, we use Stark systems in order to construct the system $\{\kappa_{n,\fn,\fl}^{\sigma}\}$. 
It is assumed in \cite{Kur12} that $\fn\fl$ is well-ordered, while 
we do not need to impose any assumptions in the construction of the system $\{\kappa_{n,\fn,\fl}^{\sigma}\}$ and the proof of its properties. 

We remark that the proof of the equality $\Theta_{k,\infty}^{i} = {\rm Fitt}^{i}_{\Lambda}(X^{\chi}_{k})$ 
is also different from that in \cite{Kur12}. 
In this paper, this equality is derived from the fact proved in \cite{sbA, sakamoto} that all higher Fitting ideals of a Selmer group are controlled by Stark systems. 
This allows us to control all higher Fitting ideals of the Iwasawa module $X_{K}^{\chi}$ even in the equivariant case. 
\end{remark}

%\clearpage

%\subsection*{Notation}

\subsection{Notation}

Let $p$ be an odd prime. For a field $k$, we fix a separable closure $\overline{k}$ of $k$ and 
denote by $G_{k} := \Gal(\overline{k}/k)$ the absolute Galois group of $k$. 

For a commutative ring $R$ and an $R$-module $M$, 
we write  
\[
M^{*} := \Hom_{R}(M,R)
\]
for the $R$-dual of $M$. 
For any integer $r \geq 0$, we define an $r$-th exterior bi-dual ${\bigcap}^{r}_{R}M$ of $M$ to be 
\[
{\bigcap}^{r}_{R}M := \left({\bigwedge}^{r}_{R}(M^{*})\right)^{*}. 
\] 

\begin{remark}\label{rem:exterior}
If $R$ is a reduced noetherian ring, 
then we have the canonical isomorphism 
\[
\left\{ x \in  Q \otimes_{R} {\bigwedge}^{r}_{R}M \ \middle| \ \Phi(a) \in R \text{ for all } \Phi \in {\bigwedge}^{r}_{R}M^{*}\right\} 
\stackrel{\sim}{\longrightarrow} {\bigcap}^{r}_{R}M.  
\]
Here $Q$ denotes  the total ring of fractions of $R$. 
This fact tells us that the notion of exterior bi-dual ${\bigcap}^{r}_{R}M$ is a generalization of the notion of Rubin-lattice. 
\end{remark}

For a profinite group $G$ and a topological $G$-module $M$, let $C^{\bullet}(G,M)$ denote the complex of inhomogeneous continuous cochains of $G$ with values in $M$. 
We also denote the object in the derived category corresponding to the complex $C^{\bullet}(G,M)$ by ${\bf R}\Gamma(G,M)$. 
For each integer $i \geq 0$, we write $H^{i}(G,M)$ for its $i$-th cohomology group. 

For any number field $k$, we denote by $S_{p}(k)$ and $S_{\infty}(k)$ the set of places of $k$ above $p$ and $\infty$, respectively. 
For a finite set $S$ of places of $k$ containing $S_{\infty}(k)$, 
we denote by $k_{S}$ the maximal extension of $k$ contained in $\overline{k}$ which is unramified outside $S$.  Set 
\[
G_{k,S} := \Gal(k_{S}/k). 
\]
For a prime $\fq$ of $k$, we denote by $k_{\fq}$ the completion of $k$ at $\fq$. 
%and by $\cO_{k_{\fq}}$ the ring of integers of $k_{\fq}$.
For an algebraic extension $K/k$, we denote by $S_{\rm ram}(K/k)$ the set of primes at which $K/k$ is ramified.

For an abelian extension $K/k$, we write 
\[
\iota \colon \bZ_{p}[[\Gal(K/k)]] \longrightarrow \Z_{p}[[\Gal(K/k)]]
\] 
for the involution induced by $\sigma \mapsto \sigma^{-1}$ for $\sigma \in \Gal(K/k)$. 
For a $\Z_{p}[[\Gal(K/k)]]$-module $M$, we set
\[
M^{\iota} := M \otimes_{\bZ_{p}[[\Gal(K/k)]], \iota} \bZ_{p}[[\Gal(K/k)]].
\]
In this paper, we regard $\bZ_{p}[[\Gal(K/k)]]^{\iota}$ as a Galois representation of $G_{k}$. 
The action of $G_{k}$ on  $\bZ_{p}[[\Gal(K/k)]]^{\iota}$ is defined by its  $\bZ_{p}[[\Gal(K/k)]]$-action.  

%\end{notation}

%For an integral domain $R$ and a finitely generated module $M$, set 
%\[
%{\rm rank}_{R}(M) := \dim_{{\rm Frac}(R)}(M \otimes_{R} {\rm Frac}(R)). 
%\]

\subsection{Acknowledgments}

The author would like to express his gratitude to his supervisor Takeshi Tsuji for many helpful discussions. 
The author would also like to thank B\"uy\"ukboduk K\^azim, Masato Kurihara and Hiroki Matsui for helpful advice and comments.  
This work was supported by the Program for Leading Graduate Schools, MEXT, Japan and JSPS KAKENHI Grant Number 17J02456.

\section{Higher rank Euler systems}\label{sec:euler}
In this section, we will recall the definition of a higher rank Euler system and define a characteristic ideal for any finitely generated module over a noetherian ring.

First, let us introduce some notations and hypotheses. 
Let $\cO$ be a complete discrete valuation ring of mixed characteristic $(0, p)$ with maximal ideal $\fm$. 
Let $k$ be a number field. 
In this section, we fix a finite set $S$ of places of a number field $k$ with $S_{\infty}(k) \cup S_{p}(k) \subseteq S$ and let $T$ be a free $\cO$-module of finite rank on which $G_{k,S}$ acts continuously. 
We also fix a $\bZ_{p}^{s}$-extension $k_{\infty} \subseteq \overline{k}$ of $k$ such that $s \geq 1$ and no prime of $k$ splits completely in  $k_{\infty}$. We then put 
\[
\Lambda := \cO[[\Gal(k_{\infty}/k)]] \ \text{ and } \ 
\bT := T \otimes_{\cO} \cO[[\Gal(k_{\infty}/k)]]^{\iota}. 
\]
We take a pro-$p$ abelian extension $\cK \subseteq \overline{k}$ of $k$ such that 
\begin{itemize}
\item $S_{\rm ram}(\cK/k) \cap S = \emptyset$, and 
\item $\cK$ contains the maximal $p$-subextension of the ray class field modulo $\fq$ for  all but finitely many primes $\fq$ of $k$. 
\end{itemize}
We denote the set of finite extensions of $k$ in  $\cK$ by $\Omega$; 
\[
\Omega := \{K \mid k \subseteq K \subseteq \cK, [K \colon k] < \infty\}. 
\]
We also fix a non-empty subset $\Sigma \subseteq S_{p}(k)$. 
For a field $K \in \Omega$, we put 
\begin{itemize}
\item $S_{K} := S \cup S_{\rm ram}(K/k)$, 
\item $\Lambda_{K} := \cO[[\Gal(k_{\infty}K/k)]]$, 
\item $\bT_{K} := T \otimes_{\bZ_{p}} \bZ_{p}[[\Gal(k_{\infty}K/k)]]^{\iota}$, and 
\item $H^{1}_{\Sigma}(\bT_{K}) := \bigoplus_{\fp \in \Sigma}H^{1}(G_{k_{\fp}},\bT_{K})$. 
\end{itemize}
For each field $K \in \Omega$, we fix a $\Lambda_{K}$-submodule $H^{1}_{f,\Sigma}(\bT_{K})$ of $H^{1}_{\Sigma}(\bT_{K})$. 
We take an integer $r > 0$. In this section, we assume the following hypotheses:
\begin{itemize}
\item[(H.0)] The class number of $k$ is coprime to $p$. 
\item[(H.1)] The module $H^{0}(G_{k_{\fp}},T/\fm T)$ vanishes for any prime $\fp \in \Sigma$.  
\item[(H.2)] For each field $K \in \Omega$, the $\Lambda_{K}$-module $H^{1}_{f,\Sigma}(\bT_{K})$ is free of rank $r$ and the canonical map $H^{1}_{\Sigma}(\bT_{K'}) \longrightarrow H^{1}_{\Sigma}(\bT_{K})$ induces 
an isomorphism  
\[
\Lambda_{K} \otimes_{\Lambda_{K'}} H^{1}_{f,\Sigma}(\bT_{K'}) \stackrel{\sim}{\longrightarrow} 
H^{1}_{f,\Sigma}(\bT_{K})
\] 
for any field $K' \in \Omega$ with $K \subseteq K'$. 
\item[(H.3)] 
The direct sum of localization maps $H^{1}(G_{k,S_{K}},\bT_{K}) \longrightarrow H^{1}_{\Sigma}(\bT_{K})$ is injective. 
\end{itemize}

\begin{definition}
For each field $K \in \Omega$, we set 
\[
H_{f}^{1}(G_{k,S_{K}},\bT_{K}) := \ker\left(H^{1}(G_{k,S_{K}},\bT_{K}) \longrightarrow H^{1}_{\Sigma}(\bT_{K})/H^{1}_{f,\Sigma}(\bT_{K})\right). 
\]
\end{definition}

\begin{remark}\label{rem:Lambda_{K}}\ 
\begin{itemize}
\item[(i)] Hypothesis~(H.0) implies that, for any field $K \in \Omega$, we have  the canonical isomorphism 
\[
\Gal(k_{\infty}K/k) \stackrel{\sim}{\longrightarrow} \Gal(k_{\infty}/k) \times \Gal(K/k). 
\]
Hence we obtain an isomorphism $\Lambda_{K} \stackrel{\sim}{\longrightarrow} \Lambda[\Gal(K/k)]$. 
In this paper, by using this isomorphism, we identify $\Lambda_{K}$ with $\Lambda[\Gal(K/k)]$. 

\item[(ii)] For any $\Lambda_{K}$-module $M$, we have the canonical isomorphism 
\[
\Hom_{\Lambda}(M, \Lambda) \stackrel{\sim}{\longrightarrow} \Hom_{\Lambda_{K}}(M, \Lambda_{K}); \varphi \mapsto 
\left(m \mapsto \sum_{g \in \Gal(K/k)}\varphi(gm)g^{-1}\right). 
\]
In particular, ${\rm Ext}^{i}_{\Lambda}(M, \Lambda) \cong {\rm Ext}^{i}_{\Lambda_{K}}(M, \Lambda_{K})$ for each integer $i \geq 0$. 
Since $\Lambda$ is a regular local ring with residue field $\cO/\fm$, we have  
\[
{\rm Ext}^{i}_{\Lambda_{K}}(\cO/\fm, \Lambda_{K}) \cong {\rm Ext}^{i}_{\Lambda}(\cO/\fm, \Lambda) \cong 
\begin{cases}
0 & \text{ if } i < \dim(\Lambda) = \dim(\Lambda_{K}), 
\\
\cO/\fm & \text{ if } i = \dim(\Lambda) = \dim(\Lambda_{K}), 
\end{cases}
\]
and we conclude that $\Lambda_{K}$ is Gorenstein since the residue field of $\Lambda_{K}$ is also $\cO/\fm$ . 
\end{itemize}
\end{remark}

\begin{definition}
For a field $K \in \Omega$, we define a $\Lambda_{K}$-module $R_{f, \Sigma}(\bT_{K})$ to be the cokernel of the canonical homomorphism $H^{1}_{f}(G_{k, S_{K}},\bT_{K}) \longrightarrow H^{1}_{f,\Sigma}(\bT_{K})$.   
%\[
%R_{f,\Sigma}(\bT_{K}) := {\rm coker}\left(H^{1}_{f}(G_{k,S},\bT_{K}) \longrightarrow H^{1}_{f,\Sigma}(\bT_{K})\right). 
%\]
\end{definition}

We note that, by the hypothesis~(H.3), we have  an exact sequence of $\Lambda_{K}$-modules 
\begin{align}\label{exactseq}
0 \longrightarrow H^{1}_{f}(G_{k,S_{K}},\bT_{K}) \longrightarrow H^{1}_{f,\Sigma}(\bT_{K}) \longrightarrow R_{f,\Sigma}(\bT_{K}) \longrightarrow 0. 
\end{align}

\begin{definition}\ 
\begin{itemize}
\item[(i)] Throughout this paper, we fix an arithmetic Frobenius element ${\rm Frob}_{\fq} \in \Gal(\overline{k}/k)$ for each prime $\fq$ of $k$. 
\item[(ii)] For each prime $\fq \not\in S$ of $k$, we define the Frobenius characteristic polynomial at $\fq$ by  
\[
P_{\fq}(x) := \det(1- x \cdot {\rm Frob}_{\fq} \mid T) \in \cO[x].
\]
\item[(iii)] 
For a field $K \in \Omega$, take a $\Lambda_{K}$-module $M_{K}$. 
Suppose that $\{M_{K}\}_{K \in \Omega}$ is an inverse system of $\Lambda[[\Gal(\cK/k)]]$-modules with transition maps  $\psi_{K',K} \colon M_{K'} \longrightarrow M_{K}$ for $K', K \in \Omega$ with $K \subseteq K'$. 
We then define a module ${\rm ES}(\{M_{K}\}_{K\in\Omega})$ by  
\[
{\rm ES}(\{M_{K}\}_{K\in\Omega}) := \left\{ (m_{K})_{K \in \Omega} \in \prod_{K \in \Omega}M_{K} \ \middle| \ 
\begin{array}{l}
\psi_{K',K}(m_{K'}) = \left(\prod_{\fq \in S_{K'} \setminus S_{K}} P_{\fq}({\rm Frob}_{\fq}^{-1})\right) \cdot m_{K} 
\\
\textrm{ for any fields $K', K \in \Omega$ with $K \subseteq K'$ } 
\end{array}
\right\}. 
\]
\item[(iv)] We set $\cE(T) := {\rm ES}(\{\Lambda_{K}\}_{K \in \Omega})$. 
\end{itemize}
\end{definition}

\begin{lemma}\label{lem:inverse-system}
We have  an inverse system $\{{\bigcap}^{r}_{\Lambda_{K}}H^{1}(G_{k,S_{K}},\bT_{K})\}_{K\in \Omega}$ of $\cO[[k_{\infty}\cK/k]]$-modules, where the transition maps are induced by the canonical maps $\bT_{K'} \longrightarrow \bT_{K}$ for  fields $K', K \in \Omega$ with $K \subseteq K'$. 
\end{lemma}
\begin{proof}
Let $K \in \Omega$ be a field. Take a field $K' \in \Omega$ containing $K$. 
Since no prime of $k$ splits completely in $k_{\infty}$ and $S_{K} \subseteq S_{K'}$, we have 
\[
H^{1}(G_{k,S_{K'}},\bT_{K}) = H^{1}(G_{k,S_{K}},\bT_{K})
\]
(see \cite[Corollary~B.3.6]{R}). 
By hypothesis~(H.1), the complex ${\bf R}\Gamma(G_{k,S_{K'}}, \bT_{K'})$ satisfies 
the assumption in Corollary~\ref{cor:amp}. Hence ${\bf R}\Gamma(G_{k,S_{K'}}, \bT_{K'})$ is a perfect complex of $\Lambda_{K'}$-modules having perfect amplitude contained in $[1,2]$ by Corollary~\ref{cor:amp}.  
Since $\Lambda_{K}$ and $\Lambda_{K'}$ are Gorenstein by Remark~\ref{rem:Lambda_{K}}~(ii), 
Lemma~\ref{lem:reduction} shows that we have the canonical homomorphism 
\[
{\bigcap}^{r}_{\Lambda_{K'}}H^{1}(G_{k,S_{K'}},\bT_{K'}) \longrightarrow 
{\bigcap}^{r}_{\Lambda_{K}}H^{1}({\bf R}\Gamma(G_{k,S_{K'}}, \bT_{K'}) \otimes^{\bL}_{\Lambda_{K'}} \Lambda_{K}). 
\]
By Theorem~\ref{thm:pot}~(i), we have  an isomorphism ${\bf R}\Gamma(G_{k,S_{K'}}, \bT_{K'}) \otimes^{\bL}_{\Lambda_{K'}} \Lambda_{K} \stackrel{\sim}{\longrightarrow} {\bf R}\Gamma(G_{k,S_{K'}}, \bT_{K})$, and hence 
\[
H^{1}({\bf R}\Gamma(G_{k,S_{K'}}, \bT_{K'}) \otimes^{\bL}_{\Lambda_{K'}} \Lambda_{K}) \stackrel{\sim}{\longrightarrow} H^{1}(G_{k,S_{K'}},\bT_{K}) = H^{1}(G_{k,S_{K}},\bT_{K}). 
\]
\end{proof}

Let us recall the definition of higher rank Euler systems for $T$. 

\begin{definition}\label{def:euler system}
For an integer $r \geq 0$, 
we define the module ${\rm ES}_{r}(T)$ of Euler systems of rank $r$ (for $T$) by  
\[
{\rm ES}_{r}(T) := {\rm ES}\left(\left\{{\bigcap}^{r}_{\Lambda_{K}}H^{1}(G_{k,S_{K}},\bT_{K})\right\}_{K \in \Omega}\right). 
\]
Note that ${\rm ES}_{0}(T) = \cE(T)$. 
Since the canonical map 
\[
{\bigcap}^{r}_{\Lambda_{K}}H^{1}_{f}(G_{k,S_{K}},\bT_{K}) \longrightarrow {\bigcap}^{r}_{\Lambda_{K}}H^{1}(G_{k,S_{K}},\bT_{K})
\] 
is injective by Lemma~\ref{lem:sub},  
we can also define 
\[
{\rm ES}_{f,r}(T) := {\rm ES}_{r}(T) \cap \prod_{K \in \Omega}{\bigcap}^{r}_{\Lambda_{K}}H^{1}_{f}(G_{k,S_{K}},\bT_{K}). 
\]
\end{definition}

\begin{remark}
The definition of the polynomial $P_{\fq}(x)$ is the same as \cite[Definition~1.2.2]{MRkoly} and slightly different from the one appeared in \cite[Definition~2.1.1]{R} and \cite[\S6.1]{bss}. 
However, one can switch back between the two choices (see \cite[\S9.6]{R} and 
\cite[Remark~3.2.3]{MRkoly}). 
Hence Lemma~\ref{lem:inv-bidual} shows that the definition of higher rank Euler systems in this paper is essentially the same as the one occurred in \cite[Definition~6.4]{bss}. 
\end{remark}

Next, let us generalize the notion of characteristic ideals. 

\begin{definition}
Let $R$ be a noetherian ring and $M$ a finitely generated $R$-module. 
Take an integer $r>0$ and an exact sequence of $R$-modules 
\[
0 \longrightarrow N \longrightarrow R^{r} \longrightarrow M \longrightarrow 0. 
\]
We define a characteristic ideal of $M$ by 
\[
{\rm char}_{R}(M) := \im\left({\bigcap}^{r}_{R}N \longrightarrow {\bigcap}^{r}_{R}R^{r} = R \right). 
\]
\end{definition}

We study basic properties of characteristic ideals in Appendix~\ref{sec:char}. 
In particular, we show in Appendix~\ref{sec:char} that the ideal ${\rm char}_{R}(M)$ is independent of the choice of the exact sequence $0 \longrightarrow N \longrightarrow R^{r} \longrightarrow M \longrightarrow 0$ (Remark~\ref{rem:indep-char}) and that, when $R$ is a normal ring, the ideal ${\rm char}_{R}(M)$ coincides with the usual one (Remark~\ref{rem:fitt-char}).

Since, by Remark~\ref{rem:Lambda_{K}}~(ii), the ring $\Lambda_{K}$ is Gorenstein for any field $K \in \Omega$, the following proposition follows from Theorem~\ref{thm:inequality}.

\begin{proposition}%[Theorem~\ref{thm:inequality}]
\label{prop:lambda-inequality}
Let $K \in \Omega$ be a field and $M$ a finitely generated $\Lambda_{K}$-module. 
Then for any $\Lambda_{K}$-submodule $N$ of $M$, we have  
\[
{\rm char}_{\Lambda_{K}}(M) \subseteq {\rm char}_{\Lambda_{K}}(N). 
\]
\end{proposition}

\begin{proposition}\label{prop:main}
Assume Hypotheses (H.0) -- (H.3). 
Let us fix an isomorphism 
\[
\varprojlim_{K\in\Omega}H^{1}_{f,\Sigma}(\bT_{K}) \stackrel{\sim}{\longrightarrow} \Lambda[[\Gal(\cK/k)]]^{r}. 
\]
Then it naturally induces an injective homomorphism of $\Lambda[[\Gal(\cK/k)]]$-modules
\[
{\rm ES}_{f,r}(T) \hooklongrightarrow \cE(T). 
\]
%such that 
%this injection depends only on the choice of the isomorphism 
%\[
%\varprojlim_{K\in\Omega}H^{1}_{f,\Sigma}(\bT_{K}) \stackrel{\sim}{\longrightarrow} \Lambda[[\Gal(\cK/k)]]^{r}
%\]
%and that 
Furthermore, we have 
\[
\cE(T) \cap \prod_{K \in \Omega} {\rm char}_{\Lambda_{K}}(R_{f, \Sigma}(\bT_{K})) = 
{\rm im}\left({\rm ES}_{f,r}(T) \hooklongrightarrow  \cE(T) \right). 
\] 
\end{proposition}

\begin{proof}
%The $\Lambda[[\Gal(\cK/k)]]$-module $\varprojlim_{K\in\Omega}H^{1}_{f,\Sigma}(\bT_{K})$ is free of rank $r$ by hypothesis (H.2).  
%Hence one can fix an isomorphism of $\Lambda[[\Gal(\cK/k)]]$-modules: 
%\[
%\varprojlim_{K\in\Omega}H^{1}_{f,\Sigma}(\bT_{K}) \stackrel{\sim}{\longrightarrow} \Lambda[[\Gal(\cK/k)]]^{r}. 
%\]
By the definition of $\cE(T)$, the fixed isomorphism 
\[
\varprojlim_{K\in\Omega}H^{1}_{f,\Sigma}(\bT_{K}) \stackrel{\sim}{\longrightarrow} \Lambda[[\Gal(\cK/k)]]^{r}
\]
induces an isomorphism of $\Lambda[[\Gal(\cK/k)]]$-modules 
\[
{\rm ES}\left(\left\{{\bigcap}^{r}_{\Lambda_{K}}H^{1}_{f,\Sigma}(\bT_{K})\right\}_{K \in \Omega}\right) \stackrel{\sim}{\longrightarrow} \cE(T). 
\]
By the same argument as in the proof of Lemma~\ref{lem:inverse-system}, 
%We remark that hypothesis~(H.1), Theorem~\ref{thm:pot}(i), Corollary~\ref{cor:amp}, and Lemma~\ref{lem:reduction} 
we obtain an inverse system 
\[
\left\{{\bigcap}^{r}_{\Lambda_{K}}H^{1}_{\Sigma}(\bT_{K})\right\}_{K \in \Omega},
\] 
and hence we get a $\Lambda[[\Gal(\cK/k)]]$-module 
\[
{\rm ES}\left(\left\{{\bigcap}^{r}_{\Lambda_{K}}H^{1}_{\Sigma}(\bT_{K})\right\}_{K \in \Omega}\right). 
\]
Thus we obtain a diagram 
\[
\xymatrix{
{\rm ES}_{f,r}(T) \ar@{^{(}->}[d] & \cE(T) \ar@{^{(}->}[d] 
\\
{\rm ES}_{r}(T) \ar[r] & {\rm ES}\left(\left\{{\bigcap}^{r}_{\Lambda_{K}}H^{1}_{\Sigma}(\bT_{K})\right\}_{K \in \Omega}\right). 
}
\] 
Furthermore, for each field $K \in \Omega$, there is also a commutative diagram 
\[
\xymatrix{
{\bigcap}^{r}_{\Lambda_{K}}H^{1}_{f}(G_{k,S_{K}},\bT_{K}) %\ar@{^{(}->}[d] 
\ar[d] \ar[r] & \Lambda_{K} \ar[d] 
\\
{\bigcap}^{r}_{\Lambda_{K}}H^{1}(G_{k,S_{K}},\bT_{K}) \ar[r] & {\bigcap}^{r}_{\Lambda_{K}}H^{1}_{\Sigma}(\bT_{K}).  
}
\] 
These diagrams show the existence of the canonical homomorphism 
${\rm ES}_{f,r}(T) \longrightarrow \cE(T)$, and the injectivity  follows from  Hypothesis~(H.3) and Lemma~\ref{lem:sub}. 
Since $H^{1}_{f,\Sigma}(\bT_{K})$ is a free $\Lambda_{K}$-module of rank $r$ by Hypothesis~(H.2), 
the second assertion follows from the exact sequence~\eqref{exactseq} and the definition of characteristic ideal. 
\end{proof}

\section{Higher rank Euler systems over totally real fields}
\label{sec:euler G_{m}} 

In this section, we will prove the main theorem of this paper. 

Throughout this section, let $k$ be a totally real field with degree $r$ and 
\[
\chi \colon G_{k} \longrightarrow \overline{\bQ}^{\times}
\] 
a non-trivial finite order even character.  
Let $p$ be an odd prime, coprime to the class number of $k$ and the order of $\chi$. 
We put 
\[
L := \overline{\bQ}^{\ker(\chi)}, \cO := \bZ_{p}[\im(\chi)], \, 
S := S_{\infty}(k) \cup S_{p}(k) \cup S_{\rm ram}(L/k), \, 
\text{ and } \Sigma := S_{p}(k). 
\]
Let 
\[
T := \cO(1) \otimes \chi^{-1}, 
\]
that is, $T \cong \cO$ as $\cO$-modules and the Galois group $G_{k,S}$ acts on $T$ via the character 
$\chi_{\rm cyc}\chi^{-1}$, where $\chi_{\rm cyc}$ denotes the cyclotomic character of $k$. 
We write $k_{\infty}/k$ for the cyclotomic $\bZ_{p}$-extension of $k$. 
We also set   
\[
\Lambda := \cO[[\Gal(k_{\infty}/k)]] \ \text{ and } \ \bT := T \otimes_{\cO} \cO[[\Gal(k_{\infty}/k)]]^{\iota}. 
\]
Let $\cK$ denote the maximal pro-$p$ abelian extension of $k$ satisfying $S_{\rm ram}(\cK/k) \cap S = \emptyset$.  
Let $\Omega$, $\Lambda_{K}$, and $\bT_{K}$ be as in \S\ref{sec:euler}. 
Put 
\[
H^{1}_{f, \Sigma}(\bT_{K}) := H^{1}_{\Sigma}(\bT_{K}) := \bigoplus_{\fp \in S_{p}(k)}H^{1}(G_{k_{\fp}},\bT_{K}). 
\]

In this setting, the injectivity of the direct sum of localization maps (the weak Leopoldt conjecture for the multiplicative group) is proved by Iwasawa in \cite{Iwa73a}:

\begin{proposition}%[{\cite[Proposition~1.3.2 and Remarks~1.3.3(i)]{PR00}}]
\label{prop:inj}
For a field $K \in \Omega$, the canonical homomorphism 
\[
H^{1}(G_{k,S_{K}},\bT_{K}) \longrightarrow H^{1}_{\Sigma}(\bT_{K})
\]
is injective, i.e.,  Hypothesis~(H.3) is satisfied. 
\end{proposition}

\begin{proof}
Let $\tilde{S}_{K}$ denote the set of primes of $KL$ above $S_{K}$. 
Then, by Shapiro's lemma, we have  the following commutative diagram: 
\[
\xymatrix{
H^{1}(G_{k,S_{K}},\bT_{K}) \ar[r] \ar[d] & \bigoplus_{\fp \in S_{p}(k)}H^{1}(G_{k_{\fp}},\bT_{K}) \ar[d]
\\
H^{1}(G_{KL, \widetilde{S}_{K}}, \bT') \ar[r]  & \bigoplus_{\fp \in S_{p}(KL)}H^{1}(G_{(KL)_{\fp}}, \bT')
}
\]
Here $\bT' := \cO(1) \otimes_{\cO} \cO[[\Gal(k_{\infty}/k)]]^{\iota}$ and the vertical arrows are the restriction maps.  
The inflation-restriction exact sequence and \cite[Lemma~B.3.2]{R} show that both of the vertical arrows are injective, and hence it suffices to prove that 
\[
H^{1}(G_{KL, \widetilde{S}_{K}}, \bT') \longrightarrow \bigoplus_{\fp \in S_{p}(KL)}H^{1}(G_{(KL)_{\fp}}, \bT')
\]
is injective. 

We note that \cite[Corollary~B.3.6]{R} implies 
\[
H^{1}(G_{KL, \widetilde{S}_{K}}, \bT') = H^{1}(G_{KL, S_{p}(KL) \cup S_{\infty}(KL)}, \bT'). 
\]
Hence, by the Poitou-Tate exact sequence, we have the exact sequence  
\[
H^{2}(G_{KL, S_{p}(KL) \cup S_{\infty}(KL)}, (\bT')^{\vee}(1))^{\vee}
\longrightarrow
H^{1}(G_{KL, \widetilde{S}_{K}}, \bT') \longrightarrow \bigoplus_{\fp \in S_{p}(KL)}H^{1}(G_{(KL)_{\fp}}, \bT'), 
\]
and the vanishing of $H^{2}(G_{KL, S_{p}(KL) \cup S_{\infty}(KL)}, (\bT')^{\vee}(1))$ follows from \cite[Theorem~11.3.2(ii)]{NSW} since $k_{\infty}/k$ is the cyclotomic $\bZ_{p}$-extension. 
\end{proof}

\subsection{Stickelberger elements and Iwasawa main conjecture}
We will recall the definition of Stickelberger elements and prove that 
the collection of Stickelberger elements is an Euler system of rank $0$. 

Let $K \in \Omega$ be a field and $n \geq 0$ an integer. 
Let $\mu_{p^{n}}$ denote the set of $p^{n}$-th roots of unity in $\overline{\bQ}$ and $\mu_{p^{\infty}} := \bigcup_{n>0}\mu_{p^{n}}$. 
To simplify the notation, we set
\[
G_{K,n} := \Gal(KL(\mu_{p^{n}})/k) = \Gal(K/k) \times \Gal(L/k) \times \Gal(k(\mu_{p^{n}})/k). 
\] 
We have the complex conjugate $c \in \Gal(L(\mu_{p})/k) = G_{k,1}$ since 
$L$ is a totally real field.  
For an $\cO[G_{k,1}]$-module $M$, we have  a decomposition  
\[
M = M^{c=1} \oplus M^{c=-1}. 
\]
Here $M^{c=\pm 1}$ denotes the $\pm 1$-eigenspace with respect to the action of the complex conjugate $c \in G_{k,1}$. 
Since $p \nmid [L(\mu_{p}) \colon k]$, we have   a decomposition  
\[
M^{c=-1} = \bigoplus_{\psi \colon {\rm odd}}M^{\psi}, 
\] 
where $\psi$ runs over all characters of $G_{k,1}$ satisfying $\psi(c)=-1$ and $M^{\psi}$ denotes the maximal submodule of $M$ on which $G_{k,1}$ acts via the character $\psi$. 
Let 
\[
\omega \colon G_{k,1} \longrightarrow \Gal(k(\mu_{p})/k) \longrightarrow \bZ_{p}^{\times} 
\] 
denote the Teichm\"ullar character. 
We write $M^{\flat}$ for the component obtained from $M^{c=-1}$ by removing $M^{\omega}$, that is, 
\[
M^{c=-1} = M^{\flat} \oplus M^{\omega}. 
\]
For an element $x$ of $M$, we write $x^{\flat}$ for the component of $x$ in $M^{\flat}$.
%Note that, if $M$ is an $\bZ_{p}[\Gal(L(\mu_{p})/k)]$-module, one can also define the module $M^{\flat}$. 

\begin{definition}
Let $\zeta_{KL,n,S}(s,\sigma)$ denote the partial zeta function for $\sigma \in G_{K,n}$: 
\[
\zeta_{KL,n,S}(s,\sigma) := \sum_{(\fa, KL(\mu_{p^{n}})/k) = \sigma}N(\fa)^{-s}
\]
where $\fa$ runs over all integral ideals of $k$ coprime to all the primes in $S_{K}$ such that the Artin symbol $(\fa, KL(\mu_{p^{n}})/k)$ is equal to $\sigma$ and 
$N(\fa)$ is the norm of $\fa$. 
\end{definition}
It is well-known that the function $\zeta_{KL,n,S}(s,\sigma)$ is continued to  a holomorphic function on $\bC \setminus \{1\}$. 
We put
\[
\theta_{KL,n,S} := \sum_{\sigma \in G_{K,n}}\zeta_{KL,n,S}(0,\sigma)\sigma^{-1} 
\]
which is contained in $\bQ[G_{K,n}]$ (see \cite{Sie70}). 

\begin{proposition}[{\cite[Proposition~IV.1.8]{tatebook}}]\label{prop:rel}
Let $K' \in \Omega$ with $K \subseteq K'$ and $n'$ an integer with $n \leq n'$. 
Then the image of $\theta_{K'L,n',S}$ in $\bQ[G_{K,n}]$ is 
\[
\left(\prod_{\fq \in S_{K'} \setminus S_{K}}(1-{\rm Frob}_{\fq}^{-1}) \right) \cdot \theta_{KL,n,S}. 
\]
\end{proposition}

%\begin{theorem}[{\cite{DR}}]\label{thm:integral}
%
%\end{theorem}

The elements $\{\theta_{KL,n,S}^{\flat}\}_{n>0}$ are norm-coherent by Proposition~\ref{prop:rel}. 
In addition, Deligne--Ribet proved in \cite{DR} that the element $\theta_{KL,n,S}^{\flat}$ is contained in 
$\bZ_{p}[G_{K,n}]^{\flat}$. 
Hence one can define an element 
\[
\theta_{KL,\infty,S}^{\flat} := \varprojlim_{n>0}\theta_{KL,n,S}^{\flat} \in \bZ_{p}[[\Gal(KL(\mu_{p^{\infty}})/k)]]^{\flat}. 
\]
Let 
\[
{\rm Tw} \colon \bZ_{p}[[\Gal(KL(\mu_{p^{\infty}})/k)]] \longrightarrow \bZ_{p}[[\Gal(KL(\mu_{p^{\infty}})/k)]]
\]
denote the homomorphism induced by $\sigma \mapsto \chi_{\rm cyc}(\sigma)\sigma^{-1}$ for $\sigma \in \Gal(KL(\mu_{p^{\infty}})/k)$. 
We then get an element
\[
%\tilde{L}_{p,K}^{\#} := {\rm Tw}(\theta_{KL,\infty,S}^{\flat}) \ \text{ and } \ 
\tilde{L}_{p,K}^{\chi} :=  e_{\chi}{\rm Tw}(\theta_{KL,\infty,S}^{\flat}) \in \Lambda_{K}, 
\]
where 
\[
e_{\chi} := \frac{1}{[L(\mu_{p}) \colon k]} \sum_{\sigma \in \Gal(L(\mu_{p})/k)}\chi(\sigma)\sigma^{-1}. 
\]
(Note that $\chi$ is an even character.)

\begin{definition}
For each prime $\fq \not\in S$, we set 
\[
u_{\fq} := \chi({\rm Frob}_{\fq})^{-1}\chi_{\rm cyc}({\rm Frob}_{\fq}){\rm Frob}_{\fq}^{-1} \in \cO[[\Gal(k^{\rm ab}/k)]]^{\times}. 
\] 
For a field $K \in \Omega$, we define a modified $p$-adic $L$-function $L_{p,K}^{\chi} \in \Lambda_{K}$ to be 
\[
L_{p,K}^{\chi} := \left(\prod_{\fq \in S_{\rm ram}(K/k)}(-u_{\fq}) \right) \cdot \tilde{L}_{p,K}^{\chi}. 
\]
\end{definition}

\begin{lemma}\label{lem:rel2}
The system $\{L_{p,K}^{\chi}\}_{K \in \Omega}$ is contained in $\cE(T) = \mathrm{ES}_{0}(T)$. 
\end{lemma}
\begin{proof}
Let $K, K' \in \Omega$ with $K \subseteq K'$. 
We note that the diagram 
\[
\xymatrix{
\cO[[\Gal(K'L(\mu_{p^{\infty}})/k)]] \ar[r]^-{{\rm Tw}} \ar[d]^{\pi} & 
\cO[[\Gal(K'L(\mu_{p^{\infty}})/k)]] \ar[d]^{\pi}
\\
\cO[[\Gal(KL(\mu_{p^{\infty}})/k)]] \ar[r]^-{{\rm Tw} } & 
\cO[[\Gal(KL(\mu_{p^{\infty}})/k)]] 
}
\]
commutes, where $\pi$ denotes the canonical projection. 
Hence Proposition~\ref{prop:rel} shows  
\begin{align*}
\pi(\tilde{L}_{p,K'}^{\chi}) &= e_{\chi}( {\rm Tw}(\pi(\theta_{K',\infty,S}^{\flat}))) 
\\
&= e_{\chi}\left({\rm Tw}\left(\biggl(\prod_{\fq \in S_{K'} \setminus S_{K}}(1-{\rm Frob}_{\fq}^{-1})\biggr) \cdot (\theta_{K,\infty,S}^{\flat})\right)\right) 
\\
&= e_{\chi}\left({\rm Tw}\biggl(\prod_{\fq \in S_{K'} \setminus S_{K}}(1-{\rm Frob}_{\fq}^{-1}) \biggr) \right)\cdot \tilde{L}_{p,K}^{\chi}
\\
&= \biggl(\prod_{\fq \in S_{K'} \setminus S_{K}}(1-u_{\fq}^{-1})\biggr) \cdot \tilde{L}_{p,K}^{\chi}. 
\end{align*}
Since $P_{\fq}(x) = 1 - \chi({\rm Frob}_{\fq})^{-1}\chi_{\rm cyc}({\rm Frob}_{\fq})x$ by definition, we have  
\[
P_{\fq}({\rm Frob}_{\fq}^{-1}) = 1-u_{\fq} = -u_{\fq} (1-u_{\fq}^{-1}).
\] 
As $S_{K'} \setminus S_{K} = S_{\rm ram}(K') \setminus S_{\rm ram}(K)$, we conclude  
\begin{align*}
\pi(L_{p,K'}^{\chi}) &= \biggl(\prod_{\fq \in S_{\rm ram}(K'/k)}(-u_{\fq})\biggr) \cdot \pi(\tilde{L}_{p,K'}^{\chi})
\\
&= \biggl(\prod_{\fq \in S_{\rm ram}(K'/k)}(-u_{\fq}) \prod_{\fq \in S_{K'} \setminus S_{K}}(1-u_{\fq}^{-1})\biggr) \cdot \tilde{L}_{p,K}^{\chi}
\\
&= \biggl(\prod_{\fq \in S_{\rm ram}(K/k)}(-u_{\fq}) \prod_{\fq \in S_{K'} \setminus S_{K}} P_{\fq}({\rm Frob}_{\fq}^{-1}) \biggr)\cdot \tilde{L}_{p,K}^{\chi}
\\
&=  \biggl(\prod_{\fq \in S_{K'} \setminus S_{K}} P_{\fq}({\rm Frob}_{\fq}^{-1})\biggr) \cdot L_{p,K}^{\chi}. 
\end{align*}
\end{proof}

\subsection{Iwasawa modules}

%In this subsection, we will compute the characteristic ideals of certain Iwasawa modules by using the Stickelberger elements $\{L_{p,K}^{\chi}\}_{K \in \Omega}$. 
In this section, we will introduce several Iwasawa modules and recall their important properties. 
We will freely use the facts in Appendix~\ref{subsec:selmer-str}. 

For a topological $\bZ_{p}$-module $M$, let 
\[
M^{\vee} := \Hom_{\rm cont}(M, \bQ_{p}/\bZ_{p})
\] 
denote the Pontryagin dual of $M$. 

\begin{definition}
For a field $K \in \Omega$, 
we write $M_{KL,\infty}$ for the maximal $p$-ramified pro-$p$ abelian extension of $KL(\mu_{p^{\infty}})$ and set 
\begin{align*}
%X_{K}^{\#} := \Gal(M_{KL,\infty}/k_{\infty}KL)^{\#} \ \text{ and } \ 
X_{K}^{\chi} := e_{\chi}\left(\cO \otimes_{\bZ_{p}} \Gal(M_{KL,\infty}/KL(\mu_{p^{\infty}}))\right). 
\end{align*}
We note that we have  the canonical isomorphism 
\[
H^{1}_{\cF_{\rm str}^{*}}(k, \bT_{K}^{\vee}(1))^{\vee} \cong X_{K}^{\chi}
\]
(see Example~\ref{exa:selmer-str} and Definition~\ref{def:dual-sel} for the definition of $H^{1}_{\cF_{\rm str}^{*}}(k, \bT_{K}^{\vee}(1))$). 
\end{definition}

Let  $\fq \nmid p$ be a prime of $k$. 
By \cite[Corollary~B.3.6]{R}, 
 the subgroup of unramified cohomology classes in $H^{1}(G_{k_{\fq}}, \bT_{K})$ coincides with $H^{1}(G_{k_{\fq}}, \bT_{K})$. 
In particular, we have   
\[
H^{1}_{\cF_{\rm str}}(G_{k_{\fq}}, \bT_{K}) = H^{1}_{\cF_{\rm can}}(G_{k_{\fq}}, \bT_{K}) = H^{1}(G_{k_{\fq}}, \bT_{K}). 
\]
Furthermore, Proposition~\ref{prop:inj} shows that $H^{1}_{\cF_{\rm str}}(k, \bT_{K})$ vanishes.  
Hence applying Theorem~\ref{pt} with $\cF_{1} = \cF_{\rm str}$ and $\cF_{2} = \cF_{\rm can}$, we obtain an exact sequence of $\Lambda_{K}$-modules: 
\begin{align}\label{exact:fundamental}
0 \longrightarrow H^{1}(G_{k,S_{K}},\bT_{K}) \longrightarrow H^{1}_{\Sigma}(\bT_{K}) \longrightarrow X_{K}^{\chi} 
\longrightarrow H^{1}_{\cF_{\rm can}^{*}}(k, \bT_{K}^{\vee}(1))^{\vee} \longrightarrow 0. 
\end{align}

%Concerning the modules $X_{K}^{\chi}$ and $Y_{K}^{\chi}$, 
%it is well-known that  both are torsion as a $\Lambda$-module and that 
%$X_{K}^{\chi}$ has no nonzero finite $\Lambda$-submodules (see \cite{Iwa73a}). 

%Furthermore, there is following conjecture (that are conjectured by Iwasawa):  
%
%\begin{conjecture}
%The $\mu$-invariant of $X_{K}^{\chi}$ vanishes, i.e., $X_{K}^{\chi}$ is free as a $\bZ_{p}$-module. 
%\end{conjecture}

%\begin{conjecture}[Greenberg]
%The order of the module $Y_{K}^{\chi}$ is finite. 
%\end{conjecture}
%
%\begin{remark}
%Gereenberg conjecture implies that the $\Lambda_{K}$-modules $R_{f,\Sigma}(\bT_{K})$ and $X_{K}^{\chi}$ are pseudo-isomorphic, and so ${\rm char}_{\Lambda_{K}}(R_{f,\Sigma}(\bT_{K})) = {\rm char}_{\Lambda_{K}}(X_{K}^{\chi})$.  
%\end{remark}

\begin{definition}
We denote by $M_{KL,S,\infty}$ the maximal $S_{K}$-ramified pro-$p$ abelian extension of $KL(\mu_{p^{\infty}})$ and put 
\begin{align*}
X_{K, S}^{\chi} := e_{\chi}\left(\cO \otimes_{\bZ_{p}} \Gal(M_{KL,S,\infty}/KL(\mu_{p^{\infty}}))\right). 
\end{align*}
We note that we have  the canonical isomorphism 
\[
H^{1}(G_{k,S_{K}}, \bT_{K}^{\vee}(1))^{\vee} = H^{1}_{\cF_{S_{K}}^{*}}(k, \bT_{K}^{\vee}(1))^{\vee} \cong X_{K, S}^{\chi}. 
\]
\end{definition}

Applying Theorem~\ref{pt} with $\cF_{1} = \cF_{S_{K}}$ and $\cF_{2} = \cF_{\rm str}$, we get an exact sequence of $\Lambda_{K}$-modules: 
\begin{align}\label{exact:diff}
0 \longrightarrow \bigoplus_{\fq \in S_{K} \setminus S} H^{1}(G_{k_{\fq}}, \bT_{K}) \longrightarrow 
X_{K,S}^{\chi} \longrightarrow X_{K}^{\chi} \longrightarrow 0. 
\end{align}
%Here we use the fact that, for each prime $\fq \not \in S_{p}(k)$ of $k$,  the subgroup of unramified cohomology classes in $H^{1}(k_{\fq}, \bT_{K})$ coincides with $H^{1}(k_{\fq}, \bT_{K})$ (see \cite[Corollary~B.3.6]{R}). 

\begin{remark}\label{rem:vanish-ram}
If $\fq \in S_{\rm ram}(L/k) \setminus S_{p}(k)$, then local duality shows that 
\[
H^{2}(G_{k_{\fq}},T/\fm T) \cong H^{0}(G_{k_{\fq}},(T/\fm T)^{\vee}(1)) = 0.
\] 
Furthermore, since $k(\mu_{p})/k$ is unramified at $\fq$, we also see  $H^{0}(G_{k_{\fq}},T/\fm T) = 0$. Hence local Euler characteristic formula implies that $H^{1}(G_{k_{\fq}},T/\fm T)$ vanishes. 
In particular,  the module $H^{1}(G_{k_{\fq}},\bT)$ vanishes, and so 
\begin{align*}
X_{k,S}^{\chi} = X_{k}^{\chi}
\end{align*}
by the exact sequence~\eqref{exact:diff}. 
\end{remark}

%For any prime $\fq \not \in S_{p}(k)$ of $k$, the module $H^{1}(k_{\fq}, \bT_{K})$ is a free $\bZ_{p}$-module, and hence we obtain the following lemma: 
%
%\begin{lemma}\label{lemma:mu-vanish-equiv}
%If the $\mu$-invariant of $X_{K}^{\chi}$ vanishes, then so does $X_{K,S}^{\chi}$. 
%\end{lemma}

\begin{lemma}\label{lemma:mu-vanish-H2}
The $\mu$-invariant of  $X_{K,S}^{\chi}$ vanishes, i.e., the module $X_{K,S}^{\chi}$ is $p$-torsion-free 
 if and only if the module $H^{2}(G_{k,S_{K}}, (\bT_{K}/p\bT_{K})^{\vee}(1))$ vanishes. 
\end{lemma}
\begin{proof}
Since $X_{K,S}^{\chi}$ is isomorphic to $H^{1}(G_{k,S_{K}}, \bT_{K}^{\vee}(1))^{\vee}$, we have  an exact sequence of $\Lambda_{K}$-modules 
\[
0 \longrightarrow H^{2}(G_{k,S_{K}}, \bT_{K}^{\vee}(1))^{\vee} \otimes_{\bZ} \bF_{p} \longrightarrow 
H^{2}(G_{k,S_{K}}, (\bT_{K}/p\bT_{K})^{\vee}(1))^{\vee} \longrightarrow X_{K,S}^{\chi}[p] \longrightarrow 0. 
\]
By the Poitou--Tate exact sequence, we also have  the exact sequence 
\[
\bigoplus_{\fq \in S_{K}}H^{0}(G_{k_{\fq}}, \bT_{K}) \longrightarrow H^{2}(G_{k,S_{K}}, \bT_{K}^{\vee}(1))^{\vee}\longrightarrow H^{1}(G_{k,S_{K}}, \bT_{K}) \longrightarrow \bigoplus_{\fq \in S_{K}}H^{1}(G_{k_{\fq}}, \bT_{K}). 
\]
\cite[Lemma~B.3.2]{R} implies that the module $H^{0}(G_{k_{\fq}}, \bT_{K})$ vanishes for each prime $\fq \in S_{K}$. 
Hence Proposition~\ref{prop:inj} shows that $H^{2}(G_{k,S_{K}}, \bT_{K}^{\vee}(1))$ vanishes. 
Therefore, we obtain an isomorphism 
\[
H^{2}(G_{k,S_{K}}, (\bT_{K}/p\bT_{K})^{\vee}(1))^{\vee} \stackrel{\sim}{\longrightarrow} X_{K,S}^{\chi}[p], 
\]
which proves this lemma. 
\end{proof}

%By using Lemmas~\ref{lemma:mu-vanish-equiv}~and~\ref{lemma:mu-vanish-H2}, 

By using Lemma~\ref{lemma:mu-vanish-H2}, 
we obtain the following well-known proposition: 

\begin{proposition}[{\cite{Iwa73b}}]\label{prop:mu=0}
Let $K \in \Omega$ be a field.  
 If the $\mu$-invariant of $X_{k}^{\chi}$ vanishes, then so does $X_{K,S}^{\chi}$. 
\end{proposition}
\begin{proof}
By the Poitou--Tate exact sequence and \cite[Lemma~B.3.2]{R}, 
we have  the following commutative diagram with exact rows: 
\begin{align*}
\xymatrix{
0 = H^{2}(G_{k,S_{K}}, (\bT/p\bT)^{\vee}(1))^{\vee} \ar@{^{(}->}[r] \ar[d] & H^{1}(G_{k,S_{K}}, \bT/p \bT) \ar[r] \ar[d] & \bigoplus_{\fq \in S_{K}} H^{1}(G_{k_{\fq}}, \bT/p \bT) \ar[d]
\\
 H^{2}(G_{k,S_{K}}, (\bT_{K}/p\bT_{K})^{\vee}(1))^{\vee} \ar@{^{(}->}[r] & H^{1}(G_{k,S_{K}}, \bT_{K}/p \bT_{K}) \ar[r]  & \bigoplus_{\fq \in S_{K}} H^{1}(G_{k_{\fq}}, \bT_{K}/p \bT_{K}). 
}
\end{align*}
Let $\fq \in S_{K}$ be a prime and fix a prime $\fq_{K}$ of $K$ above $\fq$. 
The inflation-restriction sequence and \cite[Lemma~B.3.2]{R} show
\begin{align*}
\ker\left(H^{1}(G_{k_{\fq}}, \bT/p \bT) \longrightarrow H^{1}(G_{K_{\fq_{K}}}, \bT/p \bT)\right) &= H^{1}(\Gal(K_{\fq_{K}}/k_{\fq}), (\bT/p\bT)^{G_{K_{\fq_{K}}}}) = 0.  
\end{align*}
Hence the right vertical map in the diagram is injective. 
Furthermore, $H^{2}(G_{k,S_{K}}, (\bT/p\bT)^{\vee}(1))^{\vee}$ vanishes by Lemma~\ref{lemma:mu-vanish-H2} and 
\[
H^{1}(G_{k,S_{K}}, \bT/p \bT) \stackrel{\sim}{\longrightarrow} 
H^{1}(G_{K,S_{K}}, \bT/p \bT)^{\Gal(K/k)} = 
H^{1}(G_{k,S_{K}}, \bT_{K}/p \bT_{K})^{\Gal(K/k)}
\]
by the inflation-restriction exact sequence and \cite[Lemma~B.3.2]{R}. 
Thus we conclude that 
\[
(H^{2}(G_{k,S_{K}}, (\bT_{K}/p\bT_{K})^{\vee}(1))^{\vee})^{\Gal(K/k)} = 0, 
\]
which implies $H^{2}(G_{k,S_{K}}, (\bT_{K}/p\bT_{K})^{\vee}(1)) = 0$ by Nakayama's lemma. 
Therefore, this proposition follows from 
%Lemmas~\ref{lemma:mu-vanish-equiv}~and~\ref{lemma:mu-vanish-H2}. 
Lemmas~\ref{lemma:mu-vanish-H2}. 
\end{proof}

\begin{corollary}\label{cor:kuri03}
If the $\mu$-invariant of $X_{k}^{\chi}$ vanishes, then, for a field $K \in \Omega$, we have  
\[
L_{p,K}^{\chi}  \in {\rm char}_{\Lambda_{K}}(X_{K}^{\chi}). 
\]
\end{corollary}
\begin{proof}
By Proposition~\ref{prop:mu=0}, the $\mu$-invariant of $X_{K,S}^{\chi}$ vanishes. 
In this case, Kurihara showed ${\rm Fitt}^{0}_{\Lambda_{K}}(X_{K, S}^{\chi}) = L_{p,K}^{\chi}\Lambda_{K}$ (see \cite[Proposition~2.1]{Kur03}). 
Hence we conclude that 
\[
 L_{p,K}^{\chi}  \in {\rm Fitt}^{0}_{\Lambda_{K}}(X_{K,S}^{\chi}) \subseteq 
 {\rm Fitt}^{0}_{\Lambda_{K}}(X_{K}^{\chi}) \subseteq {\rm char}_{\Lambda_{K}}(X_{K}^{\chi}). 
\]
\end{proof}

\subsection{Higher rank Euler systems for the multiplicative group}
In this subsection, we will show the existence of a higher rank Euler system related to the (modified) Stickelberger elements $\{L_{p,K}^{\chi}\}_{K \in \Omega}$. 

Recall that the prime $p > 2$ is coprime to the class number of $k$ and the order of $\chi$. 
Since $H^{1}_{f,\Sigma}(\bT_{K}) = H^{1}_{\Sigma}(\bT_{K})$, 
by the exact sequence~\eqref{exact:fundamental}, we obtain an exact sequence of $\Lambda_{K}$-modules:  
\begin{align}\label{keyexactseq}
0 \longrightarrow R_{f,\Sigma}(\bT_{K}) \longrightarrow X_{K}^{\chi} 
\longrightarrow H^{1}_{\cF^{*}_{\rm can}}(k, \bT_{K}^{\vee}(1))^{\vee} \longrightarrow 0. 
\end{align}
%Recall that $k$ is a totally real field with $r := [k \colon \bQ]$ and that $\chi \colon G_{k} \longrightarrow \overline{\bQ}^{\times}$ is a non-trivial finite order even character.  

Recall that $\fm$ denotes the maximal ideal of $\cO$.

\begin{lemma}\label{lem:str}
Suppose that 
\[
H^{0}(G_{k_{\fp}},T/\fm T) = H^{2}(G_{k_{\fp}},T/\fm T) = 0
\]
for each prime $\fp \in S_{p}(k)$. 
Then Hypotheses (H.0) -- (H.3) are satisfied. 
%Furthermore, if we fix an isomorphism 
%\[
%\varprojlim_{K\in\Omega}H^{1}_{\Sigma}(\bT_{K}) \stackrel{\sim}{\longrightarrow} \Lambda[[\Gal(\cK/k)]]^{r},
%\] 
%then this fixed isomorphism induces an injective homomorphism 
%\[
%{\rm ES}_{r}(T) \hooklongrightarrow \cE(T)
%\]
%satisfying
%\[
%\cE(T) \cap \prod_{K \in \Omega}{\rm char}_{\Lambda_{K}}(R_{f,\Sigma}(\bT_{K})) = 
%{\rm im}\left({\rm ES}_{r}(T) \hooklongrightarrow  \cE(T) \right). 
%\] 
\end{lemma}
\begin{proof}
Since we assume that $p$ is coprime to the class number of $k$ and $H^{0}(G_{k_{\fp}},T/\fm T) = 0$ for each prime $\fp \in S_{p}(k)$, 
Hypotheses~(H.0) and (H.1) are satisfied. 
Furthermore, we also assume that, for any prime $\fp \in S_{p}(k)$, the module $H^{2}(G_{k_{\fp}},T/\fm T)$ vanishes,  
Corollary~\ref{cor:H2vanish} implies Hypothesis~(H.2). 
Hypothesis~(H.3) follows from Proposition~\ref{prop:inj}. 
%The latter assertion  follows from Proposition~\ref{prop:main}. 
\end{proof}

\begin{theorem}\label{thm:main2}
Let us fix an isomorphism $\varprojlim_{K \in \Omega}H^{1}_{\Sigma}(\bT_{K}) \stackrel{\sim}{\longrightarrow} 
\Lambda[[\Gal(\cK/k)]]^{r}$. 
Suppose that 
\begin{itemize}
\item $H^{0}(G_{k_{\fp}},T/\fm T) = H^{2}(G_{k_{\fp}},T/\fm T) = 0$
for each prime $\fp \in S_{p}(k)$, and 
\item the $\mu$-invariant of $X_{k}^{\chi}$ vanishes. 
\end{itemize}
Then there is a unique Euler system $c^{\rm DR} = \{c_{K}^{\rm DR}\}_{K \in \Omega} \in {\rm ES}_{r}(T)$ of rank $r$ such that the image of $c^{\rm DR}$ under the map ${\rm ES}_{r}(T) \hooklongrightarrow \cE(T)$ (defined in Proposition~\ref{prop:main}) is $\{L_{p,K}^{\chi}\}_{K \in \Omega}$. 
\end{theorem}

\begin{remark}
By using the Kummer duality and the Ferrero--Washington theorem (proved in \cite{FeWa79}), 
we see that the $\mu$-invariant of $X_{k}^{\chi}$ vanishes if $L/\bQ$ is an abelian extension. 
Furthermore, by using the argument as the proof of Proposition~\ref{prop:mu=0}, we also see that 
the $\mu$-invariant of $X_{k}^{\chi}$ vanishes if there is a subfield $L' \subseteq L$ such that 
$L'/\bQ$ is an abelian extension and that $L/L'$ is a Galois extension of $p$-power degree. 
\end{remark}

\begin{proof}
Let $K \in \Omega$ be a field.  
Note that 
\[
H^{1}_{f}(G_{k,S_{K}},\bT_{K}) = H^{1}(G_{k,S_{K}},\bT_{K}) \ \text{ and } \ {\rm ES}_{f,r}(T) = {\rm ES}_{r}(T).  
\]
%and that Proposition~\ref{prop:mu=0} shows that the $\mu$-invariant of $X_{K}^{\chi}$ vanishes. 
Since $R_{f,\Sigma}(\bT_{K})$ is a submodule of $X_{K}^{\chi}$ by the exact sequence \eqref{keyexactseq}, Proposition~\ref{prop:lambda-inequality} shows that 
\[
{\rm char}_{\Lambda_{K}}(X_{K}^{\chi}) \subseteq {\rm char}_{\Lambda_{K}}(R_{f,\Sigma}(\bT_{K})).
\] 
Furthermore, Lemma~\ref{lem:rel2} and Corollary~\ref{cor:kuri03} imply
\[
 \{L_{p,K}^{\chi}\}_{K \in \Omega} \in \cE(T) \cap \prod_{K \in \Omega}{\rm char}_{\Lambda_{K}}(X_{K}^{\chi}) \subseteq 
\cE(T) \cap \prod_{K \in \Omega}{\rm char}_{\Lambda_{K}}(R_{f,\Sigma}(\bT_{K})). 
\]  
Thus, by Proposition~\ref{prop:main} and  Lemma~\ref{lem:str}, we have  
\[
\{L_{p,K}^{\chi}\}_{K \in \Omega} \in \cE(T) \cap \prod_{K \in \Omega}{\rm char}_{\Lambda_{K}}(R_{f,\Sigma}(\bT_{K})) = 
{\rm im}\left({\rm ES}_{r}(T) \hooklongrightarrow  \cE(T) \right). 
\]
\end{proof}

\section{Stark systems}\label{sec:stark}
In this section, we use the same notations as in \S\ref{sec:euler G_{m}}. 
We will freely use the results in Appendix~\ref{sec:bi-dual}.
 
Throughout this section, we fix a field $K \in \Omega$ and suppose that 
\begin{itemize}
\item $p$ is coprime to the class number of $k$ and the order of $\chi$, 
\item $H^{0}(G_{k_{\fp}},T/\fm T) = H^{2}(G_{k_{\fp}},T/\fm T) = 0$
for each prime $\fp \in S_{p}(k)$, and 
\item $H^{2}(G_{k_{\fq}},T/\fm T) = 0$ for each prime $\fq \in S_{\rm ram}(K/k)$. 
\end{itemize}

Let $n>0$ be an integer. For a field $K' \in \Omega$, put  
\[
R_{K',n} := \cO/p^{n}\cO[\Gal(k_{n}K'/k)] \ \text{ and } \ 
T_{K',n} := T \otimes_{\bZ_{p}}\bZ_{p}/p^{n}\bZ_{p}[\Gal(k_{n}K'/k)]^{\iota}, 
\] 
where $k_{n}$ denotes the $n$-th layler of the cyclotomic $\bZ_{p}$-extension $k_{\infty}/k$. 
Note that $R_{K', n}$ is a zero-dimensional Gorenstein local ring. 

By using the second assumption and Corollary~\ref{cor:H2vanish}, for each prime $\fp \in S_{p}(k)$, one can  fix an isomorphism 
\[
\varprojlim_{K' \in \Omega}H^{1}(G_{k_{\fp}}, \bT_{K'})   \stackrel{\sim}{\longrightarrow} 
\Lambda[[\Gal(\cK/k)]]^{[k_{\fp} \colon \bQ_{p}]},  
\] 
which induces an isomorphism 
\[
\varprojlim_{K' \in \Omega}H^{1}_{\Sigma}(\bT_{K'})   \stackrel{\sim}{\longrightarrow} 
\Lambda[[\Gal(\cK/k)]]^{r}.   
\]

\begin{remark}\label{rem:vanish}
Let $\fq \in S_{\rm ram}(K/k)$ be a prime. Since $K/k$ is an abelian $p$-power degree extension and $\fq \not\in S_{p}(k)$, we have  
\[
N(\fq) \equiv 1 \pmod{p}.  
\] 
Hence local duality 
%$H^{2}(G_{k_{\fq}},T/\fm T) \cong H^{0}(G_{k_{\fq}}, (T/\fm T)^{\vee}(1))$ 
and the vanishing of $H^{2}(G_{k_{\fq}},T/\fm T)$ imply the vanishing of $H^{0}(G_{k_{\fq}},T/\fm T)$. 
Thus, for any integer $n>0$, local Euler characteristic formula shows 
\[
H^{1}(G_{k_{\fq}}, T/\fm T) = 0. 
\]
By Remark~\ref{rem:vanish-ram}, for a prime $\fq \in S_{\rm ram}(L/k) \setminus S_{p}(k)$, we also have    
\[
H^{0}(G_{k_{\fq}},T/\fm T) = H^{1}(G_{k_{\fq}},T/\fm T) = H^{2}(G_{k_{\fq}},T/\fm T) = 0. 
\]
In particular,  the exact sequence~\eqref{exact:diff} shows that 
\begin{align}\label{keyeq}
X_{K,S}^{\chi} = X_{K}^{\chi}. 
\end{align}
\end{remark}

\subsection{Review of the theory of Stark systems}
In this subsection, we recall the definition of Stark systems and some results proved 
in \cite{bss, sakamoto}. 

We write $\cH$ for the Hilbert class field of $k$. 
Note that $p \nmid [\cH \colon k]$ since $p$ is coprime to the class number of $k$. 
Recall that $L = \overline{k}^{\ker(\chi)}$. 
Set  $q := \# \cO/\fm$ and $\ell_{K, n} := {\rm length}_{R_{K,n}}(R_{K,n})$. 
Let  
\[
\cP_{n} := \cP_{K,n} := \{\fq \not\in S_{K} \mid 
\text{$\fq$ splits completely in $\cH KL(\mu_{q^{n + \ell_{K, n}}}, (\cO_{k}^{\times})^{1/q^{n + \ell_{K, n}}})$}\}, 
\]
where $(\cO_{k}^{\times})^{1/q^{n + \ell_{K, n}}} := \{x \in \overline{k}^{\times} \mid x^{q^{n + \ell_{K, n}}} \in \cO_{k}^{\times}\}$. 
We denote by $\cN_{n}$  the set of square-free products of primes in $\cP_{n}$. 
For a prime $\fq \in \cP_{n}$, let $k(\fq)$ be the maximal $p$-extension inside the ray class field modulo $\fq$. 
For an ideal $\fn \in \cN_{n}$ we set 
\[
k(\fn) := \prod_{\fq \mid \fn}k(\fq) \ \textrm{ and } \ K(\fn) := k(\fn)K.
\] 
We write $\nu(\fn) := \# \{\fq \mid \fn\}$ for the number of prime divisors of $\fn \in \cN_{n}$. 

By definition, for a prime $\fq \in \cP_{n}$, the Frobenius element ${\rm Fr}_{\fq}$ acts trivially on $T_{K,n}$. Hence the module
\[
H^{1}_{f}(G_{k_{\fq}},T_{K,n}) := H^{1}(\Gal(k_{\fq}^{\rm ur}/k_{\fq}),T_{K,n}) \subseteq 
H^{1}(G_{k_{\fq}},T_{K,n})
\]
is a free $R_{K,n}$-module of rank $1$ (see \cite[Lemma~1.2.1(i)]{MRkoly}). Here $k_{\fq}^{\rm ur}$ denotes the maximal unramified extension of $k_{\fq}$. 
Furthermore, \cite[Lemma~1.2.3]{MRkoly} implies that 
\[
H^{1}_{/f}(G_{k_{\fq}},T_{K,n}) := H^{1}(G_{k_{\fq}},T_{K,n})/H^{1}_{f}(G_{k_{\fq}},T_{K,n})
\]
is also free of rank $1$.

\begin{definition}
Let $\fn \in \cN_{n}$ be an ideal and $\cF$ a Selmer structure on $T_{K,n}$. 
We then define a  Selmer structure $\cF^{\fn}$ on $T_{K,n}$ by the following data:
\begin{itemize}
\item $S(\cF^{\fn}) := S(\cF) \cup \{\fq \mid \fn\}$,  
\item we define a local condition at a prime $\fq \mid \fn$ by  
\begin{align*}
H^{1}_{\cF^{\fn}}(G_{k_{\fq}}, T_{K,n}) := H^{1}(G_{k_{\fq}}, T_{K,n}), 
\end{align*}
\item we define a local condition at a prime $\fq \nmid \fn$ by 
\[
H^{1}_{\cF^{\fn}}(G_{k_{\fq}}, T) := H^{1}_{\cF}(G_{k_{\fq}}, T).
\] 
\end{itemize}

\end{definition}

Let $\cF$ be a Selmer structure on $T_{K,n}$. 
Then, for any ideals $\fn,\fm \in \cN_{n}$ with $\fn \mid \fm$, we have  a cartesian diagram
\begin{align}\label{eq:exact1}
\begin{split}
\xymatrix{
H^{1}_{\cF^{\fn}}(k, T_{K,n}) \ar@{^{(}->}[r] \ar[d] & H^{1}_{\cF^{\fm}}(k, T_{K,n})  \ar[d]
\\
\bigoplus_{\fq \mid \fn} H^{1}_{/f}(G_{k_{\fq}},T_{K,n}) \ar@{^{(}->}[r] & \bigoplus_{\fq \mid \fm} H^{1}_{/f}(G_{k_{\fq}},T_{K,n}). 
}
\end{split}
\end{align}

\begin{definition}\label{def:det}
For an ideal $\fn \in \cN_{n}$, we put 
\[
W_{\fn} := \bigoplus_{\fq \mid \fn} H^{1}_{/f}(G_{k_{\fq}},T_{K,n})^{*}. 
\]
\end{definition}

By Definition~\ref{def:map}, 
the cartesian diagram~\eqref{eq:exact1} 
%and \eqref{eq:exact2} 
induces a homomorphism
\begin{align}\label{map:stark}
\Phi_{\fm,\fn} \colon {\bigcap}^{r' +\nu(\fm)}_{R_{K,n}}H^{1}_{\cF^{\fm}}(k, T_{K,n}) \otimes_{R_{K,n}} \det(W_{\fm}) \longrightarrow {\bigcap}^{r'+\nu(\fn)}_{R_{K,n}}H^{1}_{\cF^{\fn}}(k, T_{K,n}) \otimes_{R_{K,n}} \det(W_{\fn}) 
%\text{ and }
%\\
%&{\bigcap}^{\nu(\fm)}_{R_{K,n}}H^{1}_{\cF_{\rm str}^{\fm}}(k, T_{K,n}) \otimes_{R_{K,n}} \det(W_{\fm}) \longrightarrow {\bigcap}^{\nu(\fn)}_{R_{K,n}}H^{1}_{\cF_{\rm str}^{\fn}}(k, T_{K,n}) \otimes_{R_{K,n}} \det(W_{\fn}). 
\end{align}
for any integer $r' \geq 0$.

\begin{definition}
Let $r' \geq 0$ be an integer. 
The module ${\rm SS}_{r'}(T_{K,n},\cF)$ of Stark systems of rank $r'$ is defined by the inverse limit with respect to $\Phi_{\fm, \fn}$: 
\begin{align*}
{\rm SS}_{r'}(T_{K,n},\cF) &:= \varprojlim_{\fn \in \cN_{n}}{\bigcap}^{r' + \nu(\fn)}_{R_{K,n}}H^{1}_{\cF^{\fn}}(k, T_{K,n}) \otimes_{R_{K,n}} \det(W_{\fn}). 
\end{align*}
Throughout this section, by fixing isomorphism $\det(W_{\fn}) \cong R_{K,n}$, 
we identify the homomorphism~\eqref{map:stark} with 
\[
\Phi_{\fm,\fn} \colon {\bigcap}^{r' +\nu(\fm)}_{R_{K,n}}H^{1}_{\cF^{\fm}}(k, T_{K,n}) \longrightarrow {\bigcap}^{r'+\nu(\fn)}_{R_{K,n}}H^{1}_{\cF^{\fn}}(k, T_{K,n}),  
\] 
and we regard
\begin{align*}
{\rm SS}_{r'}(T_{K,n},\cF) \subseteq \prod_{\fn \in \cN_{n}}{\bigcap}^{r' + \nu(\fn)}_{R_{K,n}}H^{1}_{\cF^{\fn}}(k, T_{K,n}).  
\end{align*} 
\end{definition}

 To simplify the notation, we put $H^{1}_{\Sigma}(T_{K,n}) := \bigoplus_{\fp \in S_{p}(k)}H^{1}(G_{k_{\fp}},T_{K,n})$. 
We note that, by Corollary~\ref{cor:H2vanish}, the fixed isomorphism $\varprojlim_{K \in \Omega}H^{1}_{\Sigma}(\bT_{K}) \stackrel{\sim}{\longrightarrow} 
\Lambda[[\Gal(\cK/k)]]^{r}$ induces an $R_{K,n}$-isomorphism 
\[
H^{1}_{\Sigma}(T_{K,n}) \stackrel{\sim}{\longrightarrow} 
R_{K,n}^{r}. 
\]

%\begin{definition}\ 
%\begin{itemize}
%\item[(i)] Let $\{e_{1},\ldots, e_{r}\}$ denote the standard basis of $\Lambda[[\Gal(\cK/k)]]^{r}$ and set 
%\[
%\bL :=  \Lambda[[\Gal(\cK/k)]]e_{r}.
%\] 
%\item[(ii)] We define a Selmer structure $\cF_{\bL}$ on $T_{K,n}$ by the following data: 
%\begin{itemize}
%\item $S(\cF_{\bL}) := S_{K}$,  
%\item we define a local condition at a prime $\fq \not\in S_{p}(k)$ by  
%\begin{align*}
%H^{1}_{\cF_{\bL}}(G_{k_{\fq}}, T_{K,n}) := H^{1}_{\cF_{\rm can}}(G_{k_{\fq}}, T_{K,n}), 
%\end{align*}
%\item we define a local condition $H^{1}_{\cF_{\bL}}(G_{k_{\fp}}, T)$ at a prime $\fp \mid p$ to be the inverse image of $\bL \otimes_{\Lambda[[\Gal(\cK/k)]]} R_{K,n} =: \bL_{K,n}$ under the composite map  
%\[
%H^{1}_{\cF_{\bL}}(G_{k_{\fp}}, T_{K,n}) \longrightarrow H^{1}_{\Sigma}(T_{K,n}) \stackrel{\sim}{\longrightarrow} R_{K,n}^{r}.
%\] 
%\end{itemize} 
%\end{itemize}
%\end{definition}

%\begin{remark}
%By the choice of the fixed isomorphism $H^{1}_{\Sigma}(T_{K,n}) \stackrel{\sim}{\longrightarrow} R_{K,n}^{r}$, the $R_{K,n}$-module $\bigoplus_{\fp \mid p}H^{1}_{\cF_{\bL}}(G_{k_{\fp}}, T_{K,n})$ is isomorphic to $\bL_{K,n}$. 
%Hence we  identify the module  $\bigoplus_{\fp \mid p}H^{1}_{\cF_{\bL}}(G_{k_{\fp}}, T_{K,n})$ with $\bL_{K,n}$. 
%\end{remark}

Let us explain that we have the canonical homomorphism 
\begin{align}\label{hom:stark}
{\rm SS}_{r}(T_{K,n},\cF_{\rm can}) \longrightarrow {\rm SS}_{0}(T_{K,n},\cF_{\rm str}). 
\end{align}

 Let $\fn \in \cN_{n}$ be an ideal. 
 By the definition of the Selmer structures $\cF_{\rm str}$ and $\cF_{\rm can}$ (see Example~\ref{exa:selmer-str}), we have  an exact sequence of $R_{K,n}$-modules 
 \[
 0 \longrightarrow H^{1}_{\cF_{\rm str}^{\fn}}(k, T_{K,n})  \longrightarrow H^{1}_{\cF_{\rm can}^{\fn}}(k, T_{K,n}) \longrightarrow H^{1}_{\Sigma}(T_{K,n}). 
 \]
 Hence, by using   the fixed isomorphism $H^{1}_{\Sigma}(T_{K,n}) \stackrel{\sim}{\longrightarrow} R_{K,n}^{r}$ and Definition~\ref{def:map}, we obtain a homomorphism 
 \begin{align*}
 l_{\fn} \colon {\bigcap}^{r+\nu(\fn)}_{R_{K,n}}H^{1}_{\cF_{\rm can}^{\fn}}(k, T_{K,n}) \longrightarrow  {\bigcap}^{\nu(\fn)}_{R_{K,n}}H^{1}_{\cF_{\rm str}^{\fn}}(k, T_{K,n}).  
\end{align*}
 
%Hence we have  a natural homomorphism 
%\[
%l_{\fn} \colon {\bigcap}^{r+\nu(\fn)}_{R_{K,n}}H^{1}(G_{k,S_{K(\fn)}},T_{K,n}) \longrightarrow {\bigcap}^{\nu(\fn)}_{R_{K,n}}H^{1}_{\rm str}(G_{k,S_{K(\fn)}},T_{K,n}) 
%\]
%for an ideal $\fn \in \cN_{n}$. 

\begin{lemma}
Let $\fn, \fm \in \cN_{n}$ be ideals with $\fn \mid \fm$. Then the diagram 
\[
\xymatrix{
{\bigcap}^{r+\nu(\fm)}_{R_{K,n}}H^{1}_{\cF_{\rm can}^{\fm}}(k, T_{K,n})  \ar[rr]^-{(-1)^{r\nu(\fm)}l_{\fm}} \ar[d]_{\Phi_{\fm,\fn}} &&  {\bigcap}^{\nu(\fn)}_{R_{K,n}}H^{1}_{\cF_{\rm str}^{\fm}}(k, T_{K,n}) \ar[d]_{\Phi_{\fm,\fn}}  
\\
{\bigcap}^{r+\nu(\fn)}_{R_{K,n}}H^{1}_{\cF_{\rm can}^{\fn}}(k, T_{K,n})   \ar[rr]^-{(-1)^{r\nu(\fn)}l_{\fn}} &&  {\bigcap}^{\nu(\fn)}_{R_{K,n}}H^{1}_{\cF_{\rm str}^{\fn}}(k, T_{K,n})  
}
\]
commutes. 
Thus we obtain the canonical homomorphism 
\[
{\rm SS}_{r}(T_{K,n},\cF_{\rm can}) \longrightarrow {\rm SS}_{0}(T_{K,n},\cF_{\rm str}); (\epsilon_{\fn})_{\fn \in \cN_{n}} \mapsto ((-1)^{r\nu(\fn)}l_{\fn}(\epsilon_{\fn}))_{\fn \in \cN_{n}}. 
\]
\end{lemma}
\begin{proof}
We have the following commutative diagrams of $R_{K,n}$-modules: 
\begin{align*}
\xymatrix{
H^{1}_{\cF_{\rm str}^{\fn}}(k, T_{K,n}) \ar@{^{(}->}[r] \ar[d] & H^{1}_{\cF_{\rm str}^{\fm}}(k, T_{K,n}) \ar[d] \ar@{^{(}->}[r] & H^{1}_{\cF_{\rm can}^{\fm}}(k, T_{K,n}) \ar[d] 
\\
W_{\fn}^{*} \ar@{^{(}->}[r] & W_{\fm}^{*} \ar@{^{(}->}[r] & W_{\fm}^{*} \oplus H_{\Sigma}^{1}(T_{K,n}), 
}
\end{align*}
\begin{align*}
\xymatrix{
H^{1}_{\cF_{\rm str}^{\fn}}(k, T_{K,n}) \ar@{^{(}->}[r] \ar[d] & H^{1}_{\cF_{\rm can}^{\fn}}(k, T_{K,n}) \ar[d] \ar@{^{(}->}[r] & H^{1}_{\cF_{\rm can}^{\fm}}(k, T_{K,n}) \ar[d] 
\\
W_{\fn}^{*} \ar@{^{(}->}[r] & W_{\fn}^{*} \oplus H_{\Sigma}^{1}(T_{K,n}) \ar@{^{(}->}[r] & W_{\fm}^{*} \oplus H_{\Sigma}^{1}(T_{K,n}). 
}
\end{align*}
Since any square of the above two diagrams is cartesian, by Proposition~\ref{prop:homlem}, we get a  commutative diagram
\[
\xymatrix@C=13pt{
{\bigcap}^{r+\nu(\fm)}_{R_{K,n}}H^{1}_{\cF_{\rm can}^{\fm}}(k, T_{K,n}) \otimes_{R_{K,n}} \det(W_{\fm} \oplus H^{1}_{\Sigma}(T_{K,n})^{*}) \ar[r] \ar[d] &  {\bigcap}^{\nu(\fn)}_{R_{K,n}}H^{1}_{\cF_{\rm str}^{\fm}}(k, T_{K,n}) \otimes_{R_{K,n}} \det(W_{\fm}) \ar[d] 
\\
{\bigcap}^{r+\nu(\fn)}_{R_{K,n}}H^{1}_{\cF_{\rm can}^{\fn}}(k, T_{K,n})  \otimes_{R_{K,n}} \det(W_{\fn} \oplus H^{1}_{\Sigma}(T_{K,n})^{*})  \ar[r] &  {\bigcap}^{\nu(\fn)}_{R_{K,n}}H^{1}_{\cF_{\rm str}^{\fn}}(k, T_{K,n})  \otimes_{R_{K,n}} \det(W_{\fn}). 
}
\]
Hence this lemma follows from %the normalization in 
Definition~\ref{def:map}. 
\end{proof}

%The existence of the canonical maps 
%\[
%{\rm SS}_{r}(T_{K,n},\cF_{\rm can}) \longrightarrow {\rm SS}_{1}(T_{K,n},\cF_{\bL}) \,  
%\text{ and } \,  
%{\rm SS}_{1}(T_{K,n},\cF_{\bL}) \longrightarrow {\rm SS}_{0}(T_{K,n},\cF_{\rm str})
%\]
%follows from the same argument as that of  ${\rm SS}_{r}(T_{K,n},\cF_{\rm can}) \longrightarrow {\rm SS}_{0}(T_{K,n},\cF_{\rm str})$. 
%Furthermore, since the above two maps between Stark systems are defined by using Definition~\ref{def:map} and the following cartesian commutative diagram such that two squares are cartesian: 
%\[
%\xymatrix{
%H^{1}_{\cF_{\rm str}^{\fn}}(k, T_{K,n}) \ar@{^{(}->}[r] \ar[d] & H^{1}_{\cF_{\bL}^{\fn}}(k, T_{K,n}) \ar[d] \ar@{^{(}->}[r] & H^{1}_{\cF_{\rm can}^{\fn}}(k, T_{K,n}) \ar[d] 
%\\
%W_{\fn}^{*} \ar@{^{(}->}[r] & W_{\fn}^{*} \oplus \bL_{K,n} \ar@{^{(}->}[r] & W_{\fn}^{*} \oplus H_{\Sigma}^{1}(T_{K,n}), 
%}
%\]
%the commutativity of the diagram \eqref{diagram1} follows from Proposition~\ref{prop:homlem}. 

\begin{proposition}\label{prop:free-stark}
Let $\fn \in \cN_{n}$ be an ideal with $H^{1}_{(\cF_{\rm can}^{\fn})^{*}}(k, T_{K,n}^{\vee}(1)) = 0$. 
Then the $R_{K,n}$-module 
$H^{1}_{\cF^{\fn}_{\rm can}}(k,T_{K,n})$ is free of rank $r+\nu(\fn)$ and the projection map 
\[
{\rm SS}_{r}(T_{K,n},\cF_{\rm can}) \longrightarrow 
{\bigcap}^{r+\nu(\fn)}_{R_{K,n}}H^{1}_{\cF^{\fn}_{\rm can}}(k,T_{K,n})
\]
is an isomorphism. In particular, 
the $R_{K,n}$-module ${\rm SS}_{r}(T_{K,n},\cF_{\rm can})$ is free of rank $1$. 
\end{proposition}
\begin{proof}
We put $\cH_{\infty} := \cH (\mu_{p^{\infty}}, (\cO_{k}^{\times})^{1/p^{\infty}})$. Here $(\cO_{k}^{\times})^{1/p^{\infty}} := \bigcup_{n>0}(\cO_{k}^{\times})^{1/p^{n}}$.

Since this proposition follows from \cite[Lemma~4.6~and~Theorem~4.7]{sakamoto} (see also \cite[Theorem~5.2.15]{MRkoly} or \cite[Example~3.20]{sakamoto}), 
we only need to check that the following assumptions that appeared in \cite[\S3~and~\S4]{sakamoto} are satisfied: 
\begin{itemize}
\item[(a)] $(T/\fm T)^{G_{k}} = (T^{\vee}(1)[\fm])^{G_{k}} = 0$ and 
$T/\fm T$ is an irreducible $\cO/\fm[G_{k}]$-module, 
\item[(b)]  There is a $\tau \in \Gal(\overline{k}/\cH_{\infty})$ such that $T_{K,n}/(\tau -1)T_{K,n} \cong R_{K,n}$ as 
$R_{K,n}$-modules, 
\item[(c)] $H^{1}(\Gal(\cH_{\infty} KL/k), T/\fm T) = 
H^{1}(\Gal(\cH_{\infty}KL/k), T^{\vee}(1)[\fm]) = 0$, 
\item[(d)] the Selmer structure $\cF_{\rm can}$ is cartesian (see \cite[Definition~3.8]{sakamoto}). 
\end{itemize}

Since $\chi$ is an even non-trivial character, we have  $\chi  \neq 1, \omega$. 
Here $\omega$ denotes the Teichm\"ullar character. 
Hence Claim~(a) is clear, and $\tau = 1$ satisfies Claim~(b). 

Let us show Claim~(c). Since $p \nmid [L(\mu_{p}) \colon k]$, we have  
\[
H^{1}(\Gal(\cH_{\infty}KL/k), T/\fm T) 
= \Hom(\Gal(\cH_{\infty}KL/L(\mu_{p})), \cO/\fm \otimes \omega\chi^{-1})^{\Gal(L(\mu_{p})/k)}. 
\]
Hence if $H^{1}(\Gal(\cH_{\infty}KL/k), T/\fm T)$ dose not  vanish, there is an abelian extension $N/L(\mu_{p})$ in $\cH_{\infty}KL$ such that 
\begin{itemize}
\item $N/k$ is a Galois extension, 
\item $\Gal(N/L(\mu_{p})) \cong \cO/\fm$ as abelian groups, and  
\item $\Gal(L(\mu_{p})/k)$ acts  on $\Gal(N/L(\mu_{p}))$ via the character $\omega\chi^{-1}$. 
\end{itemize}
Put $A := \cH KL(\mu_{p^{\infty}})$. 
Since $A/k$ is an abelian extension, the action of $\Gal(L(\mu_{p})/k)$ on $\Gal(A/L(\mu_{p}))$ is trivial. 
The fact that $\omega\chi^{-1}$ is not trivial implies that  
\[
N \cap A = L(\mu_{p}). 
\] 
Since $AN/L(\mu_{p})$ is an abelian extension, 
the action of $\Gal(A/k)$ on $\Gal(AN/A)$ factors though $\Gal(L(\mu_{p})/k)$. Hence we obtain an isomorphism of $\Gal(L(\mu_{p})/k)$-modules  
\[
\Gal(AN/A) \stackrel{\sim}{\longrightarrow} \Gal(N/L(\mu_{p})). 
\]
Let $Q$ denote the maximal abelian extension of $L(\mu_{p})$ in $\cH_{\infty}KL$. 
We note that $Q/k$ is a Galois extension and $N \subseteq Q$. 
Since $Q/L(\mu_{p})$ is an abelian extension, $\Gal(L(\mu_{p})/k)$ acts canonically on $\Gal(Q/A)$. 
Hence the canonical surjection  
\[
\Gal(Q/A) \longrightarrow \Gal(AN/A) \stackrel{\sim}{\longrightarrow} \Gal(N/L(\mu_{p}))
\]
is $\Gal(L(\mu_{p})/k)$-equivariant. 
Kummer theory shows that $\Gal(L(\mu_{p})/k)$ acts on $\Gal(Q/A)$ via $\omega$. 
This contradicts the fact that $\omega \neq \omega\chi^{-1}$. 
%\[
%N \cap k^{\dagger}L = L(\mu_{p}). 
%\]
%The Galois group $\Gal(L(\mu_{p})/k)$ also acts on $\Gal(\cH k^{\dagger}KL/k^{\dagger}L)$ since 
%$\cH k^{\dagger}KL/L(\mu_{p})$ is an abelian extension. 
%Furthermore, 
%
%
%Since we have an injection 
%\[
%\Gal(\cH k^{\dagger}KL/k^{\dagger}L) \hooklongrightarrow \Gal(\cH KL(\mu_{p})/L(\mu_{p})), 
%\]
%The action of $\Gal(L(\mu_{p})/k)$ on $\Gal(\cH k^{\dagger}KL/k^{\dagger}L)$ is trivial. 
%
%

The similar argument shows that $H^{1}(\Gal(\cH_{\infty} KL/k)/k, T^{\vee}(1)[\fm])$ also vanishes.  

Let us show Claim~(d). 
By Remark~\ref{rem:vanish}, we only need to check that, for any prime $\fp \in S_{p}(k)$, 
the canonical map $H^{1}(G_{k_{\fp}},\bT_{K}) \longrightarrow H^{1}(G_{k_{\fp}}, T/\fm T) $ is surjective. 
Since we assume the vanishing of $H^{0}(G_{k_{\fp}}, T/\fm T)$  and $H^{2}(G_{k_{\fp}}, T/\fm T)$, this claim follows from Corollary~\ref{cor:H2vanish}. 
%By \cite[Lemma~5.3]{bss2}, the representation $T_{K,n}$ satisfies the hypotheses (H$_{0}$) -- (H$_{5}$) that occurred in \cite[\S3.1.3]{bss2}. 
%Hence this proposition follows from \cite[Theorem~3.14]{bss2} and \cite[Theorem~4.7]{sakamoto}. 
\end{proof}

We show in the proof of Proposition~\ref{prop:free-stark} that \cite[Hypothesis~3.12]{sakamoto} is satisfied. 
Furthermore, $T/\fm T$ and $(T/\fm T)^{\vee}(1)$ have no nonzero isomorphic $\bZ_{p}[G_{k}]$-subquotients. 
Hence the same argument as the proof of \cite[Proposition~3.6.1]{MRkoly} shows the following lemma: 

\begin{lemma}[{\cite[Proposition~3.6.1]{MRkoly}}]\label{lem:che}
Let $m$ be an integer with $n \leq m$. 
For any non-zero elements $c_{1}, c_{2} \in H^{1}(G_{k}, T_{K,n})$ and $c_{3}, c_{4} \in H^{1}(G_{k}, T_{K,n}^{\vee}(1))$, 
there are infinitely many primes $\fq \in \cP_{m}$ such that, for any  $1 \leq i \leq 4$, 
the image of $c_{i}$ under the localization map at $\fq$ is non-zero. 
\end{lemma}

\begin{remark}\label{rem:core}
Let $\cF$ be a Selmer structure on $T_{K,n}$ and $m$ a integer with $n \leq m$. 
By Lemma~\ref{lem:che}, for any integer $n \leq m$, there is an ideal $\fn \in \cN_{m}$ such that 
\[
H^{1}_{(\cF^{\fn})^{*}}(k, T_{K,n}^{\vee}(1)) = \ker\left(H^{1}_{\cF^{*}}(k, T_{K,n}^{\vee}(1)) \longrightarrow \bigoplus_{\fq \mid \fn} H^{1}(G_{k_{\fq}}, T_{K,n}^{\vee}(1))\right) = 0. 
\]
\end{remark}

Let $\fn \in \cN_{n}$ be an ideal with $H^{1}_{(\cF^{\fn}_{\rm str})^{*}}(k, T_{K,n}^{\vee}(1)) = 0$. 
Applying Theorem~\ref{pt} with $\cF_{1} = \cF_{\rm str}^{\fn}$ and $\cF_{1} = \cF_{\rm can}^{\fn}$, 
we get an exact sequence 
\[
0 \longrightarrow  H^{1}_{\cF_{\rm str}^{\fn}}(k, T_{K,n}) \longrightarrow  H^{1}_{\cF_{\rm can}^{\fn}}(k, T_{K,n})  \longrightarrow  H^{1}_{\Sigma}(T_{K,n}) \longrightarrow  0. 
\]
Hence the following proposition is obtained from this exact sequence  and Proposition~\ref{prop:free-stark}. 

\begin{proposition}\label{prop:stark-str}
Let $\fn \in \cN_{n}$ be an ideal with $H^{1}_{(\cF^{\fn}_{\rm str})^{*}}(k, T_{K,n}^{\vee}(1)) = 0$. 
Then the $R_{K,n}$-module $H^{1}_{\cF^{\fm}_{\rm str}}(k,T_{K,n})$ is free of rank $\nu(\fm)$ and the projection map 
\[
{\rm SS}_{0}(T_{K,n},\cF_{\rm str}) \longrightarrow  
{\bigcap}^{\nu(\fm)}_{R_{K,n}}H^{1}_{\cF^{\fm}_{\rm str}}(k,T_{K,n})
\]
is an isomorphism. 
In particular, the $R_{K,n}$-module ${\rm SS}_{0}(T_{K,n},\cF_{\rm can})$ is free of rank $1$ 
and the natural homomorphism
\[
{\rm SS}_{r}(T_{K,n},\cF_{\rm can}) \longrightarrow  {\rm SS}_{0}(T_{K,n},\cF_{\rm str}) 
\]
is an isomorphism. 
\end{proposition}

\begin{remark}\label{rem:core rank}
Proposition~\ref{prop:stark-str} and \cite[Lemma~4.6]{sakamoto} show that the core rank $\chi(\cF_{\rm str})$ of $\cF_{\rm str}$ is $0$ % and $\chi(\cF_{\bL}) = 1$ 
(see \cite[Definition~3.19]{sakamoto} for the definition of the core rank). 
\end{remark}

\begin{definition}
Let $\epsilon = (\epsilon_{\fn})_{\fn \in \cN_{n}} \in {\rm SS}_{0}(T_{K,n},\cF_{\rm str})$ be a system. 
For each $\fn \in \cN_{n}$, we regard $\epsilon_{\fn}$ as an $R_{K,n}$-homomorphism 
\[
\epsilon_{\fn} \colon {\bigwedge}^{\nu(\fn)}_{R_{K,n}} H^{1}_{\cF_{\rm str}^{\fn}}(k, T_{K,n})^{*} \longrightarrow R_{K,n}.  
\]
For an integer $i \geq 0$,  we define an ideal $I_{i}(\epsilon) \subseteq R_{K,n}$ by 
\[
I_{i}(\epsilon) := \sum_{\fn \in \cN_{n}, \, \nu(\fn) = i}\im(\epsilon_{\fn}). 
\]
\end{definition}

\begin{remark}\label{rem:ideal}
Let $\epsilon \in {\rm SS}_{0}(T_{K,n},\cF_{\rm str})$ be a system.  
\begin{itemize}
\item[(i)] By \cite[Theorem~4.10]{sakamoto} or \cite[Theorem~4.6(ii)]{bss}, we have  $I_{i}(\epsilon) \subseteq I_{i+1}(\epsilon)$ for an integer $i \geq 0$, and 
hence 
\begin{align*}
I_{i}(\epsilon) = \sum_{\fn \in \cN_{n}, \, \nu(\fn) \leq i}\im(\epsilon_{\fn}). 
\end{align*}
\item[(ii)] Let $\fm \in \cN_{n}$ be an  ideal such that $H^{1}_{\cF_{\rm str}^{\fm}}(k, T_{K,n})$ is free of rank $\nu(\fm)$. Then, for each integer $0 \leq i \leq \nu(\fm)$, we have  
\[
I_{i}(\epsilon) = \sum_{\fn \mid \fm, \, \nu(\fn) = i}\im(\epsilon_{\fn}) = \sum_{\fn \mid \fm, \, \nu(\fn) \leq i}\im(\epsilon_{\fn}) 
\]
(see the proof of \cite[Theorem~4.10]{sakamoto} or \cite[Corollary~4.5]{bss}). 
\end{itemize}
\end{remark}

Take an ideal $\fm \in \cN_{n+1}$ such that 
\begin{itemize}
\item the $R_{K,n}$-module 
$H^{1}_{\cF^{\fm}_{\rm str}}(k,T_{K,n})$ is free of rank $\nu(\fm)$ and the projection map 
\[
{\rm SS}_{0}(T_{K,n},\cF_{\rm str}) \longrightarrow  
{\bigcap}^{\nu(\fm)}_{R_{K,n}}H^{1}_{\cF^{\fm}_{\rm str}}(k,T_{K,n})
\]
is an isomorphism, and that 
\item the $R_{K, n+1}$-module 
$H^{1}_{\cF^{\fm}_{\rm str}}(k,T_{K,n+1})$ is free of rank $\nu(\fm)$ and the projection map 
\[
{\rm SS}_{0}(T_{K,n+1},\cF_{\rm str}) \longrightarrow  
{\bigcap}^{\nu(\fm)}_{R_{K,n+1}}H^{1}_{\cF^{\fm}_{\rm str}}(k,T_{K,n+1})
\]
is an isomorphism.
\end{itemize}
Since the $R_{K, n+1}$-module 
$H^{1}_{\cF^{\fm}_{\rm str}}(k,T_{K,n+1})$ is free of rank $\nu(\fm)$, the canonical homomorphism 
\[
{\bigwedge}^{\nu(\fm)}_{R_{K,n+1}}H^{1}_{\cF^{\fm}_{\rm str}}(k,T_{K,n+1}) \longrightarrow {\bigcap}^{\nu(\fm)}_{R_{K,n+1}}H^{1}_{\cF^{\fm}_{\rm str}}(k,T_{K,n+1})
\]
is an isomorphism. 
Hence we obtain an $R_{K,n+1}$-homomorphism 
\[
{\bigcap}^{\nu(\fm)}_{R_{K,n+1}}H^{1}_{\cF^{\fm}_{\rm str}}(k,T_{K,n+1}) \longrightarrow {\bigcap}^{\nu(\fm)}_{R_{K,n}}H^{1}_{\cF^{\fm}_{\rm str}}(k,T_{K,n}), 
\]
which induces an $R_{K,n+1}$-homomorphism  
\[
{\rm SS}_{0}(T_{K,n+1},\cF_{\rm str}) \longrightarrow  {\rm SS}_{0}(T_{K,n},\cF_{\rm str}). 
\]
This homomorphism is independent of the choice of the ideal $\fm \in \cN_{n+1}$. 

As in \cite[\S5]{sakamoto}, we define 
\[
{\rm SS}_{0}(\bT_{K},\cF_{\rm str}) := \varprojlim_{n>0}{\rm SS}_{0}(T_{K,n},\cF_{\rm str}). 
\]
For a system $\epsilon = (\epsilon_{n}) \in {\rm SS}_{0}(\bT_{K},\cF_{\rm str})$ and an integer $i \geq 0$, we also set 
\[
I_{i}(\epsilon) := \varprojlim_{n>0}I_{i}(\epsilon_{n}) \subseteq \Lambda_{K}. 
\]
We note that we have the canonical projection 
\[
{\rm pr}_{K} \colon {\rm SS}_{0}(\bT_{K},\cF_{\rm str}) \longrightarrow   \varprojlim_{n>0}{\bigcap}^{0}_{R_{K,n}}H^{1}_{\rm str}(G_{k,S_{K}},T_{K,n}) = \varprojlim_{n>0}R_{K,n} = \Lambda_{K}. 
\]

\begin{remark}
By definition, we have  $I_{0}(\epsilon) = {\rm pr}_{K}(\epsilon) \cdot \Lambda_{K}$ for a system $\epsilon \in {\rm SS}_{0}(\bT_{K},\cF_{\rm str})$. 
\end{remark}

\begin{theorem}[{\cite[Theorems~5.4 and 5.6]{sakamoto}}]\label{thm:stark}\
\begin{itemize}
\item[(i)] The $\Lambda_{K}$-module ${\rm SS}_{0}(\bT_{K},\cF_{\rm str})$ is free of rank $1$. 
\item[(ii)] If $\epsilon \in {\rm SS}_{0}(\bT_{K},\cF_{\rm str})$ is a basis, then we have  
\[
I_{i}(\epsilon) = \Fitt_{\Lambda_{K}}^{i}(X_{K}^{\chi}).
\] 
In particular, we have  ${\rm pr}_{K}(\epsilon) \cdot \Lambda_{K} = \Fitt_{\Lambda_{K}}^{0}(X_{K}^{\chi})$. 
\end{itemize}
\end{theorem}

\subsection{Kolyvagin derivative homomorphism}

Let $n>0$ be an integer. 
In this subsection, we recall the definition of the Kolyvagin derivative homomorphism 
\[
\cD_{r',n} \colon {\rm ES}_{r'}(T)  \longrightarrow  
\prod_{\fn \in \cN_{n}}{\bigcap}^{r'}_{R_{K,n}}H^{1}(G_{k,S_{K(\fn)}},T_{K,n}). 
%\ \ \text{and} \ \ 
%\cD_{0,n} \colon \cE(T) \longrightarrow  \prod_{\fn \in \cN_{n}}R_{K,n}. 
\]
for an integer $r' \geq 0$. 
We note that $S_{K(\fn)} = S_{K} \cup \{\fq \mid \fn\}$ and 
\[
H^{1}(G_{k,S_{K(\fn)}},T_{K,n}) = H^{1}_{\cF_{\rm can}^{\fn}}(k,T_{K,n})
\] 
for an ideal $\fn \in \cN_{n}$. 

Let $c = (c_{K'})_{K' \in \Omega} \in \cE(T)$ be an Euler system of rank $0$. Set 
\[
K_{n} := k_{n}K \, \text{ and } \, K_{n}(\fn) = k(\fn)K_{n}
\]
for an ideal $\fn \in \cN_{n}$. 
We note that the field $K(\fn)$ is contained in $\Omega$ and that there are canonical identifications 
\[
\Gal(K_{n}(\fq)/K_{n}) = \Gal(k_{n}(\fq)/k_{n}) =: G_{\fq} \ \text{ and } \ \Gal(K_{n}(\fn)/K_{n}) = \prod_{\fq \mid \fn} G_{\fq}. 
\]
Let $\fq \in \cP_{n}$ be a prime of $k$. 
Since $p$ is coprime to the class number of $k$ and $\fq \nmid p$,  the field extension $k(\fq)/k$ is cyclic, and hence $G_{\fq}$ is cyclic. 
We fix a generator $\sigma_{\fq}$ of $G_{\fq}$ and denote by $D_{\fq} \in \bZ[G_{\fq}]$ the Kolyvagin's derivative operator: 
\[
D_{\fq} := \sum_{i=0}^{\#G_{\fq} -1}i \sigma_{\fq}^{i}. 
\]
We also set $D_{\fn} := \prod_{\fq \mid \fn}D_{\fq} \in \bZ[\Gal(K_{n}(\fn)/K_{n})]$. 
Let 
\[
\pi_{K(\fn),n} \colon \Lambda_{K(\fn)} \longrightarrow  \cO/p^{n}\cO[\Gal(K_{n}(\fn)/k)] = R_{K,n}[\Gal(K_{n}(\fn)/K_{n})]
\] 
be the canonical projection. 
Then it is well-known that Euler system relations imply  
\[
\kappa(c)_{K, n, \fn} := \pi_{K(\fn),n}(D_{\fn}c_{K(\fn)}) \in R_{K,n}[\Gal(K(\fn)/K)]^{\Gal(K(\fn)/K)} \stackrel{\sim}{\longleftarrow} R_{K,n}
\]
(see, for example, \cite[Lemma~4.4.2]{R}).

\begin{remark}\label{rem:leading}
Let $I_{\fn} := \ker(R_{K,n}[\Gal(K_{n}(\fn)/K_{n})] \longrightarrow  R_{K,n})$ denote the augmentation ideal. 
The proof of \cite[Lemma~4.4]{Kur03} shows that $\pi_{K(\fn),n}(c_{K(\fn)}) \in I_{\fn}^{\nu(\fn)}$ and 
\[
\pi_{K(\fn),n}(c_{K(\fn)}) \equiv (-1)^{\nu(\fn)}\kappa(c)_{K, n, \fn} \cdot \prod_{\fq \mid \fn}(\sigma_{\fq}-1) 
\quad \bmod I_{\fn}^{\nu(\fn)+1}. 
\] 
Hence one can regard $(-1)^{\nu(\fn)}\kappa(c)_{K, n,\fn} \in R_{K,n}$ as the ``leading coefficient'' of $\pi_{K(\fn),n}(c_{K(\fn)})$. 
\end{remark}

\begin{definition}
Let $\fS(\fn)$ denote the set of permutations of prime divisors of $\fn$. 
For a permutation $\tau \in \fS(\fn)$, we set $\fd_{\tau} = \prod_{\tau(\fq) = \fq}\fq$. 
By using an isomorphism $\bZ/(\#G_{\fq}) = G_{\fq} \stackrel{\sim}{\longrightarrow} I_{\fq}/I_{\fq}^{2}$, we get an element $a_{\tau,\fq} \in \bZ/(\#G_{\fq})$ which corresponds to the element $P_{\tau(\fq)}({\rm Frob}_{\tau(\fq)}^{-1})$ for a prime $\fq \mid \fn/\fd_{\tau}$. 
Then we define an element $\tilde{\kappa}(c)_{n,\fn}$ by  
\[
\tilde{\kappa}(c)_{K,n,\fn} := \sum_{\tau \in \fS(\fn)} \sgn(\tau) \left(\prod_{\fq \mid \frac{\fn}{\fd_{\tau}}}a_{\tau,\fq} \right) \kappa(c)_{K, n,\fd_{\tau}}. 
\]
\end{definition}

\begin{definition}
We call the map  
\[
\cD_{0,n} \colon \cE(T) \longrightarrow  \prod_{\fn \in \cN_{n}}R_{K,n}; c \mapsto (\tilde{\kappa}(c)_{k,n,\fn})_{\fn \in \cN_{n}}
\]
a Kolyvagin derivative homomorphism. 
\end{definition}

Take a field $K' \in \Omega$ containing $K$. 
By Lemma~\ref{lem:reduction}, we have the canonical homomorphism 
\[
{\bigcap}^{r'}_{\Lambda_{K'}}H^{1}(G_{k,S_{K'}}, \bT_{K'}) \longrightarrow {\bigcap}^{r'}_{R_{K', n}}H^{1}(G_{k,S_{K'}}, T_{K', n}). 
\]
Furthermore, by \cite[Proposition~A.4]{sbA}, we have the canonical isomorphism 
\begin{align*}
\left( {\bigcap}^{r'}_{R_{K', n}}H^{1}(G_{k,S_{K'}}, T_{K', n}) \right)^{\Gal(K'/K)} 
&\cong {\bigcap}^{r'}_{R_{K, n}}H^{1}(G_{k,S_{K'}}, T_{K', n})^{\Gal(K'/K)}
\\
&= {\bigcap}^{r'}_{R_{K, n}}H^{1}(G_{k,S_{K'}}, T_{K, n}). 
\end{align*}
Hence, in the same manner, one can also define a higher Kolyvagin derivative homomorphism  
\[
\cD_{r', n} \colon {\rm ES}_{r'}(T) \longrightarrow  
\prod_{\fn \in \cN_{n}}{\bigcap}^{r'}_{R_{K,n}}H^{1}_{\cF_{\rm can}^{\fn}}(k,T_{K,n})
\]
for any integer $r' \geq 0$ (see \cite[\S4.3]{sbA} and \cite[\S\S6.3 and 6.4]{bss} for the detail).

%The proof of this proposition is based on the method of \cite{Bu, PR, rubinstark}. 

By using the localization map at $p$ and the fixed isomorphism $H^{1}_{\Sigma}(T_{K,n}) \stackrel{\sim}{\longrightarrow} R_{K,n}^{r}$, we obtain a homomorphism 
\[
H^{1}_{\cF^{\fn}_{\rm can}}(k, T_{K,n}) \longrightarrow 
%\Lambda[[\Gal(\cK/k)]]^{r}/\bL \otimes_{\Lambda[[\Gal(\cK/k)]]} \Lambda_{K'}  \cong  
R_{K,n}^{r}
\]
for each ideal $\fn \in \cN_{n}$. 
%By \eqref{map bidual}, this map induces a homomorphisim 
%\[
%\varphi_{\bL, K'} \colon {\bigcap}^{r}_{\Lambda_{K'}}H^{1}(G_{k,S_{K'}}, \bT_{K'}) \longrightarrow H^{1}(G_{k,S_{K'}}, \bT_{K'}). 
%\]
%By the argument in \cite[\S 1.2.3]{PR} (see also \cite[\S~6]{rubinstark}), the maps $\{\varphi_{\bL, K'}\}_{K' \in \Omega}$ induce a homomorphisim 
%\[
%{\rm ES}_{r}(T) \longrightarrow {\rm ES}_{1}(T). 
%\] 
%By using the $r$-th projection $H^{1}_{\Sigma}(\bT_{K'}) \stackrel{\sim}{\longrightarrow} \Lambda_{K'}^{r} \longrightarrow \Lambda_{K'}$, 
%we obtain a homomorphism 
%\[
%H^{1}(G_{k,S_{K'}}, \bT_{K'}) \longrightarrow 
%%\Lambda[[\Gal(\cK/k)]]^{r}/\bL \otimes_{\Lambda[[\Gal(\cK/k)]]} \Lambda_{K'}  \cong  
%\Lambda_{K'}
%\]
%for each field $K' \in \Omega$. 
%Hence we also get a homomorphism 
%\[
%{\rm ES}_{1}(T) \longrightarrow \cE(T). 
%\]
%We note that the composite map  ${\rm ES}_{r}(T) \longrightarrow {\rm ES}_{1}(T)  \longrightarrow \cE(T)$ 
%coincides with the map defined in Proposition~\ref{prop:main}. 
Hence by \eqref{map bidual},this map induces a homomorphisim 
\begin{align*}
\prod_{\fn \in \cN_{n}}{\bigcap}^{r}_{R_{K,n}}H^{1}_{\cF_{\rm can}^{\fn}}(k,T_{K,n}) \longrightarrow 
R_{K,n}. 
\end{align*}

Let ${\rm ES}_{r}(T) = {\rm ES}_{f,r}(T) \hooklongrightarrow \cE(T)$ denote the homomorphism defined in Proposition~\ref{prop:main}. 
The following proposition follows from the construction of the Kolyvagin derivative homomorphisms 
$\cD_{0,n}$  and $\cD_{r,n}$. 

\begin{proposition}\label{prop:diag}
The following diagram commutes: 
\[
\xymatrix{
{\rm ES}_{r}(T) \ar[rr]^-{\cD_{r,n}} \ar@{^{(}->}[d] && \prod_{\fn \in \cN_{n}}{\bigcap}^{r}_{R_{K,n}}H^{1}_{\cF_{\rm can}^{\fn}}(k,T_{K,n}) \ar[d] 
%\\
%{\rm ES}_{1}(T) \ar[rr]^-{\cD_{1,n}} \ar[d] && \prod_{\fn \in \cN_{n}}H^{1}_{\cF_{\rm can}^{\fn}}(k,T_{K,n}) \ar[d] 
\\
\cE(T) \ar[rr]^-{\cD_{0,n}} && \prod_{\fn \in \cN_{n}}R_{K,n}. 
}  
\]
%Here the left vertical map is defined in Lemma~\ref{lem:str} and the right vertical map is obtained by the fixed isomorphism $\varprojlim_{K \in \Omega}H^{1}_{\Sigma}(\bT_{K}) \stackrel{\sim}{\longrightarrow} \Lambda[[\Gal(\cK/k)]]^{r}$. 
\end{proposition}

\subsection{Regulator homomorphism}
Let $r' \geq 0$ be an integer and $\cF$ a Selmer structure on $T_{K,n}$. 
In this subsection, we recall the definition of regulator homomorphisms
\begin{align*}
{\rm Reg}_{r'} \colon {\rm SS}_{r'}(T_{K,n},\cF) \longrightarrow  \prod_{\fn \in \cN_{n}}H^{1}_{\cF_{\rm can}^{\fn}}(k,T_{K,n}) 
\end{align*}
defined in \cite{sbA, MRselmer}, and construct a Stark system $\epsilon^{\rm DR}_{K} \in {\rm SS}_{0}(\bT_{K},\cF_{\rm str})$ from the $p$-adic $L$-functions $\{L_{p,K'}^{\chi}\}_{K'\in \Omega}$ by using the higher rank Euler system $c^{\rm DR} \in {\rm ES}_{r}(T)$ constructed in Theorem~\ref{thm:main2}.

%For each prime $\fq \in \cP_{n}$, we fix an isomorphism 
%\[
%H^{1}_{/f}(G_{k_{\fq}},T_{K,n}) 
%\stackrel{\sim}{\longrightarrow} R_{K,n}. 
%\]
%Recall that $\sigma_{\fq}$ denotes the fixed generator of $G_{\fq}$. 
%Furthermore, we identify $G_{\fq} = \Gal(k(\fq)/k)$ with $\bZ/(\#G_{\fq})$ by using the fixed generator $\sigma_{\fq} \in G_{\fq}$. 

\begin{lemma}[{\cite[Lemmas~1.2.3 and 1.2.4]{MRkoly}}]\label{lem:tr}
Let $\fq \in \cP_{n}$. 
We set 
\[
H^{1}_{\rm tr}(G_{k_{\fq}},T_{K,n}) := \ker\left(H^{1}(G_{k_{\fq}},T_{K,n}) \longrightarrow H^{1}(G_{k(\fq)_{\fq}},T_{K,n})\right), 
\]
where $k(\fq)_{\fq}$ is the closure of $k(\fq)$ in $\overline{k}_{\fq}$. Put $H^{1}_{/{\rm tr}}(G_{k_{\fq}},T_{K,n}) := H^{1}(G_{k_{\fq}},T_{K,n})/H^{1}_{\rm tr}(G_{k_{\fq}},T_{K,n})$. 
\begin{itemize}
\item[(i)] The composite maps 
\[
H^{1}_{\rm tr}(G_{k_{\fq}},T_{K,n}) \longrightarrow H^{1}(G_{k_{\fq}},T_{K,n}) \longrightarrow H^{1}_{/f}(G_{k_{\fq}},T_{K,n}) 
\] 
and 
\[
H^{1}_{f}(G_{k_{\fq}},T_{K,n}) \longrightarrow H^{1}(G_{k_{\fq}},T_{K,n}) \longrightarrow H^{1}_{/{\rm tr}}(G_{k_{\fq}},T_{K,n}) 
\]
are isomorphisms. 
\item[(ii)] The Frobenius ${\rm Frob}_{\fq}$ induces an isomorphism 
\[
H^{1}_{f}(G_{k_{\fq}},T_{K,n}) \stackrel{\sim}{\longrightarrow} T_{K,n}. 
\]
\item[(iii)] The fixed generator $\sigma_{\fq} \in G_{\fq}$ induces an isomorphism 
\[
H^{1}_{/f}(G_{k_{\fq}},T_{K,n}) \stackrel{\sim}{\longrightarrow} T_{K,n}. 
\]
\end{itemize}
\end{lemma}

We fix a basis $e$ of $\bT_{K}$. 
By using the isomorphisms in Lemma~\ref{lem:tr} and the fixed basis $e$ of $\bT_{K}$, 
we identify $H^{1}_{f}(G_{k_{\fq}},T_{K,n})$, $H^{1}_{/f}(G_{k_{\fq}},T_{K,n})$, $H^{1}_{\rm tr}(G_{k_{\fq}},T_{K,n})$, and $H^{1}_{/{\rm tr}}(G_{k_{\fq}},T_{K,n})$ with $R_{K,n}$.  
We note that these identifications are compatible with $n$, i.e., for $* \in \{f, {\rm tr}, /f, /{\rm tr}\}$, the following diagram commutes: 
\[
\xymatrix{
H^{1}_{*}(G_{k_{\fq}}, T_{K,n}) \ar[r]^-{=} \ar[d] & R_{K,n} \ar[d] 
\\
H^{1}_{*}(G_{k_{\fq}}, T_{K,n-1}) \ar[r]^-{=} & R_{K,n-1}. 
} 
\]

\begin{definition}\label{def:varphi}
For an ideal $\fn \in \cN_{n}$ and a prime $\fq \in \cP_{n}$ with $\fq \mid \fn$, we put 
\[
\varphi_{\fq} \colon H^{1}_{\cF_{\rm can}^{\fn}}(k,T_{K,n}) \longrightarrow H^{1}_{/{\rm tr}}(G_{k_{\fq}},T_{K,n}) = R_{K,n}. 
\]
\end{definition}

\begin{definition}
Let $\fn \in \cN_{n}$ be an ideal and $\cF$ a Selmer structure on $T_{K,n}$. 
Then we define a  Selmer structure $\cF(\fn)$ on $T_{K,n}$ by the following data:
\begin{itemize}
\item $S(\cF(\fn)) := S(\cF) \cup \{\fq \mid \fn\}$,  
\item we define a local condition at a prime $\fq \mid \fn$ by  
\begin{align*}
H^{1}_{\cF(\fn)}(G_{k_{\fq}}, T_{K,n}) := H^{1}_{\rm tr}(G_{k_{\fq}}, T_{K,n}), 
\end{align*}
\item we define a local condition at a prime $\fq \nmid \fn$ by  
\[
H^{1}_{\cF(\fn)}(G_{k_{\fq}}, T) := H^{1}_{\cF}(G_{k_{\fq}}, T).
\] 
\end{itemize}
\end{definition}

\begin{remark}\label{rem:fs map}
By \cite[Proposition~14.3]{MRselmer}, for any integer $m$ with $n \leq m$ and ideal  $\fn \in \cN_{m}$, there are infinitely many ideals $\fm \in \cN_{m}$ with $\fn \mid \fm$ such that $H^{1}_{\cF_{\rm str}(\fm)^{*}}(k,(T/\fm T)^{\vee}(1))$ vanishes. 
Then, by \cite[Corollary~3.21]{sakamoto}, the module $H^{1}_{\cF_{\rm str}(\fm)}(k,T/\fm T)$ also vanishes since the core rank of $\chi(\cF_{\rm str})$ is zero by Remark~\ref{rem:core rank}. 
By \cite[Lemmas~3.13~and~3.14, and Corollary~3.18]{sakamoto} (see also \cite[Lemma~3.5.3]{MRkoly} and \cite[\S~3.2]{bss}),  we see that 
\[
H^{1}_{\cF_{\rm str}(\fm)}(k,T_{K,n}) = H^{1}_{\cF_{\rm str}(\fm)^{*}}(k,T_{K,n}^{\vee}(1)) = 0. 
\]
Applying Theorem~\ref{pt} with $\cF_{1} = \cF_{\rm str}(\fm)$ and $\cF_{2} = \cF_{\rm str}^{\fm}$, 
we conclude that the homomorphism 
\[
H^{1}_{\cF_{\rm str}^{\fm}}(k, T_{K,n}) \xrightarrow{\bigoplus_{\fq \mid \fm} \varphi_{\fq}} \bigoplus_{\fq \mid \fm}R_{K,n}
\]
is an isomorphism. 
Hence $H^{1}_{\cF_{\rm str}^{\fm}}(k, T_{K,n})^{*}$ is generated by the set $\{\varphi_{\fq}\}_{\fq \mid \fm}$ and, for any ideal $\fm' \mid \fm$ and system $\epsilon \in {\rm SS}_{0}(T_{K,n},\cF_{\rm str})$, the ideal $\im(\epsilon_{\fm'})$ is generated by the set 
\[
\{\epsilon_{\fm'}(\wedge_{\fq \mid \fn} \varphi_{\fq}) \mid \nu(\fn) = \nu(\fm'), \,  \fn \mid \fm\}. 
\]

We also note that, applying Theorem~\ref{pt} with $\cF_{1} = \cF_{\rm str}(\fm)$ and $\cF_{2} = \cF_{\bL}(\fm)$, we see that $H^{1}_{\cF_{\bL}(\fm)}(k,T_{K,n})$ is a free $R_{K,n}$-module of rank $1$ and 
$H^{1}_{\cF_{\bL}(\fm)^{*}}(k,T_{K,n}^{\vee}(1))$ vanishes. 
\end{remark}

\begin{definition}[{\cite[Definition~3.1.3]{MRkoly} and \cite[Definition~4.1]{sbA}}]
Let $\cF$ be a Selmer structure on $T_{K, n}$ and $r' > 0$ an integer. 
Let  $\fn \in \cN_{n}$ be an ideal and $\fq$ a prime divisor of $\fn$. 
Then we have two cartesian diagrams
\[
\xymatrix{
H^{1}_{\cF_{\fq}(\fn)}(k,T_{K, n}) \ar[d] \ar@{^{(}->}[r] &  H^{1}_{\cF(\fn)}(k,T_{K, n}) \ar[d]
\\
0 \ar[r] & H^{1}_{\rm tr}(G_{k_{\fq}}, T_{K,n}) = R_{K,n}, 
}
\]
\[
\xymatrix{
H^{1}_{\cF_{\fq}(\fn)}(k,T_{K, n}) \ar[d] \ar@{^{(}->}[r] &  H^{1}_{\cF(\fn/\fq)}(k,T_{K, n}) \ar[d]
\\
0 \ar[r] & H^{1}_{f}(G_{k_{\fq}}, T_{K,n}) = R_{K,n}. 
}
\]
Here we denote by $H^{1}_{\cF_{\fq}(\fn)}(k,T_{K, n})$ the kernel of the map 
$H^{1}_{\cF(\fn)}(k,T_{K, n}) \longrightarrow  H^{1}_{\rm tr}(G_{k_{\fq}}, T_{K,n})$. 
Hence by Definition~\ref{def:map}, we obtain maps 
\[
\varphi_{\fq} \colon {\bigcap}^{r'}_{R_{K,n}}H^{1}_{\cF(\fn)}(k,T_{K, n}) \longrightarrow {\bigcap}^{r'-1}_{R_{K,n}}H^{1}_{\cF_{\fq}(\fn)}(k,T_{K, n})
\]
and 
\[
{\rm div}_{\fq} \colon {\bigcap}^{r'}_{R_{K,n}}H^{1}_{\cF(\fn/\fq)}(k,T_{K, n}) \longrightarrow {\bigcap}^{r'-1}_{R_{K,n}}H^{1}_{\cF_{\fq}(\fn)}(k,T_{K, n}). 
\]

We say that a family of elements $\{\kappa_{\fn}\}_{\fn \in \cN_{n}} \in \prod_{\fn \in \cN_{n}}{\bigcap}^{r'}_{R_{K,n}}H^{1}_{\cF(\fn)}(k,T_{K, n})$ is a Kolyvagin system of rank $r'$ if $\{\kappa_{\fn}\}_{\fn \in \cN_{n}}$ satisfies 
\[
\varphi_{\fq}(\kappa_{\fn}) = {\rm div}_{\fq}(\kappa_{\fn/\fq})
\]
for any ideal $\fn \in \cN_{n}$ and prime divisor $\fq$ of $\fn$. 
We denote by ${\rm KS}_{r'}(T_{K,n}, \cF)$ the module of Kolyvagin systems of rank $r'$. 
\end{definition}

Let $\cF$ be a Selmer structure on $T_{K,n}$ and $r' \geq 0$ an integer. 
By definition, we have  the following cartesian diagram
\[
\xymatrix{
H^{1}_{\cF(\fn)}(k, T_{K,n}) \ar@{^{(}->}[r] \ar[d] & H^{1}_{\cF^{\fn}}(k, T_{K,n})  \ar[d]^-{\bigoplus_{\fq \mid \fn}\varphi_{\fq}}
\\
0 \ar@{^{(}->}[r] &  \bigoplus_{\fq \mid \fn}H^{1}_{/{\rm tr}}(G_{k_{\fq}}, T_{K,n}) = R_{K,n}^{\nu(\fn)}.     
}
\]
%\[
%\xymatrix{
%H^{1}_{\cF_{\rm str}(\fn)}(k, T_{K,n}) \ar@{^{(}->}[r] \ar[d] & H^{1}_{\cF_{\rm str}^{\fn}}(k, T_{K,n})  \ar[d]
%\\
%0 \ar@{^{(}->}[r] &  \bigoplus_{\fq \mid \fn}H^{1}_{/{\rm tr}}(G_{k_{\fq}}, T_{K,n}) = R_{K,n}^{\nu(\fn)}. 
%}
%\]
Hence by applying Definition~\ref{def:map} to the above cartesian diagram, 
we see that 
the homomorphisms $\{\varphi_{\fq}\}_{\fq \mid \fn}$ induce the following homomorphism 
defined by $\Phi \mapsto (-1)^{\nu(\fn)}\Phi(\wedge_{\fq \mid \fn} \varphi_{\fq})$: 
\begin{align*}
&\Reg_{\fn} \colon {\bigcap}^{r'+\nu(\fn)}_{R_{K,n}}H^{1}_{\cF^{\fn}}(k,T_{K,n}) \longrightarrow {\bigcap}^{r'}_{R_{K,n}}H^{1}_{\cF(\fn)}(k,T_{K,n}). 
%\\
%&\Reg_{\fn} \colon {\bigcap}^{\nu(\fn)}_{R_{K,n}}H^{1}_{\cF_{\rm str}^{\fn}}(k,T_{K,n}) \longrightarrow  R_{K,n}.  
\end{align*}
The regulator homomorphism ${\rm Reg}_{r'}$ is defined by $(\epsilon_{\fn})_{\fn \in \cN_{n}} \mapsto (\Reg_{\fn}(\epsilon_{\fn}))_{\fn \in \cN_{n}}$:  
\begin{align*}
{\rm Reg}_{r'} \colon {\rm SS}_{r'}(T_{K,n},\cF) \longrightarrow \prod_{\fn \in \cN_{n}}{\bigcap}^{r'}_{R_{K,n}}H^{1}_{\cF(\fn)}(k, T_{K,n}). 
%\\
%&{\rm Reg}_{0} \colon {\rm SS}_{0}(T_{K,n},\cF_{\rm str}) \longrightarrow \prod_{\fn \in \cN_{n}}R_{K,n}. 
\end{align*}

\begin{proposition}[{\cite[Proposition~4.3]{sbA}}]
Let $\cF$ be a Selmer structure. 
For an integer $r' > 0$, the image of the regulator homomorphism ${\rm Reg}_{r'} \colon {\rm SS}_{r'}(T_{K,n},\cF) \longrightarrow \prod_{\fn \in \cN_{n}}{\bigcap}^{r'}_{R_{K,n}}H^{1}_{\cF(\fn)}(k, T_{K,n})$ is contained in  ${\rm KS}_{r'}(T_{K,n},\cF)$. 
\end{proposition}

\begin{lemma}\label{lem:inj}\ 
The regulator homomorphism
\[
{\rm Reg}_{0} \colon   {\rm SS}_{0}(T_{K,n},\cF_{\rm str}) \longrightarrow \prod_{\fn \in \cN_{n}}R_{K,n} 
\]
is injective.  
\end{lemma}
\begin{proof}
By Proposition~\ref{prop:stark-str} and Remark~\ref{rem:fs map}, one can take an ideal $\fm \in \cN_{n}$ such that 
\begin{itemize}
\item the projection ${\rm SS}_{0}(T_{K,n},\cF_{\rm str})  \longrightarrow {\bigcap}^{\nu(\fm)}_{R_{K,n}}H^{1}_{\cF_{\rm str}^{\fm}}(k, T_{K,n})$ is an isomorphism, and
\item  the homomorphism 
\[
H^{1}_{\cF_{\rm str}^{\fm}}(k, T_{K,n}) \xrightarrow{\bigoplus_{\fq \mid \fm} \varphi_{\fq}} \bigoplus_{\fq \mid \fm}R_{K,n}
\] 
is an isomorphism. 
\end{itemize}
Then we have a commutative diagram
\[
\xymatrix{
{\rm SS}_{0}(T_{K,n},\cF_{\rm str})  \ar[r]^-{\cong} \ar[d]^-{{\rm Reg}_{0}} & {\bigcap}^{\nu(\fm)}_{R_{K,n}}H^{1}_{\cF_{\rm str}^{\fm}}(k, T_{K,n}) \ar[d]^-{{\rm Reg}_{\fm}} 
\\
\prod_{\fn \in \cN_{n}}R_{K,n}  \ar[r] & R_{K,n}. 
}
\]
Since $\bigoplus_{\fq \mid \fm} \varphi_{\fq} \colon H^{1}_{\cF_{\rm str}^{\fm}}(k, T_{K,n}) \longrightarrow \bigoplus_{\fq \mid \fm}R_{K,n}$ is an isomorphism, 
the homomorphism ${\rm Reg}_{\fm}$ is an isomorphism, and hence ${\rm Reg}_{0}$ is injective. 
\end{proof}

%where the right vertical arrow is induced by 
%the fixed isomorphism $H^{1}_{\Sigma}(T_{K,n})\stackrel{\sim}{\longrightarrow} 
%R_{K,n}^{r}$ (we do not multiply the sign). 

%{\color{red}The following proposition is proved by Buyukboduk in \cite{}}. 

\begin{theorem}[{\cite[Theorem~5.2~(i)]{bss}}]\label{thm:isom-koly-stark}
The regulator homomorphism
\[
{\rm Reg}_{r} \colon   {\rm SS}_{0}(T_{K,n},\cF_{\rm can}) \longrightarrow {\rm KS}_{r}(T_{K,n},\cF_{\rm can})
\]
is an isomorphism.  
\end{theorem}
\begin{proof}
This theorem follows from \cite[Theorem~5.2~(i)]{bss}, and hence we only need to check that 
Hypotheses~3.2,~3.3~and~4.2 in \cite{bss} are satisfied. 
In the proof of Proposition~\ref{prop:free-stark}, we prove that Hypotheses~3.2~and~3.3 in \cite{bss} are satisfied. 
By Proposition~\ref{prop:free-stark}, Hypotheses~4.2 is satisfied. 
\end{proof}

When $r'=1$, the following theorem is proved by Mazur--Rubin in \cite[Theorem~3.2.4]{MRkoly}. 
When $r' > 1$, it is proved by Burns--Sano and the author in \cite[Theorem~6.12]{bss}. 

\begin{theorem}\label{thm:derivative}
For any integer $r'>0$, 
the homomorphism $\cD_{r', n}$ induces a homomorphism 
\[
\cD_{r', n} \colon {\rm ES}_{r}(T) \longrightarrow  {\rm KS}_{r'}(T_{K,n}, \cF_{\rm can}). 
\]
\end{theorem}

%\begin{remark}
%In \cite{bss}, Burns--Sano and the author proved that $\im(\cD_{r,n}) \subseteq \im({\rm Reg}_{r})$ without such technical assumptions. 
%\end{remark}

%\begin{proof}
%Let $c = \{c_{K}\}_{K \in \Omega} \in {\rm ES}_{r}(T)$ be an Euler system of rank $r$ and 
%let $c_{\bL}$ denote the image of $c$ under the map ${\rm ES}_{r}(T) \longrightarrow {\rm ES}_{1}(T)$. 
%
%By \cite[Theorem~3.2.a4]{MRkoly}, we have  $\cD_{1,n}(c_{\bL})_{\fn} \in {\rm KS}_{1}(T_{K,n}, \cF_{\rm can})$. 
%Furthermore, the same argument as the proof of \cite[Proposition~3.27]{Bu} implies 
%\[
%\cD_{1,n}(c_{\bL}) \in {\rm KS}_{1}(T_{K,n}, \cF_{\bL}). 
%\]
%Hence this proposition follows from Lemma~\ref{lem:inj}(ii). 
%\end{proof}

By Proposition~\ref{prop:diag}, we have the following commuative diagram: 
\begin{align}\label{diagram2}
\begin{split}
\xymatrix{
{\rm ES}_{r}(T) \ar[rr]^-{\cD_{r,n}} \ar@{^{(}->}[d] && 
{\rm KS}_{r}(T_{K,n}, \cF_{\rm can}) \ar[d] &&
\ar[ll]^-{{\rm Reg}_{r}}_{\cong} \ar[d]^{\cong}  {\rm SS}_{r}(T_{K,n}, \cF_{\rm can})
\\
\cE(T) \ar[rr]^-{\cD_{0,n}} && \prod_{\fn \in \cN_{n}}R_{K,n} && \ar@{_{(}->}[ll]_-{{\rm Reg}_{0}} {\rm SS}_{0}(T_{K,n}, \cF_{\rm str}). 
}
\end{split}
\end{align}

To simplify the notation, we put $\cL_{p}^{\chi} := \{L_{p,K'}^{\chi}\}_{K' \in \Omega} \in \cE(T)$. 

\begin{proposition}\label{prop:stark}
Suppose that the same assumptions as in Theorem~\ref{thm:main2} hold. 
Then there is a unique system 
\[
\epsilon_{K}^{\rm DR} = (\epsilon_{K,n}^{\rm DR})_{n > 0} \in {\rm SS}_{0}(\bT_{K},\cF_{\rm str}) = \varprojlim_{n>0}{\rm SS}_{0}(T_{K,n},\cF_{\rm str})
\]
such that, for any integer $n>0$, 
\[
{\rm Reg}_{0}(\epsilon_{K,n}^{\rm DR}) =  \cD_{0,n}(\cL_{p}^{\chi}),
\] 
that is, 
\begin{align}\label{rel1}
\epsilon_{K,n,\fn}^{\rm DR}(\wedge_{\fq \mid \fn}\varphi_{\fq}) = (-1)^{\nu(\fn)} \tilde{\kappa}(\cL_{p}^{\chi})_{K,n,\fn}
\end{align}
for any ideal $\fn \in \cN_{n}$. 
In particular, we have  ${\rm pr}_{K}(\epsilon_{K}^{\rm DR}) = L_{p,K}^{\chi}$. 
\end{proposition}
\begin{proof}
Let $c^{\rm DR}$ denote the Euler system of rank $r$ constructed in Theorem~\ref{thm:main2} and $n > 0$ an integer. 
By Theorems~\ref{thm:isom-koly-stark}~and~\ref{thm:derivative}, 
there exists a unique system 
\[
\epsilon_{{\rm can}, n} \in {\rm SS}_{r}(T_{K,n},\cF_{\rm can})  
\]
such that the image of $\epsilon_{{\rm can}, n}$ under  the regulator map 
${\rm Reg}_{r} \colon {\rm SS}_{r}(T_{K,n},\cF_{\rm can}) \longrightarrow {\rm KS}_{r}(T_{K,n},\cF_{\rm can})$ is $\cD_{r,n}(c^{\rm DR})$. 
We define a system $\epsilon_{K,n}^{\rm DR} \in {\rm SS}_{0}(T_{K,n},\cF_{\rm str})$ to be the image of $\epsilon_{{\rm can}, n}$ under the isomorphism ${\rm SS}_{r}(T_{K,n},\cF_{\rm can}) \stackrel{\sim}{\longrightarrow} {\rm SS}_{0}(T_{K,n},\cF_{\rm str})$. 
Then the relation ${\rm Reg}_{0}(\epsilon_{K,n}^{\rm DR}) = \cD_{0,n}(\cL_{p}^{\chi})$ follows from the commutative diagram~\eqref{diagram2}.

Recall that the identification $H^{1}_{/{\rm tr}}(G_{k_{\fq}},T_{K,n}) = R_{K,n}$  
is compatible with $n$. Hence the regulator map ${\rm Reg}_{0} \colon {\rm SS}_{0}(T_{K,n}, \cF_{\rm str}) \longrightarrow \prod_{\fn \in \cN_{n}}R_{K,n}$ is also compatible with $n$. 
Thus, by the definition of $\cD_{0,n}$, it is easy to see that the set $\{\epsilon_{K,n}^{\rm DR}\}_{n>0}$ is an inverse system. Hence $\epsilon_{K}^{\rm DR} := (\epsilon_{K,n}^{\rm DR})_{n > 0}$ is an element of 
$\varprojlim_{n>0}{\rm SS}_{0}(T_{K,n},\cF_{\rm str}) = {\rm SS}_{0}(\bT_{K},\cF_{\rm str})$. 
\end{proof}

\begin{proposition}\label{prop:main3}
Suppose that the same assumptions as in Theorem~\ref{thm:main2} hold. 
%Let us fix an isomorphism $\varprojlim_{K \in \Omega}H^{1}_{\Sigma}(\bT_{K}) \stackrel{\sim}{\longrightarrow} 
%\Lambda[[\Gal(\cK/k)]]^{r}$. 
\begin{itemize}
\item[(i)] The system $\epsilon^{\rm DR}_{K}$ constructed in Proposition~\ref{prop:stark}  is a basis of ${\rm SS}_{0}(\bT_{K},\cF_{\rm str})$. 
\item[(ii)] For each integer $i \geq 0$, we have  
\[
I_{i}(\epsilon^{\rm DR}_{K}) = {\rm Fitt}_{\Lambda_{K}}^{i}(X_{K}^{\chi}). 
\]
\end{itemize}
\end{proposition}

%\begin{remark} 
%The system $\epsilon^{\rm DR}_{K} \in {\rm SS}_{0}(\bT_{K},\cF_{\rm str})$ depends on the choices of the Frobenius elements and the fixed isomorphism 
%\[
%\varprojlim_{K \in \Omega}H^{1}_{\Sigma}(\bT_{K}) \stackrel{\sim}{\longrightarrow} 
%\Lambda[[\Gal(\cK/k)]]^{r}. 
%\]
%Hence the ideals $I_{i}(\epsilon^{\rm DR}_{K})$ dose not depend on these choices. 
%\end{remark}

\begin{proof}
Theorem~\ref{thm:stark}(i) shows that the $\Lambda_{K}$-module ${\rm SS}_{0}(\bT_{K},\cF_{\rm str})$ is free of rank $1$. 
Let $\epsilon_{K} \in {\rm SS}_{0}(\bT_{K},\cF_{\rm str})$ be a basis. 
Then there is an element $a \in \Lambda_{K}$ such that $\epsilon_{K}^{\rm DR} = a \cdot \epsilon_{K}$. 
%In order to prove this theorem, by Theorem~\ref{thm:stark}(ii), it suffices to show that $a$ is a unit. 
Theorem~\ref{thm:stark}(ii) and Proposition~\ref{prop:stark} imply 
\[
a \cdot {\rm Fitt}_{\Lambda_{K}}^{0}(X_{K}^{\chi}) 
= I_{0}(\epsilon^{\rm DR}_{K}) 
= {\rm pr}_{K}(\epsilon^{\rm DR}_{K}) \cdot \Lambda_{K} 
=  L_{p,K}^{\chi} \cdot \Lambda_{K}. 
\]
Since ${\rm Fitt}_{\Lambda_{K}}^{0}(X_{K}^{\chi}) = L_{p,K}^{\chi} \cdot \Lambda_{K}$ by \cite[Proposition~2.1]{Kur03} and \eqref{keyeq}, we have  
\[ 
aL_{p,K}^{\chi} \cdot \Lambda_{K} = L_{p,K}^{\chi} \cdot \Lambda_{K}.
\] 
Hence we conclude that $a$ is a unit, and $\epsilon_{K}^{\rm DR}$ is a basis. 
%Hence Theorem~\ref{thm:kuri03} shows $a \cdot {\rm Fitt}_{\Lambda_{K}}^{0}(X_{K}^{\chi}) = {\rm Fitt}_{\Lambda_{K}}^{0}(X_{K}^{\chi})$. 
%Since the ideal ${\rm Fitt}_{\Lambda_{K}}^{0}(X_{K}^{\chi})$ is non-trivial and $\Lambda_{K}$ is a local ring, 
%Nakayama's lemma implies that  $a$ is a unit, and $\epsilon_{K}^{\rm DR}$ is a basis. 
Claims~(ii) follows from Claim~(i) and Theorem~\ref{thm:stark}(ii). 
\end{proof}

\section{On the Fitting ideals of $X_{K}^{\chi}$}\label{sec:fitt}

In this section, by using the system $\epsilon_{K}^{\rm DR}$ constructed in Proposition~\ref{prop:main3}, 
we prove that all higher Fitting ideals of $X_{K}^{\chi}$ are described by analytic invariants canonically associated with Stickelberger elements. 

Throughout this section, we use the same notations as in \S\S\ref{sec:euler G_{m}} and \ref{sec:stark}. 
Suppose that the odd prime $p$ is coprime to the class number of $k$ and the order of $\chi$. 
Fix a field $K \in \Omega$. 

\begin{definition}\label{def:theta}
For each integer $i \geq 0$, we define an ideal $\Theta_{K,n}^{i}$ by 
\[
\Theta_{K,n}^{i} := \langle \tilde{\kappa}(\cL_{p}^{\chi})_{K, n,\fn} \mid \nu(\fn) \leq i, \fn \in \cN_{n} \rangle = 
\langle \kappa(\cL_{p}^{\chi})_{K, n,\fn} \mid \nu(\fn) \leq i, \fn \in \cN_{n} \rangle \subseteq R_{K,n}. 
\]
Since $\cN_{n+1} \subseteq \cN_{n}$, one can define an ideal 
\[
\Theta_{K,\infty}^{i} := \varprojlim_{n>0}\Theta_{K,n}^{i} \subseteq \Lambda_{K}. 
\]
\end{definition}

\begin{remark}\label{rem:fitt}\
\begin{itemize}
\item[(i)] By Remark~\ref{rem:leading}, the ideal $\Theta_{K, n}^{i}$ 
is generated by the leading coefficients of elements $\{\pi_{K(\fn),n}(L_{p,K(\fn)}^{\chi})\}_{\fn \in \cN_{n}}$, where 
\[
\pi_{K(\fn),n} \colon \Lambda_{K(\fn)} \longrightarrow R_{K,n}[\Gal(K_{n}(\fn)/K_{n})]
\] 
is the canonical projection map. 
\item[(ii)] We have  $\Theta_{K,\infty}^{0} = L_{p,K}^{\chi} \cdot \Lambda_{K} = {\rm Fitt}_{\Lambda_{K}}^{0}(X_{K}^{\chi})$ by \cite[Proposition~2.1]{Kur03}. 
%\item[(iii)] Kurihara defined a similar ideal $\Theta^{(\delta)}_{i,K}$ called by the small $i$-th Stickelberger ideal (see \cite[\S8.1]{Kur12} for example). 
\end{itemize}
\end{remark}

%\begin{proposition}
%Assume that 
%\begin{itemize}
%\item the odd prime $p$ is coprime to the class number of $k$, 
%\item $H^{0}(G_{k_{\fp}},T/\fm T) = H^{2}(G_{k_{\fp}},T/\fm T) = 0$
%for each prime $\fp \in S_{p}(k)$, and 
%\item $H^{2}(G_{k_{\fq}},T/\fm T) = 0$ for each prime $\fq \in S_{\rm ram}(K/k)$. 
%\end{itemize}
%Then, for a non-negative integer $i$, we have  
%\[
%\Theta_{K,\infty}^{i} \subseteq {\rm Fitt}^{i}_{\Lambda_{K}}(X_{K}^{\chi}).
%\] 
%\end{proposition}
%\begin{proof}
%This proposition follows from Proposition~\ref{prop:stark}, the definition of the ideals $I_{i}(\epsilon_{K,n}^{\rm DR})$, and Remark~\ref{rem:ideal}. 
%\end{proof}

We give a proof of the following theorem in the next subsection. 
As we mentioned in Remark~\ref{rem:kurihara}, the following theorem is an equivariant generalization of the main theorem in \cite{Kur12}. 

\begin{theorem}\label{thm:main4}
Assume that 
\begin{itemize}
\item $H^{0}(G_{k_{\fp}},T/\fm T) = H^{2}(G_{k_{\fp}},T/\fm T) = 0$
for each prime $\fp \in S_{p}(k)$, and 
\item $H^{2}(G_{k_{\fq}},T/\fm T) = 0$ for each prime $\fq \in S_{\rm ram}(K/k)$, 
\item the $\mu$-invariant of $X_{k}^{\chi}$ vanishes. 
\end{itemize}
Then, for an integer $i \geq 0$, we have  
\[
\Theta_{K,\infty}^{i} = {\rm Fitt}^{i}_{\Lambda_{K}}(X_{K}^{\chi}).
\] 
\end{theorem}

\subsection{Proof of Theorem~\ref{thm:main4}}
In this subsection, suppose that the same assumptions as in Theorem~\ref{thm:main4} hold. 
We fix an integer $n > 0$. 
Since 
\[
\varprojlim_{n'>0}H^{1}_{\cF_{\rm str}}(k, T_{K,n'}) = \ker\left(H^{1}(G_{k,S_{K}},\bT_{K}) \longrightarrow H^{1}_{\Sigma}(\bT_{K})\right) = 0
\] 
by Proposition~\ref{prop:inj} 
and $H^{1}_{\cF_{\rm str}}(k, T_{K,n'})$ is finite for each integer $n' > 0$, 
one can take a sequence $\{n_{i}\}_{i=0}^{\infty}$ of positive integers such that 
$n = n_{0} < n_{1} < n_{2} < \cdots$ and a canonical map $H^{1}_{\cF_{\rm str}}(k, T_{K,n_{i+1}}) \longrightarrow H^{1}_{\cF_{\rm str}}(k,T_{K,n_{i}})$ is zero. 

Let ${\rm End}(\cP_{n_{i}})$ denote the set of maps from  $\cP_{n_{i}}$ to $\cP_{n_{i}}$. 
Fix an integer $i>0$ and take an ideal $\fn \in \cN_{n_{i}}$ and a prime $\fq \in \cP_{n_{i}}$ with $\fq \nmid \fn$. 
%Take an ideal $\fn \in \cN_{n_{i}}$ such that $i \leq \nu(\fn)$ and the homomorphism 
%\[
%H^{1}_{\rm str}(G_{k,S_{K(\fn)}},T_{K,n_{j}}) \xrightarrow{\oplus_{\fq \mid \fn} \varphi_{\fq}} \bigoplus_{\fq \mid \fn}R_{K,n_{j}}
%\]
%is an isomorphism for any integer $0 \leq j \leq i$ (see Remark~\ref{rem:fs map}(ii)). 
%Hence $\Hom_{R_{K,n}}(H^{1}_{\rm str}(G_{k,S_{K,\fm}},T_{K,n}),R_{K,n})$ is generated by $\{\varphi_{\fq}\}_{\fq \mid \fm}$ and, for any ideal $\fn \mid \fm$ and system $\epsilon \in {\rm SS}_{0}(T_{K,n},\cF_{\rm str})$, the ideal $\im(\epsilon_{\fn})$ is generated by the set \[
%\{\epsilon_{\fn}(\wedge_{\fq \mid \fd}\varphi_{\fq}) \mid \nu(\fd) = \nu(\fn), \fd \mid \fm\}. 
%\]
For a map $\sigma \in {\rm End}(\cP_{n_{i}})$, 
%For each prime $\fq \in \cP_{n_{i}}$, we take an $R_{K,n_{i}}$-homomorphism 
%\[
%f_{\fq} \colon H^{1}(G_{k},T_{K,n_{i}}) \longrightarrow R_{K,n_{i}}. 
%\]
we set 
\[
\varphi_{\fn}^{\sigma} := \wedge_{\fr \mid \fn}\varphi_{\sigma(\fr)} \colon {\bigwedge}^{\nu(\fn)}_{R_{K,n}}H^{1}(G_{k},T_{K,n}) \longrightarrow R_{K,n},  
\] 
where the homomorphism $\varphi_{\fr}$ is defined in Definition~\ref{def:varphi} and we fix an order on the set of prime divisors of $\fn$. 

Let 
\[
\epsilon_{K,n_{i}}^{\rm DR} = (\epsilon_{K,n_{i},\fn}^{\rm DR})_{\fn \in \cN_{n_{i}}} \in {\rm SS}_{0}(T_{K,n_{i}},\cF_{\rm str})
\] 
be the system constructed in Proposition~\ref{prop:stark}. 
Let 
\[
{\rm div}_{\fq} \colon  H^{1}_{/f}(G_{k_{\fq}},T_{K,n_{j}}) \stackrel{\sim}{\longrightarrow}  R_{K,n_{j}}
\]
denote the fixed isomorphism (see the paragraph after Lemma~\ref{lem:tr}). 
%the order map at the fixed prime $\fq_{K,n_{j}}$ of $KL(\mu_{p^{n_{j}}})$ (see Definition~\ref{def:order}). 
We also denote by ${\rm div}_{\fq}$ the composite map  
\[
H^{1}(G_{k},T_{K,n_{j}}) \longrightarrow H^{1}_{/f}(G_{k_{\fq}},T_{K,n_{j}}) \stackrel{\sim}{\longrightarrow}  R_{K,n_{j}}. 
\]
Recall that 
\[
\epsilon_{K,n_{i},\fn}^{\rm DR} \in {\bigcap}^{\nu(\fn)}_{R_{K,n}}H^{1}_{\cF_{\rm str}^{\fn}}(k,T_{K,n_{i}}) \otimes \det(W_{\fn}). 
\]
Here $W_{\fn} := \bigoplus_{\fq \mid \fn} H^{1}_{/f}(G_{k_{\fq}},T_{K,n_{i}})^{*}$. 

\begin{definition}
We define an element $\overline{\epsilon}_{K,n_{i},\fn, \fq} \in {\bigcap}^{\nu(\fn\fq)}_{R_{K,n}}H^{1}_{\cF_{\rm str}^{\fn\fq}}(k,T_{K,n_{i}})$ by  
\[
\epsilon_{K,n_{i},\fn\fq}^{\rm DR} = \overline{\epsilon}_{K,n_{i},\fn, \fq}  \otimes (\wedge_{\fr \mid \fn}{\rm div}_{\fr} \wedge {\rm div}_{\fq}),  
\]
where, in order to define $\wedge_{\fr \mid \fn}{\rm div}_{\fr}$, we use the same order 
on the set of prime divisors of $\fn$ as above. 
We also define an element $\overline{\epsilon}_{K,n_{i},\fn, 1} \in {\bigcap}^{\nu(\fn)}_{R_{K,n}}H^{1}_{\cF_{\rm str}^{\fn}}(k,T_{K,n_{i}})$ by  
\[
\epsilon_{K,n_{i},\fn}^{\rm DR} = \overline{\epsilon}_{K,n_{i},\fn, 1}  \otimes \wedge_{\fr \mid \fn}{\rm div}_{\fr}.   
\]

We set
\[
\kappa_{n_{i},\fn,\fq}^{\sigma} := (-1)^{\nu(\fn)}\overline{\epsilon}_{K,n_{i},\fn, \fq}(\varphi_{\fn}^{\sigma}) \in H^{1}_{\cF_{\rm str}^{\fn\fq}}(k,T_{K,n_{i}}). 
\]
Note that $\kappa_{n_{i},\fn,\fq}^{\sigma}$ is independent of the choice of the fixed order on the set of prime divisors of $\fn$. 
We denote by $\kappa_{n_{j},\fn,\fq}^{\sigma}$ the image of $\kappa_{n_{i},\fn,\fq}^{\sigma}$ in 
$H^{1}_{\cF_{\rm str}^{\fn\fq}}(k, T_{K,n_{j}})$ for $0 \leq j \leq i$. 
\end{definition}

\begin{definition}
We set 
\[
\delta_{n_{i},\fn}^{\sigma} := (-1)^{\nu(\fn)}\overline{\epsilon}_{K,n_{i},\fn, 1}(\varphi^{\sigma}_{\fn}) \in R_{K,n_{i}}. 
\]
%Note that, for any  prime divisor $\fq$ of $\fn$, we have  
%\[
%\delta_{n_{i},\fn}^{\sigma} = -\varphi_{\sigma(\fq)}(\kappa_{n_{i},\fn/\fq,\fq}). 
%\]
\end{definition}

The following proposition follows from the construction of cohomology classes $\{\kappa_{n_{j},\fn,\fq}^{\sigma}\}$. 

\begin{proposition}\label{prop:coh rel}
The elements $\{\kappa_{n_{j},\fn,\fq}^{\sigma} \mid \fn \in \cN_{n_{i}}, \fq \in \cP_{n_{i}}, \fq \nmid \fn, \sigma \in {\rm End}(\cP_{n_{i}}) \}$ satisfies the following relations: 
\begin{itemize}
\item[(i)] ${\rm div}_{\fr}(\kappa_{n_{j},\fn,\fq}^{\sigma}) = \varphi_{\sigma(\fr)}(\kappa_{n_{j},\fn/\fr,\fq}^{\sigma})$ for each prime $\fr \mid \fn$, 
\item[(ii)] ${\rm div}_{\fq}(\kappa_{n_{j},\fn,\fq}^{\sigma}) = \delta_{n_{j},\fn}^{\sigma}$, 
\item[(iii)]  $\varphi_{\sigma(\fr)}(\kappa_{n_{j},\fn,\fq}^{\sigma}) = 0$ for each prime $\fr \mid \fn$, 
\item[(iv)] $\varphi_{\sigma(\fq)}(\kappa_{n_{j},\fn,\fq}^{\sigma}) = -\delta_{n_{j},\fn\fq}^{\sigma}$.  
\end{itemize}
\end{proposition}

\begin{remark}
The key of the proof of \cite[Theorem~2.1]{Kur12} is the construction of certain cohomology classes which satisfy the same properties (i)--(iv) in Proposition~\ref{prop:coh rel} (see \cite[\S1]{Kur12}). 
However, Kurihara's construction is different from ours. 
We use a higher rank Euler system and a Stark system in order to construct cohomology classes $\{\kappa_{n_{j},\fn,\fq}^{\sigma}\}$. 
\end{remark}

\begin{proof}
Let us show Claim~(i). 
Since $\kappa_{n_{j},\fn,\fq}^{\sigma}$ and $\kappa_{n_{j},\fn/\fr,\fq}^{\sigma}$ are independent of the choice of the fixed order on the set of prime divisors of $\fn$ and of $\fn/\fr$, respectively, 
we may assume that $\varphi^{\sigma}_{\fn} = \varphi_{\sigma(\fr)} \wedge \varphi^{\sigma}_{\fn/\fr}$. 
%The cartesian diagram~\eqref{} for $\cF=\cF_{\rm str}$, $\fm = \fn\fq$, and $\fn = \fn\fq/\fr$ incudes a homomorphism 
%\[
%{\bigcap}^{\nu(\fn\fq)}_{R_{K,n}}H^{1}_{\cF_{\rm str}^{\fn\fq}}(k,T_{K,n_{i}}) \otimes \det(W_{\fn\fq}) \longrightarrow {\bigcap}^{\nu(\fn\fq/\fr)}_{R_{K,n}}H^{1}_{\cF_{\rm str}^{\fn\fq/\fr}}(k,T_{K,n_{i}}) \otimes \det(W_{\fn\fq/\fr}).  
%\]
%We also denote by ${\rm div}_{\fr}$ this homomorphism.  
%
By the definition of Stark systems, we have  $\Phi_{\fn\fq, \fn\fq/\fr}(\epsilon^{\rm DR}_{K,n_{j},\fn\fq}) = \epsilon^{\rm DR}_{K,n_{j},\fn\fq/\fr}$ (see the paragraph after Definition~\ref{def:det} for the definition of $\Phi_{\fn\fq, \fn\fq/\fr}$). 
Hence by Definition~\ref{def:map}, we have  
\[
{\rm div}_{\fr}(\overline{\epsilon}_{K,n_{j},\fn,\fq}) = \overline{\epsilon}_{K,n_{j},\fn/\fr,\fq}, 
\]
which implies 
\begin{align*}
{\rm div}_{\fr}(\kappa_{n_{j},\fn,\fq}^{\sigma}) 
&= (-1)^{\nu(\fn)}\overline{\epsilon}_{K,n_{j},\fn,\fq}(\varphi^{\sigma}_{\fn} \wedge {\rm div}_{\fr}) 
\\
&= {\rm div}_{\fr}(\overline{\epsilon}_{K,n_{j},\fn,\fq})(\varphi_{\sigma(\fr)} \wedge \varphi^{\sigma}_{\fn/\fr})
\\
&= \overline{\epsilon}_{K,n_{j},\fn/\fr,\fq}(\varphi_{\sigma(\fr)} \wedge \varphi^{\sigma}_{\fn/\fr}) 
\\
&= (-1)^{\nu(\fn/\fr)} \overline{\epsilon}_{K,n_{j},\fn/\fr,\fq}( \varphi^{\sigma}_{\fn/\fr} \wedge \varphi_{\sigma(\fr)}) 
\\
&= \varphi_{\sigma(\fr)}(\kappa_{n_{j},\fn/\fr,\fq}^{\sigma}). 
\end{align*}

Next, we will show Claim~(ii). 
By the definition of Stark systems, we have  ${\rm div}_{\fq}(\epsilon^{\rm DR}_{K,n_{j},\fn\fq}) = \epsilon^{\rm DR}_{K,n_{j},\fn}$. 
Hence by Definition~\ref{def:map}, we have  
\[
{\rm div}_{\fq}(\overline{\epsilon}_{K,n_{j},\fn,\fq}) 
%\otimes \wedge_{\fr \mid \fn} {\rm div}_{\fr} 
=  (-1)^{\nu(\fn)}\overline{\epsilon}_{K,n_{j},\fn, 1},  
%\otimes \wedge_{\fr \mid \fn} {\rm div}_{\fr},  
\]
which implies 
\begin{align*}
{\rm div}_{\fq}(\kappa_{n_{j},\fn,\fq}^{\sigma}) 
&= (-1)^{\nu(\fn)}\overline{\epsilon}_{K,n_{j},\fn,\fq}( \varphi_{\fn}^{\sigma}  \wedge {\rm div}_{\fq}) 
\\
&= {\rm div}_{\fq}(\overline{\epsilon}_{K,n_{j},\fn\fq})(\varphi_{\fn}^{\sigma}) 
\\
&= (-1)^{\nu(\fn)}\overline{\epsilon}_{K,n_{j},\fn, 1}(\varphi_{\fn}^{\sigma})
\\
&= \delta_{n_{j},\fn}^{\sigma}. 
\end{align*}

Claim~(iii) follows from the fact that $\varphi_{\sigma(\fr)} \wedge \varphi_{\fn}^{\sigma}$ vanishes for each prime divisor $\fr \mid \fn$.

By definition, 
the sign difference between $\overline{\epsilon}_{K,n_{j},\fn\fq,1}$ and $\overline{\epsilon}_{K,n_{j},\fn,\fq}$ is the same as the one between $\varphi_{\fn\fq}^{\sigma}$ and $\varphi^{\sigma}_{\fn} \wedge \varphi_{\sigma(\fq)}$. 
Since $(-1)^{\nu(\fn)} \times (-1)^{\nu(\fn\fq)} = -1$, this fact implies Claim~(iv). 
\end{proof}

When the map $\sigma$ is the identity map, put $\delta_{n_{i},\fn} := \delta_{n_{i},\fn}^{\sigma}$ for simplicity.

\begin{lemma}\label{lem:key0}
We have  
\begin{align*}
I_{i}(\epsilon_{K,n}^{\rm DR}) 
&= \langle \delta_{n,\fn}^{\sigma} \mid \fn \in \cN_{n_{i}}, \nu(\fn) = i, \sigma \in {\rm End}(\cP_{n_{i}}) \rangle. 
%\\
%%&= \langle \varphi_{\sigma(\fq)}(\kappa_{n,\fn,\fq}^{\sigma}) \mid \fn \in \cN_{n_{i}}, \fq \in \cP_{n_{i}}, \fq \nmid \fn, \nu(\fn) = i-1,\sigma \in \fS(\fn\fq) \rangle, 
%%\\
%\Theta^{i}_{K,n} &\supseteq \langle \delta_{n,\fn} \mid \fn \in \cN_{n_{i}},  \nu(\fn) = i \rangle. 
\end{align*}
\end{lemma}
\begin{proof}
This lemma follows from Remark~\ref{rem:fs map} and \eqref{rel1}. 
\end{proof}

\begin{lemma}[{\cite[Lemma~6.9]{Kur14a}}]\label{lem:key1}
Let $0 < j \leq i$ and  $c \in H^{1}_{\cF_{\rm str}^{\fn}}(k, T_{K,n_{j}})$. 
If ${\rm div}_{\fr}(c) = 0$ for any $\fr \mid \fn$, then $c = 0$ in 
$H^{1}_{\cF_{\rm str}^{\fn}}(k, T_{K,n_{j-1}})$. 
\end{lemma}
\begin{proof}
We have  an exact sequence 
\[
0 \longrightarrow H^{1}_{\cF_{\rm str}}(k, T_{K,n_{j}}) \longrightarrow 
H^{1}_{\cF_{\rm str}^{\fn}}(k, T_{K,n_{j}}) 
\xrightarrow{\oplus_{\fq \mid \fn}{\rm div}_{\fq}} \bigoplus_{\fq \mid \fn}R_{K,n_{j}}. 
\]
Hence $c \in H^{1}_{\cF_{\rm str}}(k,T_{K,n_{j}})$. Since 
$H^{1}_{\cF_{\rm str}}(k, T_{K,n_{j}}) \longrightarrow H^{1}_{\cF_{\rm str}}(k,T_{K,n_{j-1}})$ is zero, we have  $c = 0$ in 
$H^{1}_{\cF_{\rm str}^{\fn}}(k, T_{K,n_{j-1}})$. 
\end{proof}

\begin{lemma}[{\cite[Proposition~7.14]{Kur14a}}]\label{lem:key2}
Let $\fr \in \cP_{n_{i}}$ be a prime with $\fr \nmid \fn\fq$ and $\sigma \in {\rm End}(\cP_{n_{i}})$. 
Take a homomorphism $f \colon H^{1}(G_{k},T_{K,n}) \longrightarrow R_{K,n}$. 
Suppose that there is an element $z \in H^{1}_{\cF_{\rm str}^{\fq\fr}}(k, T_{K,n_{i}})$ with ${\rm div}_{\fq}(z) = 1$. 
%and ${\rm div}_{\fr}(z) \in R_{K,n}^{\times}$. 
If $\nu(\fn) = i-1$, then we have  
\[
f(\kappa_{n,\fn,\fq}^{\sigma}) \equiv -{\rm div}_{\fr}(z) f(\kappa_{n,\fn,\fr}^{\sigma}) \pmod{ I_{i-1}(\epsilon^{\rm DR}_{K,n})}. 
\]
In particular, we have  
\[
\delta_{n,\fn\fq}^{\sigma} \equiv  {\rm div}_{\fr}(z) \varphi_{\sigma(\fq)}(\kappa_{n,\fn,\fr}^{\sigma}) \pmod{ I_{i-1}(\epsilon^{\rm DR}_{K,n})}. 
\]
\end{lemma}
\begin{proof}
The proof of this lemma is based on that of \cite[Proposition~7.14]{Kur14a}. 
For an ideal $\fm \mid \fn$, we set $j := \nu(\fm)$ and 
\[
\tilde{\kappa}_{n_{i}, \fm, \fq}^{\sigma} := -{\rm div}_{\fr}(z)
\kappa_{n_{i},\fm,\fr}^{\sigma} + \delta_{n_{i},\fm}^{\sigma}z + \sum_{\fs \mid \fm}\varphi_{\sigma(\fs)}(z)\kappa_{n_{i},\fm/\fs,\fs}^{\sigma}, 
\]
where $\fs$ runs over the set of prime divisors of $\fm$. 
Let us show 
\[
\kappa_{n_{i-1-j}, \fm,\fq}^{\sigma} = \tilde{\kappa}_{n_{i-1-j}, \fm,\fq}^{\sigma}
\]
 by induction on $j \leq i-1$. 
Since ${\rm div}_{\fq}(z) = 1$, Proposition~\ref{prop:coh rel}~(ii) shows that   
\begin{align*}
&{\rm div}_{\fq}(\tilde{\kappa}_{n_{i}, \fm,\fq}^{\sigma}) = \delta_{n_{i},\fm}^{\sigma} = 
{\rm div}_{\fq}(\kappa_{n_{i}, \fm,\fq}^{\sigma}), 
\\
&{\rm div}_{\fr}(\tilde{\kappa}_{n_{i}, \fm,\fq}^{\sigma}) 
= -{\rm div}_{\fr}(z)\delta_{n_{i},\fm}^{\sigma} + {\rm div}_{\fr}(z) \delta_{n_{i},\fm}^{\sigma}
= 0 
= {\rm div}_{\fr}(\kappa_{n_{i}, \fm,\fq}^{\sigma}). 
\end{align*}
Hence if $j = 0$, then $\fm = 1$, and  this claim follows from Lemma~\ref{lem:key1}. 

Let us suppose that $j>0$. Take a prime divisor $\fs$ of $\fm$. 
Then Proposition~\ref{prop:coh rel}~(i)~and~(ii) implies  
\begin{align*}
{\rm div}_{\fs}(\tilde{\kappa}_{n_{i}, \fm,\fq}^{\sigma}) 
&= -{\rm div}_{\fr}(z){\rm div}_{\fs}(\kappa_{n_{i},\fm,\fr}^{\sigma}) 
+ \varphi_{\sigma(\fs)}(z){\rm div}_{\fs}(\kappa_{n_{i},\fm/\fs,\fs}^{\sigma}) 
+ \sum_{\fs' \mid \fm/\fs}\varphi_{\sigma(\fs')}(z){\rm div}_{\fs}(\kappa_{n_{i},\fm/\fs',\fs'}^{\sigma})  \\
&= \varphi_{\sigma(\fs)}(-{\rm div}_{\fr}(z)\kappa_{n_{i},\fm/\fs,\fr}^{\sigma}) 
+ \delta_{n_{i},\fm/\fs}^{\sigma}\varphi_{\sigma(\fs)}(z) 
+ \sum_{\fs' \mid \fm/\fs}\varphi_{\sigma(\fs')}(z)\varphi_{\sigma(\fs)}(\kappa_{n_{i},\fm/\fs\fs',\fs'}^{\sigma})
\\
&=   \varphi_{\sigma(\fs)}(\tilde{\kappa}_{n_{i}, \fm/\fs, \fq}). 
\end{align*}
Hence, the induction hypothesis and Proposition~\ref{prop:coh rel}~(i) show that  
\[
{\rm div}_{\fs}(\tilde{\kappa}_{n_{i-j}, \fm,\fq}^{\sigma}) 
= \varphi_{\sigma(\fs)}(\tilde{\kappa}_{n_{i-j}, \fm/\fs,\fq}^{\sigma})
= \varphi_{\sigma(\fs)}(\kappa_{n_{i-j}, \fm/\fs,\fq}^{\sigma}) 
= {\rm div}_{\fs}(\kappa_{n_{i-j}, \fm,\fq}^{\sigma}). 
\]
Thus, by Lemma~\ref{lem:key1}, we conclude that $\kappa_{n_{i-1-j}, \fm,\fq}^{\sigma} = \tilde{\kappa}_{n_{i-1-j}, \fm,\fq}^{\sigma}$. 

Take a homomorphism $f \colon H^{1}(G_{k},T_{K,n}) \longrightarrow R_{K,n}$. 
Since $\delta_{n,\fn}^{\sigma}$ and 
$f(\kappa_{n,\fn/\fs,\fs}^{\sigma})$ are contained in $I_{i-1}(\epsilon^{\rm DR}_{K,n})$ by definition, we have  
\begin{align*}
f(\kappa_{n, \fn,\fq}^{\sigma}) &= f(\tilde{\kappa}_{n, \fn, \fq}^{\sigma}) 
\\
&= - {\rm div}_{\fr}(z) f(\kappa_{n,\fn,\fr}^{\sigma}) 
+ \delta_{n,\fn}^{\sigma}f(z) 
- \sum_{\fs \mid \fn}\varphi_{\sigma(\fs)}(z)f(\kappa_{n,\fn/\fs,\fs}^{\sigma}) 
\\
&\equiv - {\rm div}_{\fr}(z) f(\kappa_{n,\fn,\fr}^{\sigma}) \pmod{ I_{i-1}(\epsilon^{\rm DR}_{K,n})}. 
\end{align*}
%Since $R_{K,n}$ is a zero-dimensional Gorenstein local ring, Matlis duality implies 
%\[
%\kappa_{n,\fn,\fq}^{\sigma} \equiv \kappa_{n,\fn,\fr}^{\sigma} \pmod{ I_{i-1}(\epsilon^{\rm DR}_{K,n})}.
%\] 
\end{proof}

%We give a proof of the following lemma in the next subsection. 
The following lemma is proved by Kurihara in \cite{Kur14a}. 

\begin{lemma}[{\cite[Proposition~9.3]{Kur14a}}]\label{lem:key3}
Let $\fm \in \cN_{n}$ be an ideal, 
$x \in H^{1}_{\cF_{\rm str}^{\fm}}(k,T_{K,n})$, and $\fq \mid \fm$. 
Then there are infinitely many primes $\fr \in \cP_{n}$ with $\fr \nmid \fm$ such that 
\begin{itemize}
\item $\varphi_{\sigma(\fq)}(x) \in R_{K,n} \cdot \varphi_{\fr}(x)$, and 
\item there is an element $z \in H^{1}_{\cF_{\rm str}^{\fq\fr}}(k,T_{K,n})$ with ${\rm div}_{\fq}(z) = 1$ and ${\rm div}_{\fr}(z) \in R_{K,n}^{\times}$. 
\end{itemize}
\end{lemma}

The following corollary follows from Lemmas~\ref{lem:key2} and \ref{lem:key3}. 

\begin{corollary}\label{cor:key}
Let $\sigma \in {\rm End}(\cP_{n_{i}})$ be a map and $\fn \in \cN_{n_{i}}$ an ideal with $\nu(\fn) = i$. 
Then there are infinitely many primes $\fr \in \cP_{n_{i}}$ such that 
\[
\delta^{\sigma}_{n,\fn} \in R_{K,n} \cdot \varphi_{\fr}(\kappa_{n,\fn/\fq,\fr}^{\sigma}) + I_{i-1}(\epsilon_{K,n}^{\rm DR}). 
\]
\end{corollary}

\begin{proof}[Proof of Theorem~\ref{thm:main4}]
%Let us show $I_{i}(\epsilon^{\rm DR}_{K,n}) = \Theta_{K,n}^{i} + I_{i-1}(\epsilon_{K,n}^{\rm DR})$. 
%By Remark~\ref{rem:fs map}(ii) and \eqref{rel1}, it suffices to show that 
Let $\fn \in \cN_{n_{i}}$ be an ideal satisfying $\nu(\fn) = i$. 
Take a map $\sigma \in {\rm End}(\cP_{n_{i}})$. 
Let us prove 
\[
\delta_{n,\fn}^{\sigma} \in \Theta_{K,n}^{i} + I_{i-1}(\epsilon_{K,n}^{\rm DR})
\]
by induction on $d := i - \#\{ \fq \mid  \textrm{$\fq$ is a prime divisor of $\fn$ and $\sigma(\fq) = \fq$} \}$.

If $d=0$, then the claim follows from Lemma~\ref{lem:key0}. 
Suppose that $d>0$. 
Take a prime divisor $\fq$ of $\fn$ satisfying $\sigma(\fq) \neq \fq$. 
By Corollary~\ref{cor:key}, there is a prime $\fr \in \cP_{n_{i}}$ such that 
\[
\delta_{n,\fn}^{\sigma} \in R_{K,n} \cdot \varphi_{\fr}(\kappa_{n,\fn/\fq,\fr}^{\sigma}) + I_{i-1}(\epsilon_{K,n}^{\rm DR}).  
\]
We define a map $\sigma'  \in {\rm End}(\cP_{n_{i}})$ to be  
$\sigma'(\fr) = \fr$ and $\sigma'(\fs) = \sigma(\fs)$ for $\fs \in \cP_{n_{i}} \setminus \{\fr\}$.  
Then  Proposition~\ref{prop:coh rel}~(iv) shows 
\[
\delta_{n,\fn}^{\sigma} \in R_{K,n} \cdot \delta_{n,\fn\fr/\fq}^{\sigma'} + I_{i-1}(\epsilon_{K,n}^{\rm DR}).  
\]
Since $\sigma(\fq) \neq \fq$ and $\sigma'(\fr) = \fr$, we obtain 
\[
i - \#\{ \fs \mid  \textrm{$\fs$ is a prime divisor of $\fn\fr/\fq$ and $\sigma'(\fs) = \fs$} \} 
= d-1. 
\]
Hence applying the induction hypothesis for $\sigma'$ and $\fn\fr/\fq$, 
we have  $\delta_{n,\fn\fr/\fq}^{\sigma'} \in \Theta_{K,n}^{i} + I_{i-1}(\epsilon_{K,n}^{\rm DR})$ and hence 
$\delta_{n,\fn}^{\sigma} \in \Theta_{K,n}^{i} + I_{i-1}(\epsilon_{K,n}^{\rm DR})$. 

In particular, Lemma~\ref{lem:key0} shows $I_{i}(\epsilon^{\rm DR}_{K,n}) = \Theta_{K,n}^{i} + I_{i-1}(\epsilon_{K,n}^{\rm DR})$ for any integer $i>0$.  
Hence Proposition~\ref{prop:main3} implies  
\[
\Fitt_{\Lambda_{K}}^{i}(X_{K}^{\chi}) = I_{i}(\epsilon^{\rm DR}_{K}) = \Theta_{K,\infty}^{i} + I_{i-1}(\epsilon_{K}^{\rm DR}) = 
\Theta_{K,\infty}^{i} + \Fitt_{\Lambda_{K}}^{i-1}(X_{K}^{\chi}). 
\]
Since $\Fitt_{\Lambda_{K}}^{0}(X_{K}^{\chi}) \subseteq \Theta_{K,\infty}^{1}$ by Remark~\ref{rem:fitt}(ii), 
we conclude that $\Fitt_{\Lambda_{K}}^{i}(X_{K}^{\chi})  = \Theta_{K,\infty}^{i}$ for any  integer $i>0$. 
\end{proof}

\begin{appendix}
%\appendixpage
%\noappendicestocpagenum
%\addappheadtotoc

\section{Notes on Galois cohomology}

\subsection{Remarks on continuous cochain complexes}

The contents of this subsection are based on \cite[Appendix~A]{BS}. 

Let $R$ be a complete noetherian local ring  with maximal ideal $\fm$, and $k$ a number field. 
Fix a finite set  $S$ of places of $k$ satisfying $S_{\infty}(k) \cup S_{p}(k) \subseteq S$. 
Let $T$ be a free $R$-module of finite rank on which $G_{k,S}$ acts continuously; namely, the composite map  
\[
G_{k,S} \longrightarrow R[G_{k,S}] \longrightarrow {\rm End}_{R}(T)
\]
is continuous. 
Let $G$ be one of the profinite groups in $\{G_{k_{\fp}} \mid \fp \in S_{p}(k)\} \cup \{G_{k,S}\}$. 
Note that the pair $(T,G)$ satisfies \cite[Hypothesis~A]{Pot13}. 

\begin{definition}
Let $A$ be a noetherian ring and $M^{\bullet}$ a complex of $A$-modules. 
\begin{itemize}
\item[(i)] We say that $M^{\bullet}$ is a perfect complex if $M^{\bullet}$ is quasi-isomorphic to a bounded complex $P^{\bullet}$ of projective $A$-modules of finite type. 
Furthermore, we say that $M^{\bullet}$ has perfect amplitude contained in $[a,b]$ if $P^{i} = 0$ for every $i < a$ and $i > b$.  
\item[(ii)] We write $D_{\rm perf}^{[a,b]}(A)$ for the full subcategory consisting of perfect complexes having perfect amplitude contained in $[a,b]$ in the derived category of the abelian category of $A$-modules. 
\end{itemize}
\end{definition}

\begin{theorem}[{\cite[Theorem~1.4~(1) and Corollary~1.2]{Pot13}}]\label{thm:pot}\ 
\begin{itemize}
\item[(i)] For an ideal $I$ of $R$, we have the canonical isomorphism 
\[
{\bf R}\Gamma(G,T) \otimes^{\bL}_{R}R/I \stackrel{\sim}{\longrightarrow} {\bf R}\Gamma(G,T/IT). 
\]
\item[(ii)] We have  ${\bf R}\Gamma(G,T) \in D_{\rm perf}^{[0,2]}(R)$. 
\end{itemize}
\end{theorem}

\begin{corollary}[{\cite[Corollary~A.5]{BS}}]\label{cor:amp}
If $H^{0}(G,T/\fm T)=0$, then we have   ${\bf R}\Gamma(G,T) \in D_{\rm perf}^{[1,2]}(R)$. 
\end{corollary}
\begin{proof}
By Theorem~\ref{thm:pot}(ii), there is a complex 
\[
P^{\bullet} = [ \ \cdots \longrightarrow 0 \longrightarrow P^{0} \longrightarrow P^{1} \longrightarrow P^{2} \longrightarrow 0 \longrightarrow \cdots \ ]
\] 
of finitely generated free $R$-modules such that $P^{\bullet}$ is quasi-isomorphic to ${\bf R}\Gamma(G,T)$. 
The vanishing of $H^{0}(G,T/\fm T)$ implies $(\fm^{i}/\fm^{i+1} \otimes_{R} T)^{G} = 0$. 
Hence the exact sequence 
\[
0 \longrightarrow \fm^{i}/\fm^{i+1} \longrightarrow R/\fm^{i+1} \longrightarrow R/\fm^{i} \longrightarrow 0
\]
shows $(T/\fm^{i} T)^{G} = 0$ for any integer $i \geq 1$ and 
hence $H^{0}(P^{\bullet}) = T^{G} = 0$, which means that  the map $P^{0} \longrightarrow P^{1}$ is injective. 
Thus it suffices to show that the cokernel 
\[
X := {\rm coker}\left(P^{0} \longrightarrow P^{1}\right)
\]
 is a projective $R$-module. 
Since $R$ is a noetherian local ring, we only need to check 
\[
{\rm Tor}_{1}^{R}(X,R/\fm) = 0
\] 
by the local criterion for flatness. 
Since $P^{1}$ is a flat $R$-module, we have 
\begin{align*}
{\rm Tor}_{1}^{R}(X,R/\fm) &= \ker \left(P^{0}\otimes_{R} R/\fm \longrightarrow P^{1}\otimes_{R} R/\fm\right) 
\\
&=  H^{0}({\bf R}\Gamma(G,T) \otimes^{\bL}_{R}R/\fm). 
\end{align*}
By Theorem~\ref{thm:pot}(i), 
we conclude 
\[
{\rm Tor}_{1}^{R}(X,R/\fm) = H^{0}({\bf R}\Gamma(G,T/\fm T)) = (T/\fm T)^{G} = 0. 
\]
\end{proof}

\begin{corollary}[{\cite[Corollary~A.7]{BS}}]\label{cor:H2vanish}
If  $H^{0}(G,T/\fm T) = H^{2}(G,T/\fm T) = 0$, 
then the $R$-module $H^{1}(G,T)$ is free and, 
for any ideal $J$ of $R$, the canonical homomorphism 
\[
H^{1}(G,T) \otimes_{R} R/J \longrightarrow H^{1}(G,T/JT)
\]
 is an isomorphism.  
\end{corollary}
\begin{proof}
Since $H^{2}(G,T/\fm T) = 0$, by Theorem~\ref{thm:pot}(ii), the module $H^{2}(G,T/JT)$ also vanishes for any ideal $J$ of $R$, and hence 
${\bf R}\Gamma(G,T/JT) \in D_{\rm parf}^{[1,1]}(R/J)$ by Corollary~\ref{cor:amp}.   
In particular, $H^{1}(G,T)$ is free. The second assertion follows from Theorem~\ref{thm:pot}(i). 
\end{proof}

\subsection{Selmer structures}\label{subsec:selmer-str}

In this subsection, we recall the definition of Selmer structures. The contents of this subsection are based on \cite{MRkoly}. 
In this subsection, we assume  that the residue field of $R$ is finite. 

\begin{definition}\label{selmer structure}
A Selmer structure $\cF$ on $T$ is a collection of the following data:
\begin{itemize}
\item a finite set $S(\cF)$ of places of $k$, including the set $S$, 
\item a choice of $R$-submodule $H^{1}_{\cF}(G_{k_{\fq}}, T) \subseteq H^{1}(G_{k_{\fq}}, T)$ for each $\fq \in S(\cF)$. 
\end{itemize}
Put 
\[
H^{1}_{\cF}(G_{k_{\fq}}, T) := H^{1}_{f}(G_{k_{\fq}}, T) := \ker \left( H^{1}(G_{k_{\fq}}, T) \longrightarrow H^{1}(\cI_\fq, T) \right) 
\] 
for each prime $\fq \not\in S(\cF)$. 
Here $k_{\fq}^{\rm ur}$ denotes the maximal unramified extension of $k_{\fq}$ and $\cI_{\fq} := \Gal(\overline{k}_{\fq}/k_{\fq}^{\rm ur})$. 
We call $H^{1}_{\cF}(G_{k_{\fq}}, T)$ the local condition of $\cF$ at a prime $\fq$ of $k$. 
\end{definition}

\begin{remark}
Since $p$ is an odd prime, we have $H^1(k_v, T) = 0$ for any infinite place $v$ of $k$. 
Thus we ignore local condition at any infinite place of $k$.  
\end{remark}

\begin{example}\label{exa:selmer-str}
Suppose that $R$ is $p$-torsion-free. 
\begin{itemize}
\item[(i)] We define a canonical Selmer structure $\cF_{\rm can}$ on $T$ by the following data:
\begin{itemize}
\item $S(\cF_{\rm can}) := S$,  
\item we define a local condition at a prime $\fq \nmid p$ by  
\begin{align*}
H^{1}_{\cF_{\rm can}}(G_{k_{\fq}}, T) := \ker \left( H^{1}(G_{k_{\fq}}, T) \longrightarrow H^{1}(\cI_\fq, T \otimes_{\bZ_p} \bQ_p) \right),  
\end{align*}
\item we define a local condition at a prime $\fp \mid p$ by  
\[
H^{1}_{\cF_{\rm can}}(G_{k_{\fq}}, T) := H^{1}(G_{k_{\fq}}, T).
\] 
\end{itemize}

\item[(ii)] We define a strict Selmer structure $\cF_{\rm str}$ on $T$ by the following data:
\begin{itemize}
\item $S(\cF_{\rm str}) := S$,  
\item we define a local condition at a prime $\fq \nmid p$ by  
\begin{align*}
H^{1}_{\cF_{\rm str}}(G_{k_{\fq}}, T) := \ker \left( H^{1}(G_{k_{\fq}}, T) \longrightarrow H^{1}(\cI_\fq, T \otimes_{\bZ_p} \bQ_p) \right),  
\end{align*}
\item we define a local condition at a prime $\fp \mid p$ by  
\[
H^{1}_{\cF_{\rm str}}(G_{k_{\fq}}, T) := 0. 
\] 
\end{itemize}

\item[(iii)] We define a Selmer structure $\cF_{S}$ on $T$ by the following data:
\begin{itemize}
\item $S(\cF_{S}) := S$,  
\item we define a local condition at a prime $\fq \in S$ by  
\begin{align*}
H^{1}_{\cF_{S}}(G_{k_{\fq}}, T) := 0. 
\end{align*}
\end{itemize}

\end{itemize}
\end{example}

\begin{remark}
Let $\cF$ be a Selmer structure on $T$ and $R \longrightarrow R'$ a surjective ring homomorphism. 
Then $\cF$ induces a Selmer structure on the $R'$-module $T \otimes_R R'$, that we will denote by $\cF_{R'}$. 
If there is no risk of confusion, then we write $\cF$ instead of $\cF_{R'}$. 
\end{remark}

\begin{definition}
Let $\cF$ be a Selmer structure on $T$. 
Set
\begin{align*}
H^1_{/\cF}(G_{k_{\fq}}, T) := H^1(G_{k_{\fq}}, T)/H^1_{\cF}(G_{k_{\fq}}, T)
\end{align*} 
for any prime $\fq$ of $k$. 
We define the Selmer group $H_{\cF}^{1}(k, T) \subseteq H^{1}(G_{k,S(\cF)}, T)$ associated with $\cF$ 
to be the kernel of the direct sum of localization maps: 
\begin{align*}
H_{\cF}^{1}(k,T) := \ker \left( H^{1}(G_{k, S(\cF)},T) \longrightarrow \bigoplus_{\fq \in S(\cF)} 
H^{1}_{/\cF}(G_{k_{\fq}}, T) \right), 
\end{align*}
where $\fq$ runs through all the primes of $k$. 
\end{definition}

For a topological $\bZ_{p}$-module $M$, let $M^{\vee} := \Hom_{\rm cont}(M, \bQ_{p}/\bZ_{p})$ denote the Pontryagin duality of $M$. 
For any integer $n > 0$, let $\mu_{p^{n}}$ denote the group of all $p^{n}$-th roots of unity. 
The Cartier dual of $T$ is defined by 
\[
T^{\vee}(1) := T^{\vee} \otimes_{\bZ_{p}}\varprojlim_{n}\mu_{p^{n}} = \Hom_{\rm cont}(T, \bQ_{p}/\bZ_{p}) \otimes_{\bZ_{p}}\varprojlim_{n}\mu_{p^{n}}. 
\] 
Then we have the local Tate paring
\begin{align*}
\langle \ ,\  \rangle_{\fq} \colon H^{1}(G_{k_{\fq}}, T) \times H^{1}(G_{k_{\fq}}, T^{\vee}(1)) \longrightarrow \bQ_{p}/\bZ_{p} 
\end{align*}
for each prime $\fq$ of $k$. 

\begin{definition}\label{def:dual-sel}
Let $\cF$ be a Selmer structure on $T$. 
Put 
\begin{align*}
H^{1}_{\cF^{*}}(G_{k_{\fq}},  T^{\vee}(1)) 
:= \{ x \in H^{1}(G_{k_{\fq}}, T^{\vee}(1)) \mid \langle y ,x  \rangle_{\fq} = 0 
\text{ for any } y \in H^{1}_{\cF}(G_{k_{\fq}}, T)\}. 
\end{align*}
In this manner, the Selmer structure $\cF$ on $T$ gives rise to a Selmer structure $\cF^*$ on $T^{\vee}(1)$. 
Hence we obtain the dual Selmer group $H^{1}_{\cF^{*}}(k, T^{\vee}(1))$ associated with $\cF$; 
\begin{align*}
H^{1}_{\cF^{*}}(k, T^{\vee}(1)) := \ker \left( H^{1}(G_{k,S(\cF)}, T^{\vee}(1)) 
\longrightarrow \bigoplus_{\fq \in S(\cF)} H^{1}_{/\cF^*}(G_{k_{\fq}}, T^{\vee}(1)) \right).  
\end{align*}
\end{definition}

In this paper, we refer to the following theorem as global duality. 

\begin{theorem}[{\cite[Theorem~2.3.4]{MRkoly}}]\label{pt}
Let $\cF_{1}$ and $\cF_2$ be Selmer structures on $T$ satisfying 
$H^{1}_{\cF_{1}}(k_{\fq}, T) \subseteq H^{1}_{\cF_{2}}(k_{\fq}, T)$ for any prime $\fq$ of $k$, 

Then we have  the following exact sequence: 
\begin{align*}
0 \longrightarrow H^{1}_{\cF_{1}}(k, T) \longrightarrow H^{1}_{\cF_{2}}(k, T) 
\longrightarrow \bigoplus_{\fq} H^{1}_{\cF_{2}}(G_{k_{\fq}}, T)/H^{1}_{\cF_{1}}(G_{k_{\fq}}, T) 
\\
\longrightarrow H^{1}_{\cF_{1}^{*}}(k, T^{\vee}(1))^\vee \longrightarrow H^{1}_{\cF_{2}^{*}}(k, T^{\vee}(1))^\vee \longrightarrow 0, 
\end{align*}
where $\fq$ runs through all the primes of $k$ which satisfies 
$H_{\cF_{1}}^{1}(G_{k_{\fq}}, T) \neq H^{1}_{\cF_{2}}(G_{k_{\fq}}, T)$. 
\end{theorem}

\section{Exterior bi-dual}\label{sec:bi-dual}

We fix a  noetherian commutative ring $R$. For each $R$-module $M$ we set
\[ 
M^{*} :=\Hom_R(M,R). 
\]

\subsection{Definition of the exterior bi-dual} 

Let $M$ and $N$ be $R$-modules, and $r \geq 0$ an integer.  
For a map $\varphi \colon M \longrightarrow N$, there exists a unique homomorphism of $R$-modules
\[
{{\bigwedge}}_R^{r+1} M \longrightarrow N \otimes_{R} {{\bigwedge}}_R^{r} M
\]
with the property that
\[
m_1\wedge\cdots\wedge m_{r+1} \mapsto \sum_{i=1}^{r+1} (-1)^{i+1} \varphi(m_i) \otimes  m_1\wedge\cdots\wedge m_{i-1}\wedge m_{i+1} \wedge \cdots \wedge m_{r+1}
\]
for each subset $\{m_i\}_{1\le i\le r+1}$ of $M$. 
We will also denote this map by $\varphi$. 
In particular, for integers $0 \leq r \leq s$, we have  the canonical homomorphism
\[
{{\bigwedge}}_R^{r} M^* \longrightarrow \Hom_R\left({{\bigwedge}}_R^s M, {{\bigwedge}}_R^{s-r} M\right);  \ \varphi_1\wedge \cdots \wedge \varphi_r \mapsto \varphi_r \circ \cdots \circ \varphi_1.
\]

\begin{definition}\label{def exterior bidual}
We define the $r$-th exterior bi-dual ${{\bigcap}}_R^r M$ of $M$ by  
\[
{{\bigcap}}_R^r M :=\left({{\bigwedge}}_R^r M^\ast \right)^\ast.
\]
\end{definition}

Note that there is a canonical homomorphism
\[
\xi_M^r: {{\bigwedge}}_R^r M \longrightarrow {{\bigcap}}_R^r M; \ m \mapsto (\Phi \mapsto \Phi(m)),
\]
which is neither injective nor surjective in general.
However, if $M$ is a finitely generated projective $R$-module, then $\xi_M^r$ is an isomorphism.

\begin{definition}
For integers $0 \leq r \leq s$ and $\Phi \in {{\bigwedge}}_R^r M^\ast$, define an $R$-homomorphism
\begin{eqnarray} \label{map bidual}
{{\bigcap}}_R^s M \longrightarrow {{\bigcap}}_R^{s-r} M
\end{eqnarray}
as the $R$-dual of the $R$-homomorphism 
\[
{{\bigwedge}}_R^{s-r} M^\ast \longrightarrow {{\bigwedge}}_R^{s} M^\ast; \ \Psi \mapsto \Phi\wedge \Psi.\]
We denote the map (\ref{map bidual}) also by $\Phi$. 
\end{definition}
One can check that the diagram 
\[
\xymatrix{
{{\bigwedge}}_R^s M \ar[r]^-{\Phi} \ar[d]_{\xi_M^s} & {{\bigwedge}}_R^{s-r}M \ar[d]^{\xi_M^{s-r}}
\\
{{\bigcap}}_R^s M \ar[r]^-{\Phi} & {{\bigcap}}_R^{s-r} M
}
\]
is commutative. 

\subsection{Basic properties}
In this subsection, suppose that 
\begin{itemize}
\item $R$ is a zero dimensional Gorenstein local ring. 
\end{itemize}
Hence all free $R$-modules are injective by the definition of $R$.  
In particular, the functor $(-)^* = \Hom_R(-, R)$ is exact. 
Furthermore, Matlis duality tells us that 
the canonical map 
\[
\xi_{M}^{1} \colon M \longrightarrow {\bigcap}^1_R M
\] 
is an isomorphism for any finitely generated $R$-module $M$.

\begin{lemma}[{\cite[Lemma~2.1]{sakamoto}}]\label{lemma:base}
Let $F \stackrel{h}{\longrightarrow} M \stackrel{g}{\longrightarrow} N \longrightarrow 0$ 
be an exact sequence of $R$-modules and $r \geq 0$ an integer. 
If $F$ is a free $R$-module of rank $s \leq r$, then there is a unique $R$-homomorphism 
\begin{align*}
\varphi \colon{\bigwedge}^{r-s}_R N \otimes_{R}  \det(F)  \longrightarrow {\bigwedge}^{r}_R M
\end{align*}
such that $\varphi\left(\wedge^{r-s}g(b) \otimes a \right) = \left( \wedge^{s}h (a) \right) \wedge b$ for any $a \in \det(F)$ and $b \in {\bigwedge}^{r-s}_R M$. 
\end{lemma}
%\begin{proof} 
%Since $g$ is surjective and ${\bigwedge}^{s+1}_R F = 0$, the map $\phi$ is well-defined and characterized by the relation $\phi\left( a \otimes \wedge^{r-s}g(b) \right) = \left( \wedge^{s}h (a) \right) \wedge b$. 
%\end{proof}

%\begin{lemma}\label{commlem}
%Let $r \geq 0$ be an integer. 
%Suppose that we have the following commutative diagram of $R$-modules:
%\begin{align*}
%\xymatrix{ 
%F \ar[r]^{h} \ar[d]^{\alpha} & X \ar[r]^{g} \ar[d]^{\beta} & Y \ar[r] \ar[d]^{\gamma} & 0
%\\
%F' \ar[r]^{h'} & X' \ar[r]^{g'} & Y' \ar[r] & 0, 
%}
%\end{align*}
%where the horizontal rows are exact, and both $F$ and $F'$ are free $R$-modules of rank $s \leq r$. 
%Then the following diagram is commutative: 
%\begin{align*}
%\xymatrix{
%\det(F) \otimes_{R} {\bigwedge}^{r-s}_R N \ar[r]^-{\phi} \ar[d]^{\wedge^{s} \alpha \otimes \wedge^{r-s}\gamma} 
%& {\bigwedge}^{r}_R M \ar[d]^{\wedge^r \beta} 
%\\
%\det(F') \otimes_{R} {\bigwedge}^{r-s}_R N'  \ar[r]^-{\phi'} & {\bigwedge}^{r}_R M',  
%}
%\end{align*}
%where $\phi$ and $\phi'$ are the maps defined in Lemma~\ref{base}. 
%\end{lemma}
%\begin{proof}
%Let $a \in \det(F)$ and $b \in {\bigwedge}^{r-s}_R M$. 
%By Lemma~\ref{base}, we have 
%\begin{align*}
%(\phi' \circ (\wedge^{s}\alpha \otimes \wedge^{r-s} \gamma))(a \otimes \wedge^{r-s}g(b)) 
%&= \phi'( \wedge^{s} \alpha(a) \otimes \wedge^{r-s} (g' \circ \beta)(b))
%\\
%&=  (\wedge^{s} (h' \circ \alpha)(a)) \wedge (\wedge^{r-s} \beta(b))
%\\
%&= \wedge^{r}\beta ((\wedge^s h(a)) \wedge b)
%\\
%&= (\wedge^{r}\beta \circ \phi)(a \otimes \wedge^{r-s}g(b)). 
%\end{align*}
%This completes the proof. 
%\end{proof}

For $i \in \{1, 2\}$, let $M_{i}$ be an $R$-module and $F_{i}$ a free $R$-module of rank $s_{i}$. 
Suppose that we have  the following cartesian diagram:
\begin{align*}
\xymatrix{
M_{1} \ar@{^{(}->}[r] \ar[d] & M_{2} \ar[d] 
\\
F_{1} \ar@{^{(}->}[r] &  F_{2}, 
}
\end{align*}
where the horizontal arrows are injective. 
The $R$-module $F_2/F_1$ is free of rank $s_2 - s_1$ since $F_{1}$ is an injective $R$-module. 
Applying Lemma~\ref{lemma:base} to the exact sequence 
$(F_{2}/F_1)^* \longrightarrow M_2^* \longrightarrow M_1^* \longrightarrow 0$, we obtain an $R$-homomorphism 
\begin{align*}
\widetilde{\Phi} \colon   {\bigcap}^r_R M_2 \otimes_{R} \det((F_2/F_1)^*) \longrightarrow {\bigcap}^{r-s_2+s_1}_R M_1
\end{align*}
for any integer $r \geq s_2 - s_1$. 
%Therefore we obtain the following map which play an important role in the theories of Kolyvagin and Stark systems. 

\begin{definition}[{\cite[Definition~2.3]{sakamoto}}]\label{def:map}
For any cartesian diagram as above and an integer $r \geq s_2 - s_1$, we have  an $R$-homomorphism 
\begin{align*}
\Phi \colon   {\bigcap}^r_R M_2 \otimes_{R} \det(F_2^*) 
&\stackrel{\sim}{\longrightarrow} {\bigcap}^r_R M_2  \otimes_{R}  \det((F_2/F_1)^*) \otimes_R \det(F_1^*) 
\\
&\longrightarrow  {\bigcap}^{r-s_2+s_1}_R M_1 \otimes_{R} \det(F_1^*), 
\end{align*}
where the first map is induced by the isomorphism 
\[
\det((F_2/F_1)^*) \otimes_R \det(F_1^*) \xrightarrow{\sim} \det(F_2^*); a \otimes b \longrightarrow a \wedge \widetilde{b}, 
\] 
where $\widetilde{b}$ is a lift of $b$ in ${\bigwedge}^{s_1}_R F_2^*$ and the second map is induced by $\widetilde{\Phi}$. 
\end{definition}

%The map $\Phi$ is a generalization of the map defined in \cite[Proposition~A.1]{MRkoly}. 
%Furthermore, we have the following proposition which is a generalization of \cite[Proposition~A.2]{MRkoly} 
%to zero dimensional Gorenstein local rings. 

\begin{proposition}[{\cite[Proposition~2.4]{sakamoto}}]\label{prop:homlem}
Suppose that we have  the following commutative diagram of $R$-modules
\begin{align*}
\xymatrix{
M_{1} \ar@{^{(}->}[r] \ar[d] & M_{2} \ar[d] \ar@{^{(}->}[r] & M_3 \ar[d] 
\\
F_{1} \ar@{^{(}->}[r] &  F_{2} \ar@{^{(}->}[r] & F_3
}
\end{align*}
such that $F_1$, $F_2$, and $F_3$ are free of finite rank and the two squares are cartesian. 
Let $s_i = {\rm rank}_R (F_i)$ and  $r \geq s_3 - s_1$.  
For $i,j \in \{1,2,3\}$ with $j < i$, we denote by
\begin{align*}
\Phi_{ij} \colon {\bigcap}^{r-s_3 +s_i}_R M_{i} \otimes_{R}  {\rm det}(F_{i}^*) \longrightarrow 
 {\bigcap}^{r-s_3 + s_j}_R M_{j} \otimes_{R} {\rm det}(F_{j}^*)
\end{align*}
the map given by Definition~\ref{def:map}. Then we have  
\begin{align*}
\Phi_{31} = \Phi_{21} \circ \Phi_{32}. 
\end{align*}
\end{proposition}

\subsection{Base change}
In this subsection, we slightly generalize the results proved in \cite[Appendix~B]{BS}. 
Recall that $R$ is a noetherian commutative ring. Let $M$ be a finitely generated $R$-module.  

\begin{definition}\ 
\begin{itemize}
\item[(i)] The ring $R$ is said to satisfy the condition (G$_{1}$) if the local ring $R_{\fr}$ is Gorenstein for any  prime $\fr$ of $R$ with ${\rm ht}(\fr) \leq 1$. 
\item[(ii)] For an  integer $n \geq 0$, we say that the $R$-module $M$ satisfies the Serre's condition (S$_{n}$) if the inequality ${\rm depth}(M_{\fr}) \geq \min\{n, \dim{\rm Supp}(M_{\fr})\}$ holds for all prime ideal $\fr$
 of $R$. Here $M_{\fr} := R_{\fr} \otimes_{R} M$. 
 \item[(iii)] The $R$-module $M$ is said to be $2$-syzygy if there exists an exact sequence $0 \longrightarrow M \longrightarrow  P^{0} \longrightarrow P^{1}$ of $R$-modules with $P^{i}$ projective. 
\end{itemize}
\end{definition}

\begin{remark}\label{rem:equiv}
A noetherian ring satisfying the conditions (G$_{1}$) and (S$_{2}$) has the following characterization (proved by Evans--Griffith and Matsui--Takahashi--Tsuchiya in \cite{MTT17}). 

The noetherian ring $R$ satisfies (G$_{1}$) and (S$_{2}$) if and only if, for any finitely generated $R$-module $M$, the following assertions are equivalent: 
\begin{itemize}
\item[(i)] $M$ satisfies (S$_{2}$);  
\item[(ii)]  $M$ is $2$-syzygy;  
\item[(iii)] $M$ is reflexive i.e., $M \longrightarrow M^{**}$ is an isomorphism. 
\end{itemize}
\end{remark}

\begin{remark}
Normal rings and Gorenstein rings satisfy (G$_{1}$) and (S$_{2}$). 
\end{remark}

\begin{remark}\label{rem:base change}
If $R'$ is a flat $R$-algebra, then we have the canonical isomorphism 
\[
R' \otimes_{R} {\bigcap}^{r}_{R}M \stackrel{\sim}{\longrightarrow} {\bigcap}^{r}_{R'}(R' \otimes_{R} M) 
\]
for any integer $r \geq 0$. 
However, in general, there dose not exist a canonical homomorphism from 
$R' \otimes_{R} {\bigcap}^{r}_{R}M$ to ${\bigcap}^{r}_{R'}(R' \otimes_{R} M)$. 
We will show the existence of such a homomorphism if $R$ and $R'$ satisfy (G$_{1}$) and (S$_{2}$) (see Lemma~\ref{lem:reduction}). 
\end{remark}

\begin{definition}
%Suppose that the Krull dimension $\dim(R)$ of $R$ is at least $1$. 
We say that the finitely generated $R$-module $M$ is pseudo-null if 
$M_{\fr}$ vanishes for any prime $\fr$ of $R$ with ${\rm ht}(\fr) \leq 1$. 
When $\dim(R) \leq 1$,  the module $M$ is pseudo-null if and only if $M = 0$. 
\end{definition}

%\begin{lemma}\ 
%\begin{itemize}
%\item[(i)] 
%There is a finitely generated $R$-module $\mathrm{Tr}(M)$ such that  
%\[
%\coker\left(M \longrightarrow M^{**} \right) = \Ext_{R}^{2}(\mathrm{Tr}(M),R). 
%\]
%\item[(ii)] If $R$ satisfies (G$_{0}$) and (S$_{1}$), then we have  
%\[
%\ker\left(M \longrightarrow M^{**} \right) = \ker\left(M \longrightarrow Q \otimes_{R} M \right). 
%\]
%\end{itemize}
%\end{lemma}
%\begin{proof}
%Let us show claim~(i). Take an exact sequence $R^{a} \longrightarrow R^{b} \longrightarrow M \longrightarrow 0$ and set $\mathrm{Tr}(M) := \coker\left((R^{b})^{*} \longrightarrow (R^{a})^{*}\right)$. 
%Since $\ker\left((R^{b})^{*} \longrightarrow (R^{a})^{*}\right) = M^{*}$, we get an exact sequence 
%\[
%R^{b} \longrightarrow M^{**} \longrightarrow \Ext_{R}^{2}(\mathrm{Tr}(M),R) \longrightarrow 0, 
%\]
%which implies claim~(i). 
%
%If $R$ satisfies (G$_{0}$) and (S$_{1}$), the total quotient ring $Q$ of $R$ is a zero dimensional Gorenstein ring. 
%Hence Matlis duality shows that $Q \otimes_{R} M \longrightarrow Q \otimes_{R} M^{**}$ is an isomorphism. 
%Since $M^{**} \longrightarrow Q \otimes_{R} M^{**}$ is injective, we have  
%\[
%\ker\left(M \longrightarrow M^{**} \right) = 
%\ker\left(M \longrightarrow Q \otimes_{R} M^{**} \right)
%= \ker\left(M \longrightarrow Q \otimes_{R} M \right). 
%\]
%\end{proof}

\begin{lemma}\label{lem:ext vanish}
If $R$ satisfies (S$_{2}$) and $M$ is a pseudo-null $R$-module, then we have  
\[
\Hom_{R}(M,R) = \Ext_{R}^{1}(M,R) = 0.
\] 
\end{lemma}
\begin{proof}
We may assume $\dim(R) > 1$. 
Since $R$ is a noetherian ring, there is a sequence of submodules $0 = M_{0} \subseteq M_{1} \subseteq \cdots \subseteq M_{n} = M$ such that $M_{i+1}/M_{i}$ is isomorphic to $R/\fr$ for some prime ideal $\fr$ of $R$ for each integer $0 \leq i < n$. 
Hence we may assume that $M = R/\fr$ for some prime ideal $\fr$ of $R$. 
Since $R/\fr$ is pseudo-null, we have  ${\rm ht}(\fr) \geq 2$. 

There is an exact sequence 
\[
0 \longrightarrow \Hom_{R}(R/\fr,R) \longrightarrow R \longrightarrow \Hom_{R}(\fr,R) \longrightarrow \Ext_{R}^{1}(R/\fr,R) \longrightarrow 0. 
\]
The ring $R$ has no embedded primes since $R$ satisfies (S$_{1}$). 
Let $Q$ denote the total ring of fractions of $R$. 
we have  $\fr \cdot Q = Q$ since ${\rm ht}(\fr) \geq 2$ and $R$ has no embedded primes. 
Hence we conclude that $\Hom_{R}(R/\fr,R)$ vanishes and there is a canonical identification 
\[
\Hom_{R}(\fr,R) = \{f \in Q \mid f \cdot \fr \subseteq R \}. 
\]
It suffices to show that 
\[
R = \{f \in Q \mid f \cdot \fr \subseteq R \}. 
\]
Let $f \in Q$ satisfying $f \cdot \fr \subseteq R$. 
There is a regular element $x \in R$ with $y := xf \in R$. 
Let us show $y \in xR$. 
Let $xR = \bigcap_{i=1}^{n}I_{i}$ be a minimal primary decomposition of $xR$. 
Since $x$ is a regular element and $R$ satisfies (S$_{2}$), the ring $R/xR$ satisfies (S$_{1}$) and hence $R/xR$ has no embedded primes. 
This fact implies the prime ideal $\sqrt{I_{i}}$ has height one. 
Since ${\rm ht}(\fr) \geq 2$, we have  
\[
y \in y R_{\sqrt{I_{i}}} = y \fr R_{\sqrt{I_{i}}} = xR_{\sqrt{I_{i}}} \subseteq I_{i}R_{\sqrt{I_{i}}}. 
\]
As $I_{i}$ is a $\sqrt{I_{i}}$-primary ideal, we see $y \in I_{i}$. 
Thus we conclude that $y \in \bigcap_{i=1}^{n}I_{i} = xR$. 
\end{proof}

\begin{lemma}\label{lem:expre}
Let $0 \longrightarrow M \longrightarrow P^{1} \stackrel{\alpha}{\longrightarrow} P^{2}$ 
be an exact sequence of $R$-modules, where $P^{1}$ and $P^2$ are free $R$-modules of finite rank. 
Put $I = {\rm im}(P^{1} \longrightarrow P^{2})$. 
If $R$ satisfies (G$_{1}$) and (S$_{2}$), then the canonical map 
\[
{\bigcap}^{r}_{R}M \longrightarrow {\bigcap}^{r}_{R}P^{1} \stackrel{\sim}{\longleftarrow} {\bigwedge}^{r}_{R}P^{1}
\]
induces an isomorphism 
\[
{\bigcap}^{r}_{R}M \stackrel{\sim}{\longrightarrow} \ker\left({\bigwedge}^{r}_{R}P^{1} \stackrel{\alpha}{\longrightarrow} I \otimes_{R} {\bigwedge}^{r-1}_{R}P^{1}\right). 
\]
for any integer $r \geq 1$. 
\end{lemma}

\begin{remark}
When $R$ is Gorenstein, this lemma has already been proved in \cite[Lemma~B.2]{BS}. 
\end{remark}

\begin{proof}
Since $P^{2}$ is a projective $R$-module, we have an isomorphism 
\[
{\rm Ext}^{1}_{R}(I,R) \stackrel{\sim}{\longrightarrow} {\rm Ext}^{2}_{R}(P^{2}/I,R). 
\]
Hence, since $P^{1}$ is a projective $R$-module, by taking $R$-duals to the exact sequence $0 \longrightarrow M \longrightarrow P^{1} \longrightarrow I \longrightarrow 0$, we obtain  an exact sequence of $R$-modules
\[
0 \to I^{*} \longrightarrow (P^{1})^{*} \longrightarrow M^{*} \longrightarrow 
{\rm Ext}^{2}_{R}(P^{2}/I,R) \longrightarrow  0. 
\] 
Put $Y := \im\left((P^{1})^{*} \longrightarrow M^{*}\right)$. 
This exact sequence gives rise to the exact sequence 
\[
I^{*} \otimes_{R} {\bigwedge}^{r-1}_{R} (P^{1})^{*} \longrightarrow {\bigwedge}^{r}_{R}(P^{1})^{*} \longrightarrow {\bigwedge}^{r}_{R}Y \longrightarrow 0. 
\]
Passing to $R$-duals once again and using the canonical identification ${\bigwedge}^{r}_{R}P^{1} = {\bigcap}^{r}_{R}P^{1}$, we get an exact sequence 
\[
0 \longrightarrow \left({\bigwedge}^{r}_{R}Y \right)^{*} \longrightarrow {\bigwedge}^{r}_{R}P^{1} \longrightarrow \left(I^{*} \otimes_{R} {\bigwedge}^{r-1}_{R}(P^{1})^{*}\right)^{*}. 
\]
Since $I$ is a submodule of $P^{2}$, the canonical map $I \longrightarrow I^{**}$ is injective, 
and hence we have the commutative diagram 
\[
\xymatrix{
{\bigwedge}^{r}_{R}P^{1} \ar[r]^-{\alpha} \ar[d] &  I \otimes_{R} {\bigwedge}^{r-1}_{R}P^{1} \ar@{^{(}->}[d]
\\
\left(I^{*} \otimes_{R} {\bigwedge}^{r-1}_{R}(P^{1})^{*}\right)^{*}  &  I^{**} \otimes_{R} {\bigcap}^{r-1}_{R}P^{1} \ar[l]_-{\cong}, 
}
\]
where the bottom arrow is the composite map  
\begin{align*}
I^{**} \otimes_{R} {\bigcap}_{R}^{r-1}P^{1} &= \Hom_{R}(I^{*}, R) \otimes_{R} {\bigcap}_{R}^{r-1}P^{1} 
\\ 
&\stackrel{\sim}{\longrightarrow} \Hom_{R}\left(I^{*},{\bigcap}_{R}^{r-1}P^{1}\right) 
\\
&= \left(I^{*} \otimes_{R} {\bigwedge}^{r-1}_{R}(P^{1})^{*}\right)^{*}. 
\end{align*}
Thus we conclude
\[
\left({\bigwedge}^{r}_{R}Y \right)^{*} = \ker\left({\bigwedge}^{r}_{R}P^{1} \longrightarrow I \otimes_{R} {\bigwedge}^{r-1}_{R}P^{1}\right). 
\]
%To conclude with the proof of this lemma, it suffices to show that 
%a canonical map 
%\[
%{\bigcap}^{r}_{R}H^{1}(G,T) %\stackrel{\sim}{\longrightarrow}  
%\longrightarrow \left({\bigwedge}^{r}_{R}Y\right)^{*}
%\]
%is an isomorphism. 
We remark that ${\rm Ext}^{2}_{R}(P^{2}/I,R)$ is pseudo-null since $R$ satisfies (G$_{1}$).   Hence the kernel and cokernel of the homomorphism 
\[
{\bigwedge}^{r}_{R}Y \longrightarrow {\bigwedge}^{r}_{R}M^{*}
\]
are also pseudo-null.  
By taking $R$-duals to the homomorphism ${\bigwedge}^{r}_{R}Y \longrightarrow {\bigwedge}^{r}_{R}M^{*}$, Lemma~\ref{lem:ext vanish} shows that the map 
\[
{\bigcap}^{r}_{R}M \stackrel{\sim}{\longrightarrow}  \left({\bigwedge}^{r}_{R}Y\right)^{*}
\]
is an isomorphism, which completes the proof. 
\end{proof}

\begin{lemma}\label{lem:reduction}
Let $R \longrightarrow R'$ be a ring homomorphism of noetherian rings satisfying (G$_{1}$) and (S$_{2}$). 
Let $C \in \bigcup_{n>0}D^{[1,n]}_{\rm perf}(R)$. 
Then for any integer $r \geq 0$, there is a canonical homomorphism 
\[
{\bigcap}^{r}_{R}H^{1}(C) \longrightarrow {\bigcap}^{r}_{R'}H^{1}(C \otimes^{\mathbb{L}}_{R} R')
\]
\end{lemma}
\begin{proof}
This lemma follows from Lemma~\ref{lem:expre}. 
\end{proof}

\begin{lemma}\label{lem:inv-bidual}
Let $J$ be a directed poset and $\{R_{j}\}_{j \in J}$ an inverse system of noetherian rings satisfying (G$_{1}$) and (S$_{2}$). 
Set $R := \varprojlim_{j \in J}R_{j}$ and let 
\[
C = [ \ \cdots \longrightarrow 0 \longrightarrow P^{1} \longrightarrow P^{2} \longrightarrow P^{3} \longrightarrow \cdots \ ] \in \bigcup_{n>0}D^{[1,n]}_{\rm perf}(R)
\] 
such that $P^{i}$ is a finitely generated projective $R$-module. 
If $R$ is a noetherian ring satisfying (G$_{1}$) and (S$_{2}$),  
then, for any integer $r \geq 0$, 
the canonical maps 
\[
{\bigcap}^{r}_{R}H^{1}(C) \longrightarrow {\bigcap}^{r}_{R_{j}}H^{1}(C \otimes^{\mathbb{L}}_{R} R_{j})
\] 
for $j \in J$ induce an isomorphism 
\[
{\bigcap}^{r}_{R}H^{1}(C) \stackrel{\sim}{\longrightarrow} 
\varprojlim_{j \in J}{\bigcap}^{r}_{R_{j}}H^{1}(C \otimes^{\mathbb{L}}_{R} R_{j}). 
%\stackrel{\sim}{\longrightarrow} \ker\left({\bigwedge}^{r}_{R}P^{1} \longrightarrow I \otimes_{R} {\bigwedge}^{r-1}_{R}P^{1}\right).
\]
\end{lemma}
\begin{proof}
For $j \in J$ and $i \in \{1,2\}$, put $P^{i}_{j} := P^{i} \otimes_{R} R_{j}$ and  $I_{j} := {\rm im}(P^{1}_{j} \longrightarrow P^{2}_{j})$. 
We also set  $I := {\rm im}(P^{1} \longrightarrow P^{2})$. 
Since $P^{1}_{j}$ is a free $R_{j}$-module, by Lemma~\ref{lem:expre}, we have  
\begin{align*}
\varprojlim_{j \in J}{\bigcap}^{r}_{R}H^{1}(C \otimes^{\mathbb{L}}_{R} R_{j}) 
&=  \varprojlim_{j \in J}\ker\left({\bigwedge}^{r}_{R_{j}}P^{1}_{j} \longrightarrow I_{j} \otimes_{R_{j}} {\bigwedge}^{r-1}_{R_{j}}P^{1}_{j}\right)
\\
%\varprojlim_{i \in I}\ker\left({\bigwedge}^{r}_{R_{i}}P^{1}_{i} \longrightarrow I_{i} \otimes_{R_{i}} {\bigwedge}^{r-1}_{R_{i}}P^{1}_{i}\right) 
&= 
\ker\left({\bigwedge}^{r}_{R}P^{1} \longrightarrow \left(\varprojlim_{j \in J}I_{j} \right)\otimes_{R} {\bigwedge}^{r-1}_{R}P^{1}\right)
\end{align*}
Since $I_{j}$ is a submodule of $P^{2}_{j}$, the canonical map $I \longrightarrow \varprojlim_{j\in J}I_{j}$ is injective. This shows that
\[
\ker\left({\bigwedge}^{r}_{R}P^{1} \longrightarrow I \otimes_{R} {\bigwedge}^{r-1}_{R}P^{1}\right)
= \ker\left({\bigwedge}^{r}_{R}P^{1} \longrightarrow \left(\varprojlim_{j \in J}I_{j} \right)\otimes_{R} {\bigwedge}^{r-1}_{R}P^{1}\right)
\]
and completes the proof  by Lemma~\ref{lem:expre}. 
\end{proof}

\section{Characteristic ideal}\label{sec:char}

Let $R$ be a noetherian ring and $M$ a finitely generated $R$-module. 
We write $Q$ for the total ring of fractions of $R$.

\begin{lemma}\label{lem:sub}
Suppose that $R$ satisfies (G$_{0}$) and (S$_{1}$). 
For an $R$-submodule $N$ of $M$ and an integer $r>0$, the canonical homomorphism 
\[
{\bigcap}^{r}_{R}N \longrightarrow {\bigcap}^{r}_{R}M
\]
is injective. 
\end{lemma}
\begin{proof}
By taking $R$-duals to the tautological exact sequence $0 \to N \longrightarrow M \longrightarrow M/N \longrightarrow 0$, we obtain an exact sequence of $R$-modules 
\[
0 \to (M/N)^{*} \longrightarrow M^{*} \longrightarrow N^{*} \longrightarrow \Ext_{R}^{1}(M/N,R). 
\]
Since $R$ satisfies (G$_{0}$) and (S$_{1}$), the ring $Q$ is Gorenstein with $\dim(Q) = 0$. 
Hence the module $Q \otimes_{R} \Ext_{R}^{1}(M/N,R) = \Ext_{Q}^{1}(Q \otimes_{R} M/N,Q)$ vanishes, and  we have a commutative diagram 
\[
\xymatrix{
{\bigcap}^{r}_{R}N \ar[r] \ar[d] &  {\bigcap}^{r}_{R}M \ar[d] 
\\
Q \otimes_{R} {\bigcap}^{r}_{R}N \ar@{^{(}->}[r] & Q \otimes_{R} {\bigcap}^{r}_{R}M.  
}
\]
As ${\bigcap}^{r}_{R}N = \left({\bigwedge}^{r}_{R}N^{*}\right)^{*}$, any $R$-regular element is ${\bigcap}^{r}_{R}N$-regular. Hence the left vertical arrow is injective, which completes the proof. 
\end{proof}

\begin{remark}
When $R$ satisfies (G$_{0}$) and (S$_{1}$), for an ideal $I$ of $R$, the canonical homomorphism $I^{**} \longrightarrow R^{**} = R$ is injective by Lemma~\ref{lem:sub}. 
Therefore, in this case, one can regard $I^{**}$ as an ideal of $R$ and $I$ as a submodule of $I^{**}$. 
\end{remark}

\begin{definition}\label{def:char}
Take an integer $r>0$ and an exact sequence of $R$-modules 
\[
0 \longrightarrow N \longrightarrow R^{r} \longrightarrow M \longrightarrow 0. 
\]
We then define the characteristic ideal ${\rm char}_{R}(M)$ of $M$ by  
\[
{\rm char}_{R}(M) := \im\left({\bigcap}^{r}_{R}N \longrightarrow {\bigcap}^{r}_{R}R^{r} = R \right). 
\]
\end{definition}

\begin{remark}\label{rem:flat}
If $R'$ is a flat $R$-algebra, then we have  ${\rm char}_{R}(M)R' = {\rm char}_{R'}(R' \otimes_{R}M)$. 
\end{remark}

\begin{remark}\label{rem:indep-char}
%The uniqueness of minimal free resolutions of finitely generated modules over local noetherian rings implies that 
%the ideal ${\rm char}_{R}(M)$ is independent of the choice of the exact sequence 
%$0 \longrightarrow N \longrightarrow R^{r} \longrightarrow M \longrightarrow 0$. 
The ideal ${\rm char}_{R}(M)$ is independent of the choice of the exact sequence
$0 \longrightarrow N \longrightarrow R^{r} \longrightarrow M \longrightarrow 0$. 

In fact, by Remark~\ref{rem:flat}, we may assume that $R$ is a local noetherian ring. 
Then any free resolution of $M$ is isomorphic to the direct sum of the minimal free resolution of $M$ and a trivial complex. 
Hence it suffices to show that 
\[
\im\left({\bigcap}^{r}_{R}N \longrightarrow {\bigcap}^{r}_{R}R^{r} = R \right) = 
\im\left({\bigcap}^{r+s}_{R}(N \oplus R^{s}) \longrightarrow {\bigcap}^{r+s}_{R}R^{r+s} = R \right).
\]
The homomorphism $(R^{r+s})^{*} = (R^{r})^{*} \oplus  (R^{s})^{*} \longrightarrow N^{*} \oplus (R^{s})^{*} = (N \oplus R^{s})^{*}$ induces a commutative diagram
\[
\xymatrix{
{\bigwedge}^{r}_{R}(R^{r})^{*} \otimes  \bigwedge^{s}_{R}(R^{s})^{*} \ar[r] \ar[d]^-{=} & 
{\bigwedge}^{r}_{R}(N)^{*} \otimes  {\bigwedge}^{s}_{R}(R^{s})^{*}  \ar@{^{(}->}[d] 
\\
 {\bigwedge}^{r+s}_{R}(R^{r+s})^{*}  \ar[r] & {\bigwedge}^{r+s}_{R}(N \oplus R^{s})^{*}.   
}
\]
By taking $R$-duals to this commutative diagram, we obtain the following commutative diagram: 
\[
\xymatrix{
{\bigcap}^{r+s}_{R}(N \oplus R^{s}) \ar[r] \ar@{->>}[d] & {\bigcap}^{r+s}_{R}R^{r+s} \ar[d]^-{=}
\\
{\bigcap}^{r}_{R}N \otimes {\bigcap}^{s}_{R}R^{s} \ar[r] & {\bigcap}^{r}_{R}R^{r} \otimes {\bigcap}^{s}_{R}R^{s}.   
}
\]
Hence  $\im\left({\bigcap}^{r}_{R}N \longrightarrow {\bigcap}^{r}_{R}R^{r} = R \right) = \im\left({\bigcap}^{r+s}_{R}(N \oplus R^{s}) \longrightarrow {\bigcap}^{r+s}_{R}R^{r+s} = R \right)$. 
\end{remark}

\begin{example}  
For an ideal $I$ of $R$, we have ${\rm char}_{R}(R/I) = {\rm im}\left(I^{**} \longrightarrow R^{**} = R\right)$. 
\end{example}

\begin{proposition}\label{prop:fitt-char}\
\begin{itemize}
\item[(i)] We have  ${\rm Fitt}_{R}^{0}(M) \subseteq {\rm char}_{R}(M)$. 
\item[(ii)] If the projective dimension of $M$ is at most $1$, then ${\rm Fitt}_{R}^{0}(M) = {\rm char}_{R}(M)$. 
\end{itemize}
\end{proposition}
\begin{proof}
Take an integer $r>0$, an exact sequence of $R$-modules 
\[
0 \longrightarrow N \longrightarrow R^{r} \longrightarrow M \longrightarrow 0, 
\]
and a surjection $R^{s} \longrightarrow N$. Then we have  a commutative diagram 
\[
\xymatrix{
{\bigwedge}^{r}_{R}R^{s} \ar[r] \ar[d] & {\bigwedge}^{r}_{R}R^{r} = R \ar[d]^-{=} 
\\
{\bigcap}^{r}_{R}N \ar[r] & {\bigcap}^{r}_{R}R^{r} = R. 
}
\]
Hence we conclude 
\begin{align*}
{\rm Fitt}_{R}^{0}(M) &= \im\left({\bigwedge}^{r}_{R}R^{s} \longrightarrow {\bigwedge}^{r}_{R}R^{r} = R\right) 
\subseteq \im\left({\bigcap}^{r}_{R}N \longrightarrow {\bigcap}^{r}_{R}R^{r} = R \right) 
= {\rm char}_{R}(M). 
\end{align*}
Furthermore, if the projective dimension of $M$ is at most $1$, then $N$ is a projective $R$-module. In this case, the canonical map ${\bigwedge}^{r}_{R}R^{s} \longrightarrow 
{\bigwedge}^{r}_{R}N = {\bigcap}^{r}_{R}N$ is surjective, and hence 
${\rm Fitt}_{R}^{0}(M) = {\rm char}_{R}(M)$. 
\end{proof}

\begin{proposition}\label{prop:artin-char}
If $R$ is a zero dimensional Gorenstein ring, then we have  
\[
{\rm char}_{R}(M) = \Ann_{R}(M). 
\] 
\end{proposition}
\begin{proof}
Take an integer $r > 0$ and an exact sequence of $R$-modules 
\[
0 \longrightarrow N \longrightarrow R^{r} \longrightarrow M \longrightarrow 0. 
\]
Since the functor $(-)^{*}$ is exact, we have  an exact sequence 
\[
M^{*} \otimes_{R} {\bigwedge}^{r-1}_{R}(R^{r})^{*} \longrightarrow {\bigwedge}^{r}_{R}(R^{r})^{*} 
\longrightarrow {\bigwedge}^{r}_{R}N^{*} \longrightarrow 0. 
\]
Note that $M = M^{**}$ by Matlis duality. 
By taking $R$-duals to this exact sequence, we obtain an exact sequence 
\[
0 \longrightarrow {\bigcap}^{r}_{R}N \longrightarrow R \longrightarrow 
\left(M^{*} \otimes_{R} {\bigwedge}^{r-1}_{R}(R^{r})^{*} \right)^{*} = 
M^{r}. 
\]
The homomorphism $R \longrightarrow M^{r}$ corresponds to the homomorphism $R^{r} \longrightarrow M$ 
under the canonical isomorphism $\Hom_{R}(R,M^{r}) \stackrel{\sim}{\longrightarrow} \Hom_{R}(R^{r},M)$. 
Hence we conclude 
\[
\mathrm{char}_{R}(M) = \Ann_{R}\left(\im(R \longrightarrow M^{r})\right) = \Ann_{R}(M). 
\]
\end{proof}

\begin{remark}
In general, the characteristic ideal is not additive on the (split) short exact sequence of finitely generated $R$-modules. 
We give two counter-examples. 

\begin{itemize}
\item[(i)] Suppose that $R$ is a zero dimensional Gorenstein ring. We then have 
\[
{\rm char}_{R}(R/I \oplus R/J) = I \cap J  
\]
for any ideals $I$ and $J$ of $R$ by Proposition~\ref{prop:artin-char}. 
Hence if $IJ \neq I \cap J$, we have 
\[
{\rm char}_{R}(R/I \oplus R/J) \neq {\rm char}_{R}(R/I){\rm char}_{R}(R/J). 
\]

\item[(ii)] Suppose that $R$ is a noetherian local ring. 
Take a regular element $r \in R$ and an ideal $I$ of $R$ with $r \in I$. 
We then have an exact sequence of $R$-modules
\[
0 \longrightarrow I/rR \longrightarrow R/rR \longrightarrow R/I \longrightarrow 0. 
\]
We note that ${\rm char}_{R}(R/rR) = rR \cong R$ since $r$ is a regular element. 
Therefore, if we have 
\[
{\rm char}_{R}(R/I){\rm char}_{R}(I/rR) = {\rm char}_{R}(R/rR) \cong R, 
\]
the ideal 
${\rm char}_{R}(R/I)$ is invertible. 
Since an invertible ideal is projective, the ideal  ${\rm char}_{R}(R/I)$ is principal. 
Hence we conclude that if ${\rm char}_{R}(R/I)$ is not principal, then we have 
\[
{\rm char}_{R}(R/rR) \neq {\rm char}_{R}(R/I){\rm char}_{R}(I/rR). 
\]
\end{itemize}
\end{remark}

We devote the rest of this appendix to proving the following important property of characteristic ideals.

\begin{theorem}\label{thm:inequality}
Let $R$ be a noetherian ring satisfying (G$_{1}$) and (S$_{2}$). 
Let $M$ be a finitely generated $R$-module. 
Then for any $R$-submodule $N$ of $M$, we have  
\[
{\rm char}_{R}(M) \subseteq {\rm char}_{R}(N). 
\]
\end{theorem}

\begin{lemma}\label{lem:reflexible}
Let $R$ be a noetherian ring satisfying (G$_{0}$) and (S$_{1}$).  
Let $M$ be a finitely generated $R$-module. 
Then the homomorphisms 
\[
\xi^{1}_{M^{*}} \colon M^{*} \longrightarrow M^{***} \,\, \textrm{ and } \,\,  (\xi^{1}_{M})^{*} \colon  M^{***} \longrightarrow M^{*} 
\]
are isomorphisms. 
\end{lemma}
\begin{proof}
The composite map  
\[
M^{*} \xrightarrow{\xi^{1}_{M^{*}}} M^{***}  \xrightarrow{(\xi^{1}_{M})^{*}}  M^{*}
\] 
is the identity map.
Hence it suffices to prove that the homomorphism $ (\xi^{1}_{M})^{*} \colon M^{***} \longrightarrow M^{*}$ is injective. 
To see this, we take an exact sequence $R^{a} \longrightarrow R^{b} \longrightarrow M \longrightarrow 0$ and 
set $C := {\rm coker}\left(M^{*} \hooklongrightarrow (R^{b})^{*}\right)$. 
We then have  the commutative diagram with exact rows: 
\[
\xymatrix{
& R^a \ar[r] \ar[d] & R^b \ar[r] \ar[d]^{\cong} &M \ar[r]\ar[d] & 0 & 
\\
0 \ar[r] & C^* \ar[r] & (R^b)^{**} \ar[r] & M^{**} \ar[r] & \Ext_R^1(C, R) \ar[r] & 0,  
}
\]
and it induces an exact sequence of $R$-modules 
\[
M \longrightarrow M^{**} \longrightarrow \Ext_{R}^{1}(C, R) \longrightarrow 0. 
\]
Since $Q \otimes_{R} \Ext_{R}^{1}(C, R) = \Ext_{Q}^{1}(Q \otimes_{R} C, Q) = 0$, 
the module $\Ext_{R}^{1}(C,R)^{*}$ also vanishes. In particular, the homomorphism $ (\xi^{1}_{M})^{*} \colon M^{***} \longrightarrow M^{*}$  is injective. 
\end{proof}

\begin{proposition}\label{prop:reflexible}
Let $R$ be a noetherian ring satisfying (G$_{0}$) and (S$_{1}$).  
Let $M$ be a finitely generated $R$-module. 
Then the characteristic ideal  ${\rm char}_{R}(M)$ is reflexive, that is, ${\rm char}_{R}(M) = {\rm char}_{R}(M)^{**}$. 
\end{proposition}
\begin{proof}
Take an exact sequence of $R$-modules 
\[
0 \longrightarrow N \longrightarrow R^{r} \longrightarrow M \longrightarrow 0
\]
with $r>0$. Then Lemmas~\ref{lem:sub} shows that we have the canonical isomorphism 
\[
{\bigcap}_{R}^{r}N \stackrel{\sim}{\longrightarrow} {\rm char}_{R}(M). 
\]
Hence this proposition follows from Lemma~\ref{lem:reflexible}.  
\end{proof}

\begin{lemma}\label{lem:inequality}
Let $R$ be a noetherian ring satisfying (G$_{0}$) and (S$_{2}$).  
Let $I$ and $J$ be ideals of $R$. 
If $I_{\fr} \subseteq J_{\fr}$ for any prime $\fr$ of $R$ with ${\rm ht}(\fr) \leq 1$, then we have  
$I^{**} \subseteq J^{**}$.  
\end{lemma}
\begin{proof}
By assumption, the $R$-module $(I+J)/J$ is pseudo-null. Hence $(I+J)^{*} \longrightarrow J^{*}$ is an isomorphism by Lemma~\ref{lem:ext vanish}. This implies $J^{**} = (I+J)^{**} \supseteq I^{**}$.
\end{proof}

\begin{remark}\label{rem:fitt-char}
Let $R$ be a noetherian ring satisfying (G$_{0}$) and (S$_{2}$). 
By Proposition~\ref{prop:reflexible} and Lemma~\ref{lem:inequality}, 
we see that the ideal ${\rm char}_{R}(M)$ is determined by ideals ${\rm char}_{R_{\fr}}(M_{\fr})$ for any primes $\fr$ of $R$ with ${\rm ht}(\fr) \leq 1$. 
Hence one can easily show the following properties of characteristic ideals: 
\begin{itemize}
\item[(i)] If there is a pseudo-isomorphism $M \longrightarrow N$ of finitely generated $R$-modules, then we have 
${\rm char}_{R}(M) = {\rm char}_{R}(N)$. 
\item[(ii)] When $R$ is a normal ring, we have ${\rm Fitt}_{R}^{0}(M)^{**} = {\rm char}_{R}(M)$ for any finitely generated $R$-module, and hence, in this case, the notion of the characteristic ideal coincides with the usual one.  
\end{itemize}
\end{remark}

\begin{proof}[Proof of Theorem~\ref{thm:inequality}]
By Proposition~\ref{prop:reflexible}, the characteristic ideals ${\rm char}_{R}(M)$ and ${\rm char}_{R}(N)$ are reflexive. Hence, by Lemma~\ref{lem:inequality} and the condition (G$_{1}$), 
we may assume that $R$ is a Gorenstein local ring with $\dim(R) \leq 1$. 
 
By the horseshoe lemma, 
a projective resolution of $M$ can be built up inductively with the $n$-th item in the resolution of $M$ equal to the direct sum of the $n$-th items in projective resolutions of $N$ and $M/N$. 
Hence one can take the following commutative diagram with exact lows: 
\[
\xymatrix{
0 \ar[r] & Y \ar[r]^-{\beta} \ar@{^{(}->}[d] & R^{r} \ar[r] \ar@{^{(}->}[d] & N \ar[r] \ar@{^{(}->}[d] & 0 
\\
0 \ar[r] & X \ar[r]^-{\alpha} \ar[d] & R^{r+s} \ar[r] \ar[d] & M \ar[r] \ar[d] & 0
\\
0 \ar[r] & X/Y \ar[r] & R^{s} \ar[r] & M/N \ar[r] & 0. 
}
\]
Here $r$ is a positive integer and the homomorphism $R^{r} \hooklongrightarrow R^{r+s}$ is a split injection. 

Since $R$ is Gorenstein and $\dim(R) \leq 1$, the module 
$\Ext_{R}^{1}(X/Y,R) \stackrel{\sim}{\longrightarrow} \Ext_{R}^{2}(M/N,R)$ vanishes. 
Hence by taking $R$-duals to the above commutative diagram, we get exact sequences of $R$-modules 
\[
\xymatrix{
0 \ar[r] & (R^{*})^{s} \ar[r] \ar[d] & (R^{*})^{r+s} \ar[r] \ar[d]^-{\alpha^{*}} & (R^{*})^{r} \ar[r] \ar[d]^-{\beta^{*}} & 0 
\\
0 \ar[r] & (X/Y)^{*} \ar[d] \ar[r]  & X^{*} \ar[d] \ar[r] & Y^{*}  \ar[d] \ar[r] & 0
\\
& \Ext_{R}^{1}(M/N,R) \ar[r] & \Ext_{R}^{1}(M,R) \ar[r] & \Ext_{R}^{1}(N,R) \ar[r] & 
0. 
}
\]
We fix a basis $\omega \in {\bigwedge}^{s}_{R}(R^{*})^{s}$. 
We also denote by $\omega$ the images of $\omega$ in ${\bigwedge}^{s}_{R}(R^{*})^{r+s}$ and ${\bigwedge}^{s}_{R}X^{*}$. 
Then the homomorphism 
\[
{\bigwedge}^{r}_{R}(R^{*})^{r+s} \longrightarrow {\bigwedge}^{r+s}_{R}(R^{*})^{r+s}; x \mapsto x \wedge \omega
\] 
induces an $R$-isomorphism 
\begin{align*}
{\bigwedge}^{r}_{R}(R^{*})^{r} \stackrel{\sim}{\longrightarrow} {\bigwedge}^{r+s}_{R}(R^{*})^{r+s}. 
\end{align*}
We also have the composite map  
\begin{align}\label{hom1}
{\bigwedge}^{r}_{R}X^{*} \xrightarrow{x \mapsto x \wedge \omega}
{\bigwedge}^{r+s}_{R}X^{*} \longrightarrow
 \left({\bigwedge}^{r+s}_{R}X^{*}\right)^{**}. 
\end{align}
The homomorphism $Q \otimes_{R} (R^{*})^{s} \longrightarrow Q \otimes_{R} (X/Y)^{*}$ is surjective since 
$Q \otimes_{R} \Ext_{R}^{1}(M/N,R)$ vanishes. 
Hence the composite map  
\[
(X/Y)^{*} \otimes_{R} {\bigwedge}^{r-1}_{R}X^{*} \longrightarrow 
{\bigwedge}^{r}_{R}X^{*} \longrightarrow Q \otimes_{R} {\bigwedge}^{r+s}_{R}X^{*}
\]
is zero. 
The canonical homomorphism $ \left({\bigwedge}^{r+s}_{R}X^{*}\right)^{**} \longrightarrow Q \otimes_{R} {\bigwedge}^{r+s}_{R}X^{*}$ is injective, and hence the composite map  
\[
(X/Y)^{*} \otimes_{R} {\bigwedge}^{r-1}_{R}X^{*} \longrightarrow {\bigwedge}^{r}_{R}X^{*} \longrightarrow \left({\bigwedge}^{r+s}_{R}X^{*}\right)^{**} 
\]
is also zero. 
Since the kernel of the surjection ${\bigwedge}^{r}_{R}X^{*} \longrightarrow {\bigwedge}^{r}_{R}Y^{*}$ coincides with the image of the homomorphism  $(X/Y)^{*} \otimes_{R} {\bigwedge}^{r-1}_{R}X^{*} \longrightarrow {\bigwedge}^{r}_{R}X^{*}$, the homomorphism~\eqref{hom1} induces an $R$-homomorphism 
\[
{\bigwedge}^{r}_{R}Y^{*} \longrightarrow \left({\bigwedge}^{r+s}_{R}X^{*}\right)^{**}. 
\]
Hence, we have the following commutative diagram: 
\[
\xymatrix{
{\bigwedge}^{r}_{R}(R^{r})^{*} \ar@/^30pt/[rrr]^-{\cong} \ar[d]_-{\wedge^{r}\beta^{*}} &  \ar@{->>}[l]  {\bigwedge}^{r}_{R}(R^{*})^{r+s} \ar[rr]^-{x \mapsto x \wedge \omega} \ar[d]_{\wedge^{r}\alpha^{*}} &&  {\bigwedge}^{r+s}_{R}(R^{*})^{r+s} \ar[d]_{\wedge^{r+s}\alpha^{*}} \ar[r]^-{\cong} & \left({\bigwedge}^{r+s}_{R}(R^{*})^{r+s}\right)^{**} \ar[d]_{(\wedge^{r+s}\alpha^{*})^{**}}
\\
{\bigwedge}^{r}_{R}Y^{*} \ar@/^-40pt/[rrrr]^{} & \ar@{->>}[l] {\bigwedge}^{r}_{R}X^{*} \ar[rr]^-{x \mapsto x \wedge \omega} && {\bigwedge}^{r+s}_{R}X^{*} \ar[r] & \left({\bigwedge}^{r+s}_{R}X^{*}\right)^{**}. 
}
\]
Therefore, we obtain a commutative diagram 
\[
\xymatrix{
{\bigcap}^{r+s}_{R}X  \ar[d]^-{\cap^{r+s}\alpha} & \ar[l]_-{(\xi^{1})^{*}} \left({\bigwedge}^{r+s}_{R}X^{*}\right)^{***} \ar[r] \ar[d]^-{(\wedge^{r+s}\alpha^{*})^{***}} & {\bigcap}^{r}_{R}Y \ar[d]^{{\cap^{r+s}\beta}} 
\\
{\bigcap}^{r+s}_{R}R^{r+s}  & \ar[l]_-{\cong} \left({\bigwedge}^{r+s}_{R}(R^{r+s})^{*}\right)^{***} \ar[r]^-{\cong}  &  {\bigcap}^{r}_{R}R^{r}. 
}
\]
By Lemma~\ref{lem:reflexible}, the homomorphism $\left({\bigwedge}^{r+s}_{R}X^{*}\right)^{***} \longrightarrow {\bigcap}^{r+s}_{R}X$ is an isomorphism, and hence this commutative diagram shows that 
\begin{align*}
{\rm char}_{R}(M) &= {\rm im}\left({\bigcap}^{r+s}_{R}X \longrightarrow  {\bigcap}^{r+s}_{R}R^{r+s} = R\right)
\\
&= {\rm im}\left(\left({\bigwedge}^{r+s}_{R}X^{*}\right)^{***} \longrightarrow  \left({\bigwedge}^{r+s}_{R}(R^{r+s})^{*}\right)^{***} = R\right)
\\
&\subseteq {\rm im}\left({\bigcap}^{r}_{R}Y  \longrightarrow  {\bigcap}^{r}_{R}R^{r} = R\right)
\\
&=  {\rm char}_{R}(N). 
\end{align*}
\end{proof}

\end{appendix}

\end{document}